\documentclass[10pt,a4paper]{article}

\usepackage{amsmath,amsthm,amssymb,amsfonts,amscd}
\usepackage[utf8]{inputenc}
\usepackage{verbatim}
\usepackage[all]{xy}
\usepackage{latexsym}
\usepackage{epsfig}
\usepackage[hidelinks]{hyperref}

\theoremstyle{plain} \newtheorem{prop}{Proposition} \newtheorem{theo}{Theorem} \newtheorem{teo}{Teorema} \newtheorem*{theo*}{Theorem} \newtheorem*{teo*}{Teorema} \newtheorem*{con}{Conjecture}  \newtheorem*{conj}{Conjetura} \newtheorem{lem}{Lemma} 
\theoremstyle{definition} \newtheorem{defi}{Definition}
\theoremstyle{remark} \newtheorem{rem}{Remark}

\author{Miguel Fern\'{a}ndez Duque}
\title{Local Uniformization of Foliations for Rational Archimedean Valuations}
\makeindex
\begin{document}
\maketitle
\tableofcontents

\section{Introduction}

This work is devoted to show the following statement
\renewcommand{\thetheo}{\Roman{theo}}
\begin{theo}\label{th:I}
	Let $ k $ be a field of characteristic zero and let $ K/k $ be a finitely generated field extension. Let $ \mathcal{F} $ be a rational codimension one foliation of $ K/k $. Given a $ k $-rational archimedean valuation $ \nu $ of $ K/k $, there is a projective model $ M $ of $ K/k $ such that $ \mathcal{F} $ is log-final at the center of $ \nu $ in $ M $.
\end{theo}
In order to prove this theorem we have introduced a non usual tool in the theory of foliations, the truncation, which behavior from the differential and integrable viewpoint is better than the expected.

Theorem \ref{th:I} is a first step in the proof of the following conjecture, whose complete poof is the goal of future works:
\begin{con}
	Let $ k $ be a field of characteristic zero and let $ K/k $ be a finitely generated field extension. Let $ \mathcal{F} $ be a rational codimension one foliation of $ K/k $. Given a valuation $ \nu $ of $ K/k $, there is a projective model $ M $ of $ K/k $ such that $ \mathcal{F} $ is log-final at the center of $ \nu $ in $ M $.
\end{con}

\begin{center}
	 - - -
\end{center}

Consider a finitely generated field extension $ K/k $ with transcendence degree $ \operatorname{tr.deg}(K/k)=n $ over an algebraically closed zero characteristic field $ k $. Let $ \{z_1,z_2,\dots,z_n\} \subset K $ be a transcendence basis of such field extension. We have the following tower of fields
$$ k \subset k(z_1,z_2,\dots,z_n) \subset K \ , $$
where $ K $ is a finitely generated algebraic field extension (thus separable since $ \operatorname{char}(k)=0 $) of $ k(z_1,z_2,\dots,z_n) $. The module of Kähler differentials $ \Omega_{K/k} $ is a $ K $-vector space of dimension $ n = \operatorname{tr.deg}(K/k) $ and $ \{dz_1,dz_2,\dots,dz_n\} $ is a basis of such space. A \emph{rational codimension one foliation $ \mathcal{F} $ of $ K/k $} is a $ K $-vector subspace of dimension one of $ \Omega_{K/k} $ such that for every $ 1 $-form $ \omega \in \mathcal{F} $ the integrability condition
$$ \omega \wedge d \omega = 0  $$
is satisfied.

This definition of foliation agrees with the classical definition in complex projective geometry. Consider the complex projective space $ \mathbb{P}_{\mathbb{C}}^{n} $ and a affine cover $ \mathbb{P}_{\mathbb{C}}^{n} = U_0 \cup U_1 \cup \cdots \cup U_n $. A codimension one foliation of $ \mathbb{P}_{\mathbb{C}}^{n} $ is given by $ n+1 $ integrable homogeneous polynomial $ 1 $-forms 
$$
W_i=\sum_{j=1}^n P^i_j(z^i_1,z^i_2,\ldots,z^i_n) \, dz^i_j\ , \quad i=0,1,\dots,n \ ,
$$
defined over the affine charts $ U_i \simeq \mathbb{C}[z^i_1,z^i_2,\ldots,z^i_n] $, in such a way that
$$
W_{i\vert_{U_i\cap U_j}}= G_{ij} \, W_{\ell\vert_{U_i\cap U_j}} \ ,
$$
where $ G_{ij} $ is an invertible rational function $ U_i \cap U_j $. The field of rational functions of any affine chart $ U_i $ and $ \mathbb{P}_{\mathbb{C}}^{n} $ itself, is
$$ K \simeq \mathbb{C}(z^i_1,z^i_2,\ldots,z^i_n) $$
for any index $ i $. All the $ 1 $-forms $ W_i $ can be considered as elements of $ \Omega_{K/\mathbb{C}} $. Any of them spans the same vector subspace of dimension $ 1 $
$$ \mathcal{F} = <W_i> \subset \Omega_{K/\mathbb{C}} \ , $$
which is a codimension one rational foliation of $ K/\mathbb{C} $ following ou definition.

\begin{center}
	 - - -
\end{center}

A projective model of $ K/k $ is a projective $ k $-variety $ M $, in the sense of scheme theory, such that $ K = \kappa(M) $ is its field of rational functions. Take a regular $ k $-rational point $ Y \in M $, it means, a point such that the local ring $ \mathcal{O}_{M,Y} $ is regular and its residue field is $ \kappa_{M,Y} \simeq k $. A system of generators $z_1,z_2,\ldots,z_n$ of the maximal ideal ${\mathfrak m}_{M,Y}$ is also a transcendence basis of $ K/k $, thus it provides a basis $ dz_1,dz_2,\ldots,dz_n $ of $ \Omega_{K/k} $. Consider a system of generators of ${\mathfrak m}_{M,Y}$ of the form $ \boldsymbol{z}=(\boldsymbol{x},\boldsymbol{y}) $, where $ \boldsymbol{x}=(x_1,x_2,\cdots,x_r) $ and $ \boldsymbol{y}=(y_1,y_2,\cdots,y_{n-r}) $. Let $\Omega_{{\mathcal O}_{M,Y}/k}(\log \boldsymbol{x})$ be the ${\mathcal O}_{M,Y}$-submodule of $\Omega_{K/k}$ generated by $\Omega_{{\mathcal O}_{M,Y}/k}$ and the logarithmic differentials
$$
\frac{dx_1}{x_1},\frac{dx_2}{x_2},\dots,\frac{dx_r}{x_r} \ .
$$  
We have that $\Omega_{{\mathcal O}_{M,Y}/k}(\log \boldsymbol{x})$ is a free ${\mathcal O}_{M,Y}$-module of rank $ n $ generated by 
$$
\frac{dx_1}{x_1},\frac{dx_2}{x_2},\dots,\frac{dx_r}{x_r}, dy_1,dy_2,\ldots,dy_{n-r} \ .
$$ 
Let $\mathcal F$ be a codimension one rational foliation of $K/k$. Consider
$$
{\mathcal F}_{M,Y}(\log \boldsymbol{x})= {\mathcal F}\cap \Omega_{{\mathcal O}_{M,Y}/k}(\log \boldsymbol{x}) \ . 
$$
$ {\mathcal F}_{M,Y}(\log \boldsymbol{x}) $ is a free ${\mathcal O}_{M,Y}$-module of rank $ 1 $ generated by an integrable $ 1 $-form
$$
\omega=\sum_{i=1}^r a_i\frac{dx_i}{x_i}+\sum_{j=1}^{n-r}b_jdy_j \ ,
$$
where the coefficients $a_1,a_2,\dots,a_r,b_1,b_2,\dots,b_{n-r}\in {\mathcal O}_{M,Y}$ have no common factor. We say that
\begin{enumerate}
	\item $\mathcal F$ is \emph{$ \boldsymbol{x} $-log elementary} at $Y\in M$ if $ (a_1,a_2,\dots,a_r) = {\mathcal O}_{M,Y} $;
	\item $\mathcal F$ is \emph{$ \boldsymbol{x} $-log canonical} at $Y\in M$ if $ (a_1,a_2,\dots,a_r)\subset {\mathfrak m}_{M,Y}$ and in addition
	$$ (a_1,a_2,\dots,a_r) \not \subset (x_1,x_2,\dots,x_r)+ \mathfrak{m}_{M,Y}^2 \ . $$
\end{enumerate}
We say that $\mathcal F$ is \emph{$ x $-log final} at $Y\in M$ if it is $ \boldsymbol{x} $-log elementary or $ \boldsymbol{x} $-log canonical. Finally, we say that $\mathcal F$ is \emph{log-final} at $Y\in M$ if it is $ \boldsymbol{x} $-log final for certain system of generators $(\boldsymbol{x},\boldsymbol{y})$ of $ \mathfrak{m}_{M,Y} $. 

To be log-final is the algebraic version of the concept of pre-simple singularity of the complex analytic case (\cite{Sei},\cite{CaCe},\cite{Ca04}). Let us briefly recall this definition. Consider a foliation of $ (\mathbb{C}^2,\boldsymbol{0}) $ given locally by
$$ a(x,y)dx+b(x,y)dy=0 \ . $$  
The origin $ (0,0) $ is a pre-simple singularity if the foliation is non singular (one of the series $a(x,y)$ or $b(x,y)$ is a unit) or if the Jacobian matrix
$$
\left(\begin{array}{cc}
{\partial b}/{\partial x}(0,0) & {-\partial a}/{\partial x}(0,0) \\
{\partial b}/{\partial y}(0,0) & {-\partial a}/{\partial y}(0,0)
\end{array}\right)
$$
is non-nilpotent. In this situation we can always take local analytic coordinates $x',y'$ such that the foliation is locally given by
$$
 a'(x',y')\frac{dx'}{x'}+ b'(x',y')dy'=0 \ ,
$$
where $a'(x',y')=y'+\cdots$, thus the foliation is $ x' $-log-final at the origin, with respect to the local analytic coordinates $ (x',y') $. A more detailed study can be found in \cite{CaCeDe}. In general, for foliations over complex ambient spaces of arbitrary dimension, the concept of pre-simple singularity introduced in \cite{CaCe} and \cite{Ca04} is equivalent to the property of being log-final.

In the case of foliations over algebraic varieties of dimension $ 2 $ or $ 3 $, the theorem proved in this work, as well as the conjecture, are consequences of the following global results of reduction of singularities (see \cite{Sei} for the two-dimensional case and \cite{Ca04} for the dimension $ 3 $):
\begin{theo*}[A. Seidenberg, 1968; F. Cano 2004]
Let $\mathcal F$ be a codimension one rational foliation of $({\mathbb C}^n,0)$, $n=2,3$. There is a finite composition of blow-ups
$$ (\mathbb{C}^{n},0)  \leftarrow  (M_1,Z_1) \leftarrow \cdots \leftarrow (M_N,Z_N) = (M,Z) $$
such that $ \mathcal{F} $ is log-final at every point $ Y \in Z $.
\end{theo*}
In the case of ambient spaces o dimension $ n\geq 4 $ the global reduction of singularities of codimension one foliations is an open problem.

\begin{center}
 - - -
\end{center}

Resolution of singularities of algebraic varieties over a ground field of characteristic $ 0 $ was achieved by Hironaka in its celebrated paper \cite{Hi1}.
\begin{theo*}[Hironaka's Reduction of Singularities, 1964]
	Let $K/k$ be a finitely generated field extension, where $k$ has characteristic zero. There is a non-singular projective model $M$ of $K/k$.
\end{theo*}
Before the work of Hironaka, the problem was solved for dimension at most three. The case of complex curves was already treated by Newton in 1676. For surfaces, the first rigorous proof is due to Walker in 1935 \cite{Wal}. The case of $ 3 $-folds was solved by Zariski in 1944 \cite{Zar44Ann}. Before this result, Zariski proved local uniformization for algebraic varieties in characteristic zero in arbitrary dimension \cite{Zar40}.
\begin{theo*}[Zariski's Local Uniformization, 1940]
	Let $K/k$ be a finitely generated field extension, where $k$ has characteristic zero and consider a valuation $ \nu $ of $ K/k $. There is a projective model $M$ of $K/k$ such that the center of $ \nu $ in $M$ is a regular point. 
\end{theo*}
In \cite{Zar44Bul}, Zariski showed the compactness of the Zariski-Riemann space (the space of valuations of $K/k$), which implies that a finite number of projective models are enough to support local uniformizations for any valuation. Then, he obtained his global result by patching projective models, but this method only works in dimension at most $ 3 $.
 
The general problems of resolution of singularities has a long history after the works of Zariski and Hironaka. The case of complex analytic spaces was achieved by Aroca, Hironaka and Vicente \cite{AroHiVi}. The huge original proofs of Hironaka have been analyzed very carefully with emphasis in the constructiveness and functorial properties by Villamayor \cite{Vil}, Bierstone and Milman \cite{BierMil} and others.
 
One of the keys of the resolution of singularities is the maximal contact theory and its differential version, developed by Giraud \cite{Gir}. This is one of the starting points of the strongest known results in positive characteristic, due to Cossart and Piltant, who proved resolution of singularities of $ 3 $-folds in positive characteristic \cite{CoPil08,CoPil09}. This work improves the results of Abhyankar, which shows resolution in positive characteristic for surfaces \cite{Abh56}, and for $ 3 $-folds in the case of fields of characteristic at least $ 7 $ \cite{Abh66}. All of these results in positive characteristic are obtained passing through local uniformization.

Another related problems are concerning the monomialization of morphisms in zero characteristic due to Cutkosky \cite{Cut}. The difficulties in this case are close to the ones for case of vector fields or foliations. 
 
Reduction of singularities of vector fields in dimension two was achieved by Seidenberg \cite{Sei}. In dimension $ 3 $, there is partial results due to Cano \cite{Ca87}, an then this author, Roche and Spivakovsky obtain a global reduction of singularities through local uniformization \cite{CaRoSp}, using the axiomatic Zariski's patching method developed by Piltant \cite{Pil}. Recently McQuillan and Panazzolo have treated the $ 3 $-dimensional case from a non-birational viewpoint \cite{McQPa}.

\begin{center}
 - - -
\end{center}

In this work we treat the case of $ k $-rational rank one valuations of $ K/k $. In the classical Zariski's approach to the Local Uniformization problem, this is the starting point,  and it concentrates the main algorithmic and combinatoric difficulties. We hope the situation to be similar in the case of codimension one foliations, and that starting from the result obtained in this work we can complete the proof of the conjecture in future works.

A valuation $ \nu : K^* \rightarrow \Gamma $ of $ K/k $ is $ k $-rational if its residue field $ \kappa_{\nu} $ is isomorphic to the base field $ k $. The valuation $ \nu $ is archimedean if and only if there is an inclusion of ordered groups $ \Gamma \subset (\mathbb{R},+) $. 

Let $M$ be a projective model of $K/k$. The center of $\nu$ at $M$ is the unique point $Y\in M$ such that for every $ \phi \in \mathcal{O}_{M,Y} $ we have
$$ \nu(\phi) \geq 0 \quad \text{and} \quad \nu(\phi)>0 \Leftrightarrow \phi \in \mathfrak{m}_{M,Y} \ .  $$

Such a point always exists and it is unique (see \cite{Zar} or \cite{Va}). In addition, there is a tower of fields
$$ k \subset \kappa_{M,Y} \subset \kappa_{\nu} \ . $$
Since we only consider $ k $-rational valuations we have that $ k =\kappa_{M,Y} $ and therefore the centers of $ \nu $ in each projective model are $k$-rational points (in particular they are closed points). 

The \emph{rational rank} $ \operatorname{rat.rk}(\nu) $ is the dimension over $ \mathbb{Q} $ of $ \Gamma \otimes_{\mathbb{Z}} \mathbb{Q} $. Abhyankar's inequality guarantees that $ \operatorname{rat.rk}(\nu) \leq \operatorname{tr.deg}(K/k) $. The rational ranks correspond with the maximum number of elements $ \phi_1,\phi_2,\dots,\phi_r \in K^* $ with $ \mathbb{Z} $-independents values $ \nu(\phi_1), \nu(\phi_2), \dots, \nu(\phi_r) \in \Gamma $.

Our technical results are stated in terms of {\em parameterized regular local models}. A parameterized regular local model ${\mathcal A}$ for $K/k , \nu $ is a pair
$$
{\mathcal A}=({\mathcal O}, (\boldsymbol{x},\boldsymbol{y}))
$$
such that
\begin{enumerate}
	\item There is a projective model $M$ of $K/k$ such that the center $Y$ of $ \nu $ is a regular point of $M$ and ${\mathcal O}={\mathcal O}_{M,Y}$;
	\item The list $(\boldsymbol{x},\boldsymbol{y})=(x_1,x_2,\ldots,x_r,y_1,y_2,\dots,y_{n-r})$, where $ r = \operatorname{rat.rk}(\nu) $, is a regular system of parameters of ${\mathcal O}$ and the values $ \nu(x_1),\nu(x_2),\ldots,\nu(x_r) $ are ${\mathbb Z}$-independent. 
\end{enumerate}
The existence of parameterized regular local models is proved using the global resolution of singularities of Hironaka \cite{Hi1}. This proof can be found in \cite{CaRoSp}, where such models are introduced.

According with this terminology, given a rational codimension one foliation $ \mathcal{F} $ of $ K/k $, we denote
$$
{\mathcal F}_{\mathcal A}={\mathcal F}\cap \Omega_{{\mathcal O}/k}(\log \boldsymbol{x})={\mathcal F}_{M,Y}(\log \boldsymbol{x}) \ .
$$
We say that $\mathcal F$ is \emph{$\mathcal A$-final} if $ \mathcal{F}_{\mathcal{A}} $ is $\boldsymbol{x}$-log final. 

We will use transformations ${\mathcal A}\rightarrow {\mathcal A}'$ of parameterized regular local models, called {\em basic operations}, that have an underlying morphism $ \mathcal{O} \rightarrow \mathcal{O}' $ which can be either a blow-up or the identity morphism. Let us describe them:
\begin{itemize}
	\item {\em Ordered coordinate changes}. The underlying morphism $ \mathcal{O} \rightarrow \mathcal{O}' $ is the identity. Given an index $0\leq \ell\leq n-r$ we consider a new coordinate $ y'_{\ell} $ given by
	$$
	y_\ell'=y_{\ell}+\psi(\boldsymbol{x},y_1,y_2,\ldots,y_{\ell-1}),
	$$ 
	where $\psi(\boldsymbol{x},y_1,y_2,\ldots,y_{\ell-1})\in k[\boldsymbol{x},y_1,y_2,\ldots,y_{\ell-1}]$ is written as
	$$
	\psi(\boldsymbol{x},y_1,y_2,\ldots,y_{\ell-1})=\sum_I \boldsymbol{x}^I \psi_I(y_1,y_2,\ldots,y_{\ell-1})
	$$
	with $\nu(\boldsymbol{x}^I) \geq \nu(y_\ell)$ if $ \psi_I \neq 0 $. 
	\item {\em Blow-ups con centros of codimensión dos}. \item {\em Blow-ups of codimension two centers}. The center of the blow-up will be either $x_i=x_j=0$ or $x_i=y_j=0$. The ring ${\mathcal O}'$ is 
	$$
	{\mathcal O}'={\mathcal O}[\boldsymbol{x}',\boldsymbol{y}']_{(\boldsymbol{x}',\boldsymbol{y}')}
	$$
	where the coordinates $(\boldsymbol{x}',\boldsymbol{y}')$ are given by:
	\begin{enumerate}
		\item If the centeris $x_i=x_j=0$ and in addition $\nu(x_i)<\nu(x_j)$, then $x'_j :=x_j/x_i$.
		\item If the centeris $x_i=y_j=0$ and in addition $\nu(x_i)<\nu(y_j)$, then $y'_j :=y_j/x_i$.
		\item If the centeris $x_i=y_j=0$ and in addition $\nu(x_i)>\nu(y_j)$, then $x'_i :=x_i/y_j$.
		\item If the centeris $x_i=y_j=0$ and in addition $\nu(x_i)=\nu(y_j)$, then $y'_j := y_j/x_i-\xi$, where $\xi\in k^* $ is the unique constant such that $\nu(y_j/x_i-\xi)>0$.
	\end{enumerate}
	The last case is a \emph{blow-up with translation}. The remaining cases are \emph{combinatorial blow-up}.
\end{itemize}
Theorem \ref{th:I} is a consequence of the following statement in terms of parameterized regular local models:
\begin{theo}\label{th:II}
	Let $ k $ be a field of characteristic zero and let $ K/k $ be a finitely generated field extension. Let $ \mathcal{F} $ be a rational codimension one foliation of $ K/k $. Given a $ k $-rational archimedean valuation $ \nu $ of $ K/k $ and a parameterized regular local model $ \mathcal{A} $ of $K/k, \nu$, there is a finite composition of basic transformations
	$$ {\mathcal A}={\mathcal A}_0\rightarrow {\mathcal A}_1\rightarrow \cdots\rightarrow {\mathcal A}_N={\mathcal B} $$
	such that ${\mathcal F}$ is ${\mathcal B}$-final.  
\end{theo}

\begin{center}
	- - -
\end{center}

We systematically consider the formal completion $\widehat{\mathcal O}$ of the local ring $\mathcal O$. A first reason for doing that is of practical nature, since
$$
\widehat{\mathcal O}=k[[x,y]],
$$
we can consider the elements of ${\mathcal O}$ as being formal series. The second reason to consider the formal completion is due to the fact that the solutions of differential equations with coefficients in $ \mathcal{O} $ need not to be in the same ring $\mathcal O$ (even in the case that we are working in the analytic category). 

Let us illustrate this with an example. If our ``problem object'' is a function $f\in {\mathcal O}$, after finitely many basic transformations we obtain a parameterized regular local model $ \mathcal{A}' = (\mathcal{O}',(\boldsymbol{x}',\boldsymbol{y}')) $ such that 
$$ f=\boldsymbol{x}'^{\boldsymbol{p}} U \ , \quad U \in \mathcal{O}' \setminus \mathfrak{m}' \ . $$
This is a direct consequence of Zariski's Local Uniformization. For completeness we include a proof for the case of functions which we will use as a guide for the case of differential $ 1 $-forms. If we consider the foliation given by $ df=0 $ we have that it is $ \mathcal{A}' $-final, in particular it is $ \boldsymbol{x} $-log elementary. This property is always satisfied by foliations having a first integral: it is always possible to reach a model in which the foliation is $ \boldsymbol{x} $-log elementary. However, this does not happen in general. In dimension two we have an example given by Euler's Equation:
$$
(y-x)\frac{dx}{x}-xdy=0 \ .
$$
The foliation of $ (\mathbb{C}^2,\boldsymbol{0}) $ given by this equation is $x$-log canonical. In addition, it has an invariant formal curve with equation $\hat y=0$ where
$$
\hat{y}=y-\sum_{n=0}^{\infty} n!\, x^{n+1} \ .
$$
Note that if we consider formal coordinates $ (x,\hat{y}) $ the foliation is given by
$$
\frac{dx}{x}+\frac{d\hat y}{\hat y}=0 \ ,
$$ 
thus it would be ``$x\hat y$-log elementary''. If we consider the valuation of $ \mathbb{C}(x,y)/\mathbb{C} $ given by
$$
\nu(f(x,y))=\operatorname{ord}_t(f(t,\sum_{n=0}^{\infty} n! t^{n+1}))
$$
we can check that it is not possible to reach by means of basic transformations a parameterized regular local model in which the foliation is $ x $-log elementary. Regardless of the basic transformations we perform the foliation will be $ x $-log canonical. In fact, $ \nu $ is the only valuation of $ \mathbb{C}(x,y)/\mathbb{C} $ which satisfies this property. This is due to the fact that $ \nu $ follows the infinitely near points of the invariant formal curve $ \hat{y}=0 $.

In some sense, we can think that the differential $ 1 $-form $ \omega = x dy-(y-x){dx}/{x}$ has ``infinite value'' with respect to $\nu$. The property of `having infinite value'' can be also satisfied by formal functions $\hat{f} \in \hat{\mathcal O}\setminus {\mathcal O}$ (in this example $ \hat{y} \in \hat{\mathcal{O}} $ has ``infinite value''), but it can never be satisfied by elements $ f\in {\mathcal O}\setminus \{0\}$. Note that although $ \omega $ has ``infinite value'' its coefficients do not have it.

\begin{center}
	- - -
\end{center}

One of the main differences of our procedure with respect to the classical Zariski's approach to Local Uniformization is that we proceed by systematically considering truncations relative to a given element of the value group. This method is essential for us since thanks to it we can control control the integrability condition inside the general induction procedure. In addition, this method can be applied to formal functions as well as differential $ 1 $-forms with formal coefficients.
 
Let $ \mathcal{A} = \big( \mathcal{O} ; (\boldsymbol{x},\boldsymbol{y}) \big) $ be a parameterized regular local model for $ K,\nu $, where $ \boldsymbol{x}=(x_1,x_2,\dots,x_r) $ and $ \boldsymbol{y}=(y_1,y_2,\dots,y_{n-r}) $. For each index $ 0 \leq \ell \leq n-r $, consider the power series ring 
$$
R_{\mathcal A}^{\ell} := k[[\boldsymbol{x},y_1,y_2,\ldots,y_{\ell}]] \ .
$$ 
Note that $ R_{\mathcal A}^{n-r} \simeq \hat{\mathcal O} $. For each index $ 0 \leq \ell \leq n-r $ we define $ N_{\mathcal A}^{\ell} $ as the $R_{\mathcal A}^{\ell}$-module generated by
$$
\frac{dx_1}{x_1}, \frac{dx_2}{x_2}, \dots,\frac{dx_r}{x_r}, dy_1, dy_2, \dots, dy_{\ell} \ .
$$
Any element $ \omega\in N_{\mathcal A}^{\ell} $ may be written in a unique way as $\omega=\sum_Ix^I\omega_I$, where
$$
\omega_I= \sum_{i=1}^ra_{I,i}(y_1,y_2,\ldots,y_\ell)\frac{dx_i}{x_i}+
\sum_{j=1}^\ell b_{I,j}(y_1,y_2,\ldots,y_\ell){dy_j} \ .
$$
We define the {\em explicit value} $\nu_{\mathcal A}(\omega)$ to be the minimum among the values $\nu(x^I)$ such that $\omega_I\ne 0$. 

Let us fix an element of the value group $\gamma\in \Gamma$. Take  $\omega\in N_{\mathcal A}^{\ell}$ and let $x^{I_0}$ be the monomial such that $\nu_{\mathcal A}(\omega)=\nu(x^{I_0})$. We say that $\omega$ is {\em $\gamma$-final in $\mathcal A$} if one of the following situations holds
\begin{itemize}
	\item {\em $\gamma$-final dominant case:}  $\nu_{\mathcal A}(\omega)\leq \gamma$ and at least one of the coefficients $a_{I_0,i}$ satisfies $a_{I_0,i}(0)\ne 0$ (equivalently if $\omega_{I_0}$ is $x$-log elementary). 
	\item {\em $\gamma$-final recessive case:} $\nu_{\mathcal A}(\omega)=\nu(x^{I_0})>\gamma$.
\end{itemize}
Theorem \ref{th:II} (thus Theorem \ref{th:I} too) is a consequence of the following result stated in terms of truncations:
\begin{theo}[Truncated Local Uniformization]\label{th:III}
	Let $\omega\in N_{\mathcal A}^{\ell}$ be a $ 1 $-form and fix a value $\gamma\in \Gamma$. If
	$$ \nu_{\mathcal A}(\omega\wedge d\omega)\geq 2\gamma \ , $$
	then there is a finite composition of basic transformations ${\mathcal A}\rightarrow {\mathcal B}$, which do not affect to the variables $y_j$ with $j>\ell$, such that $\omega$ is $\gamma$-final in $\mathcal B$.
\end{theo}
The main and more technical part of this work, Chapters \ref{ch:truncated_preparation} and \ref{ch:getting_final_forms}, is devoted to prove Theorem \ref{th:III}. In Chapter \ref{ch:final} we show how to derive Theorem \ref{th:I} from Theorem \ref{th:III}.

\begin{center}
	- - -
\end{center}

We prove Theorem \ref{th:III} by induction on the number of dependent variables $ \ell $. Instead of perform arbitrary compositions of basic transformations, we will restrict ourselves to the \emph{$ \ell $-nested transformations} which we define by induction. A \emph{$ 0 $-nested transformation} is a finite composition of blow-ups of the kind $ x_i=x_j=0 $ (in particular all oh them are combinatorial). A \emph{$ \ell $-nested transformation} is a finite composition of $ (\ell-1) $-nested transformations, \emph{$ \ell $-Puiseux's packages} and ordered changes of the $ \ell $-th variable.

Given a dependent variable $ y_{\ell} $, its value linearly depends on the value of the independent variables $ \boldsymbol{x} $. It means that there are integers $ d \geq 1 $ and $ p_1,\dots,p_r $ such that
$$
d \nu(y_{\ell}) = p_1 \nu(x_1) + \cdots + p_r \nu(x_r) \ ,
$$
or equivalently, the \emph{contact rational function $ y_{\ell}^d / \boldsymbol{x^p} $} has zero value. Since the valuation is $ k $-rational, there is a unique constant $ \xi \in k^* $ such that $ y_{\ell}^d / x^p - \xi $ has positive value. A \emph{$ \ell $-Puiseux's package} ${\mathcal A}\rightarrow {\mathcal A}'$ is any finite composition of blow-ups of parameterized regular local models with centers of the kind $ x_i=x_j=0 $ or $ x_i = y_{\ell} = 0 $ such that all the blow-ups are combinatorial except the last one. In particular, we have that $ y_{\ell}'=y_{\ell}^d / x^p - \xi $. The existence of Puiseux's packages follows form the resolution of singularities of the binomial ideal $ (y_{\ell}^d  - \xi x^p)$.

The Puiseux's packages were introduced in \cite{CaRoSp} for the treatment of vector fields. In the two-dimensional case, the Puiseux's packages are directly related with the Puiseux's pairs of the analytic branches that the valuation follows.

The statements $T_3(\ell)$, $T_4(\ell)$ and $T_5(\ell)$ formulated below are proved by induction on $ \ell $. The statement $T_3(\ell)$ is about $1$-forms with formal coefficients. In particular, Theorem \ref{th:III} is equivalent to $T_3(n-r)$. The statement $T_4(\ell)$ is about formal functions, while $T_5(\ell)$ deals with pairs function-form - in fact this result can be state and prove in a more general way for finite lists os functions and $ 1 $-forms without adding difficulties, but is this precise formulation which we will use inside the induction procedure -.

The concept of explicit value previously defined for $ 1 $-forms extends directly to the case of functions, pairs function-form and $ p $-forms for any $ p\geq 2 $. In the same way, given a value $ \gamma \in \Gamma $ we extend the definition of $ \gamma $-final $ 1 $-forms to the case of functions and pairs function-form.

Let $ F \in R_{\mathcal{A}}^{\ell} $ be a formal function and let $ \omega \in N_{\mathcal{A}}^{\ell} $ be a $ 1 $-form. We state the following results: 
\begin{description}
	\item[$\boldsymbol{T_3(\ell)} $ :]Assume that $\nu_{\mathcal A}(\omega\wedge d\omega)\geq 2\gamma$. There is a $ \ell $-nested transformation $ \mathcal{A} \rightarrow \mathcal{B} $ such that $\omega $ is $ \gamma $-final in $ \mathcal{B} $.
	\item[$\boldsymbol{T_4(\ell)} $ :] There is a $ \ell $-nested transformation $ \mathcal{A} \rightarrow \mathcal{B} $ such that $ F $ is $ \gamma $-final in $ \mathcal{B} $.
	\item[$\boldsymbol{T_5(\ell)} $ :] Assume that $\nu_{\mathcal A}(\omega\wedge d\omega)\geq 2\gamma$. There is a $ \ell $-nested transformation $ \mathcal{A} \rightarrow \mathcal{B} $ such that the pair $(\omega, F) $ is $ \gamma $-final in $ \mathcal{B} $.
	
\end{description}
The statements ${T_3(0)} $, ${T_4(0)} $ and ${T_5(0)} $ can be proved easily by means of the control of the Newton Polyhedron of an ideal under combinatorial blow-ups \cite{Spi83}. 

As induction hypothesis we assume that statements ${T_3(k)} $, ${T_4(k)} $ and ${T_5(k)} $ are true for all $k\leq \ell$. We will see that ${T_3(\ell+1)} $ implies ${T_4(\ell+1)} $ and ${T_5(\ell+1)} $, although in Chapter \ref{ch:function_case} we include the proof of ${T_4(\ell+1)} $ since we will use this proof as a guide for the case of differential $ 1 $-forms. The more difficult step is to prove ${T_3(\ell+1)} $ making use of the hypothesis induction.

\begin{center}
	- - -
\end{center}
Assuming the hypothesis induction, we divide the poof of ${T_3(\ell+1)} $ in two steps:
\begin{enumerate}
	\item Process of \emph{$ \gamma $-preparation} of a $ 1 $-form $ \omega \in N_{\mathcal{A}}^{\ell+1} $ by means of $ \ell $-nested transformations (Chapter \ref{ch:truncated_preparation}).
	\item Getting $ \gamma $-final $ 1 $-forms. We define invariants related to a $ \gamma $-prepared $ 1 $-form $ \omega \in N_{\mathcal{A}}^{\ell+1} $ and controlling them we determine a $ (\ell+1) $-nested transformation in such a way that $ \omega $ is $ \gamma $-final in the parameterized regular local model obtained (Chapter \ref{ch:getting_final_forms}).
\end{enumerate}
We end this introduction with a brief description of these steps.
\begin{center}
	- - -
\end{center}

As an example, we show how to obtain Local Uniformization of a function using the method of truncations. The difficulties encountered in the process of $ \gamma $-preparation of a $ 1 $-form do not appear in the case of functions. The main difference lies in the nature of the objects to which we apply the induction hypothesis. Given a formal function $ F \in R_{\mathcal{A}}^{\ell + 1} $ we can write it as a power series in the last dependent variable
$$ F = \sum_{s\geq 0} y_{\ell+1}^s F_s(\boldsymbol{x},y_1,y_2,\dots,y_{\ell}) \ . $$
The coefficients $ F_s $ belong to $ R_{\mathcal{A}}^{\ell} $ thus we can apply $ T_4(\ell) $ with them. However, given a $ 1 $-form $ \omega \in N_{\mathcal{A}}^{\ell+1} $, if we ``decompose'' it in the same way
$$ \omega = \sum_{s\geq 0} y_{\ell+1}^s \omega_s(\boldsymbol{x},y_1,y_2,\dots,y_{\ell}) \ , $$
the ``coefficients'' $ \omega_s $ do not belong to $ N_{\mathcal{A}}^{\ell} $, so we can not use $ T_3(\ell) $ directly.

Since we proceed by induction on the parameter $ \ell $ we denote $ z = y_{\ell+1} $ and $ \boldsymbol{y}=(y_1,y_2,\dots,y_{\ell}) $. As before, given $ F \in R_{\mathcal{A}}^{\ell + 1} $ we consider the decomposition $ F=\sum z^s F_s $. We define the \emph{Truncated Newton Polygon} $ \mathcal{N}(F;\mathcal{A},\gamma) $ as the positive convex hull in $ \mathbb{R}^2 $ of the points $(0,\gamma/\nu(z))$, $(\gamma,0)$ and $(\nu_{\mathcal A}(F_s),s)$ for $ s\geq 0 $. In the same way we define the\emph{Truncated Dominant Newton Polygon} $ \mathrm{Dom}\mathcal{N} $, this time considering $(0,\gamma/\nu(z))$, $(\gamma,0)$ and only the points $(\nu_{\mathcal A}(F_s),s)$ such that $ F_s $ is $(\gamma-s\nu(z))$-final dominant. The function $ F $ is $ \gamma $-prepared in $ \mathcal{A} $ if 
$$
{\mathcal N}= {\mathrm Dom}{\mathcal N}
$$
and in addition $(\nu_{\mathcal A}(F_s),s)$ is an interior point of $\mathcal N$ for any non dominant level $F_s$.

The $ \gamma $-preparation of a function has no major difficulties. First, since a $ \ell $-nested transformation does not mix the levels $ F_s $ between them, we can perform a $ \ell $-nested transformation in such a way that we obtain the maximum number of $ (\gamma - s\nu(z)) $-final levels $ F_s $ (this operation does not use of the induction hypothesis). Then, the induction hypothesis allow us to reach a $ \gamma $-prepared situation directly: is enough to apply $ T_4(\ell) $ to the levels $ F_s $ with $ s \leq \gamma/\nu(z) $ which are non-dominant. 

Once we have a $ \gamma $-prepared situation we take the \emph{critical height} $ \chi $, which corresponds with the highest point of the side of $\mathcal N$ with slope $-1/\nu(z)$ (see Section \ref{se:critical_height_function}).

Now we must to study the behavior of $ \chi $ after performing a $ z $-Puiseux's package with contact rational function $z^d/\boldsymbol{x^p} $, followed by a new $ \gamma $-preparation.

If the \emph{ramification index} $ d $ is strictly greater than $ 1 $ we obtain $ \chi' < \chi $. If at one moment we have that the critical height is $ 0 $, after one more $ z $-Puiseux's package and a $ \gamma $-preparation we reach a $ \gamma $-final situation.

Therefore, it only remains to consider the case $ d=1 $ and $ \chi > 0 $ indefinitely. In this situation, each $ z $-Puiseux's package increases the explicit value of $ F $. If a one moment such value becomes greater than $ \gamma $ we have reach a $ \gamma $-final recessive situation. However, it may happen that the value ``accumulates'' before reaching $ \gamma $. This phenomenon is due to the possibility of having $ d \geq 2 $ before the $ \gamma $-preparation, but $ d= 1 $ after performing it. In this situation we use a partial Tschirhausen (a certain kind of ordered change of coordinates). In the case of differential $ 1 $-forms we also use this kind of change of variables, but its determination is more subtle.

Let us illustrate this situation with an example. Define recursively the following elements of $ \mathbb{C}(x,y,z) $ for $ j \geq 0 $:
$$
f_{j+1}=\frac{f_j}{g_{j}} \ , \quad  g_{j+1}= \frac{g_j^2}{f_j} - 1 \ , \quad h_{j+1}= \frac{g_j h_j}{f_j} - 1 \ ,
$$
where
$$
f_0 =x \ , \quad g_0 = y \ , \quad h_0 = z \ .
$$
Let $ \nu $ be a valuation of $ \mathbb{C}(x,y,z)/\mathbb{C} $ such that
$$
\nu(x) = 1 \ , \quad \nu(y) = \nu(z) = \frac{1}{2} \ , \quad \nu(y - z) = 1 \ ,
$$
and in addition there is a transcendental series $\phi=\sum_{k=1}^\infty c_kx^k$ such that
$$
\nu \left(y-z- \sum_{k=1}^t c_k x^k \right) = t+1 \quad \text{for all } t \geq 1 \ .
$$
We have that $ \mathcal{A}=(\mathcal{O},(x,y,z)) $ where $ \mathcal{O} = \mathbb{C}[z,y,z]_{(x,y,z)} $ is a parameterized regular local model of $ \mathbb{C}(x,y,z)/\mathbb{C},\nu $. Using the recurrence relations we conclude that for any $ j \geq 0 $ we have
$$
\nu(f_j) = 2^{-j} \ , \quad \nu(g_j) = \nu(h_j) = 2^{-j-1} \quad \text{and} \quad \nu(g_j - h_j) = 2^{-j} \ .
$$
Consider as the problem object the formal function $ F \in \mathbb{C}[[x,y,z]] $ given by
$$
F = z - y - \phi
$$
and fix a value $ \gamma \geq 1 $. The function $ F $ is not $ \gamma $-prepared (note that the $ 0 $-level $ F_0 = -y -\phi  $ has explicit value $ 0 < \gamma $ and it is not $ 0 $-final dominant). In order to $ \gamma $-prepare $ F $ is enough to perform a $ y $-Puiseux's package $ \mathcal{A} \rightarrow \mathcal{A}'_1 $, formed by two blow-ups, whose equations are
$$
x = x_1^2(y_1+1) \ , \quad y = x_1 (y_1+1) \ ,
$$
where $ (x_1,y_1,z) $ are the local coordinates in $ \mathcal{A}'_1 $. We have $ x_1 = f_1 $ and $ y_1 = g_1 $. Moreover,
$$
F = z - x_1 (y_1+1) - \phi \ ,
$$
thus $ F $ is $ \gamma $-prepared (note that $x_1^2$ divides $\phi$). In $ \mathcal{A}'_1 $ the ramification index of the dependent variable $ z $ is $ 1 $, therefore the combinatorial blow-up with center $ (x_1,z) $ is a $ z $-Puiseux's package $ \mathcal{A}'_1 \rightarrow \mathcal{A}_1 $. We have local coordinates $ (x_1,y_1,z_1) $ where $ z_1 $ is given by
$$
z = x_1 (z_1+1) \ .
$$
It follows that
$$
F = x_1 (z_1+\xi_1) - x_1 (y_1+\xi_1) - \phi = x_1 \left(z_1 - y_1 - \frac{\phi}{x_1} \right)\ .
$$
The situation in $ \mathcal{A}_1 $ is similar to the one we have in $ \mathcal{A} $,but now the variables have half value than the original ones. We can iterate this process by considering an infinite sequence 
\begin{equation*}
	\xymatrix{ \mathcal{A} \ar[r]^{\pi_1}  & \mathcal{A}'_1 \ar[r]^{\tau_1} & \mathcal{A}_1 \ar[r]^{\pi_2} & \cdots }
\end{equation*}
where $ \pi_i : \mathcal{A}_i \rightarrow \mathcal{A}'_{i+1} $ is a $ y_i $-Puiseux's package (in fact it is also a $ \gamma $-preparation for $ F $) and $ \tau_i : \mathcal{A}'_i \rightarrow \mathcal{A}_{i} $ is a $ z_i $-Puiseux's package. We have $ \mathcal{A}_i = (\mathcal{O}_i,(x_i,y_i,z_i)) $, where $ \mathcal{O}_i = \mathbb{C}[z_i,y_i,z_i]_{(x_i,y_i,z_i)} $. Moreover, we know the values of the coordinates in any $ \mathcal{A}_i $ since $ x_i=f_i $, $ y_i=g_i $ and $ z_i=h_i $. Furhermore, in all these models the function $ F $ is $ \gamma $-prepared and it is written as
$$
F = x_i^{2^i -1} U_i \left( z_i - y_i - \sum_{k=1}^\infty c_k x_i^{2^i(k-1)+1} V_i \right) \ ,
$$
where $ U_i $ and $ V_i $ are units of $ \mathcal{O}_i = \mathbb{C}[z_s,y_s,z_s]_{(x_i,y_i,z_i)} $. Thus we have
$$
\nu_{\mathcal{A}_i}(F) = \nu(x_i^{2^i -1}) = (2^i -1)\nu(f_i) = 1 - 2^{-i} < 1 \leq \gamma \ ,
$$
so $ F $ is not $ \gamma $-final in any $ \mathcal{A}_i $.

In order to avoid this accumulation phenomenon we perform Tschirnhausen coordinate change $ \mathcal{A} \rightarrow \mathcal{B} $ given by
$$
\tilde{z} = z - y \ .
$$
We have $ \nu(\tilde{z})=\nu(x) $ thus the ramification index of $ \tilde{z} $ is $ 1 $. The formal function $ F $ is written in this variables as
$$
F = \tilde{z} - \phi \ ,
$$
thus it is $ \gamma $-prepared. A combinatorial blow-up with center $ (x,\tilde{z}) $ is a $ \tilde{z} $-Puiseux's package $ \mathcal{B} \rightarrow \mathcal{B}_1 $. In $ \mathcal{B}_1 $ we have local coordinates $ (x,y,\tilde{z}_1) $ where $ \tilde{z}_1 $ is given by
$$
\tilde{z}_1 = \frac{\tilde{z}}{x} - c_1
$$
and its value is $ \nu(\tilde{z}_1) = 1 $. The function $ F $ is written now as
$$
F= x \left( \tilde{z}_1 - \sum_{k=1}^{\infty} c_{k+1} x^k \right) \ ,
$$
so it is $ \gamma $-prepared. Performing $ \tilde{z} $-Puiseux's packages
\begin{equation*}
	\xymatrix{
		\mathcal{B} \ar[r]^{\theta_1}  & \mathcal{B}_1 \ar[r]^{\theta_2} & \cdots }
\end{equation*}
we obtain parameterized regular local models $ \mathcal{R_i} $ with coordinates $ (x,y,\tilde{z}_i) $, with $ \nu(\tilde{z}_i) = 1 $, in such a way that $ F $ is written as
$$
F = x^i ( \tilde{z}_i - \sum_{k=1}^{\infty} c_{k+i} x^k \ .
$$
In any of these models $ F $ is $ \gamma $-prepared and we have
$$
\nu_{\mathcal{B}_i}(F) = \nu(x^i) = i \ .
$$
We see that for any index $ i $ greater than $ \gamma $ the function $ F $ is $ \gamma $-final recessive in $ \mathcal{B}_i $.

In this example we see that the formal function $ F $ has ``infinite value'' with respect to $ \nu $. The same happens if we consider the $ 1 $-form with formal coefficients $ dF $. Moreover, we see that this accumulation phenomenon can also happen when dealing with convergent functions: once we have fixed a value $ \gamma \geq 1 $ is is enough to consider the function $ F_T = z-y -\sum_{k=1}^T c_k x^k $ for any $ T > \gamma $. Although $ \nu(F_T) = T+1 $, while we perform Puiseux's packages exclusively, the explicit value of $ F $ will be strictly less than $ 1 $. Again, $ dF $ provides an example of the same phenomenon for $ 1 $-forms.

\begin{center}
	- - -
\end{center}

Let us explain now how to $ \gamma $-prepare a $ 1 $-form $ \omega \in N_{\mathcal{A}}^{\ell+1}$ such that
$$
	\nu_{\mathcal A}(\omega\wedge d\omega)\geq 2\gamma \ .
$$
Consider the decomposition in powers of the dependent variable $ z $
$$
\omega = \sum_{k=0}^{\infty} z^k \omega_k \ , \quad \omega_k = \eta_k + f_k \frac{dz}{z} \ , 
$$
where $ \eta_k \in N_{\mathcal{A}}^{\ell} $ and $ f_k \in R_{\mathcal{A}}^{\ell} $. Note that unlike the case of functions, the levels $ \omega_s $ of a $ 1 $-form $ \omega \in N_{\mathcal{A}}^{\ell+1} $ do not belong to $ N_{\mathcal{A}}^{\ell} $. For each level $ \omega_s = \eta_s + f_s dz/z$ we associate a pair $ (\eta_s,f_s) $ in such a way that we can apply $ T_5(\ell) $ as we will detail later.

The explicit value $ \nu_{\mathcal{A}}(\omega_k) $ of a level $ \omega_k $ is the minimum among $ \nu_{\mathcal{A}}(\eta_k) $ and $ \nu_{\mathcal{A}}(f_k) $. Given a value $\alpha\in \Gamma$ we say that $ \omega_k $ is \emph{$\alpha$-final dominant} if one of the following conditions holds:
\begin{itemize}
	\item $ \nu_{\mathcal{A}}(\eta_k) < \nu_{\mathcal{A}}(f_k) $ and $ \eta_k $ is $\alpha$-final dominant;
	\item $ \nu_{\mathcal{A}}(f_k) < \nu_{\mathcal{A}}(\eta_k) $ and $ f_k $ is $\alpha$-final dominant;
	\item $ \nu_{\mathcal{A}}(f_k) = \nu_{\mathcal{A}}(\eta_k) $ and both $ f_k $ and $\eta_k$ are $\alpha$-final dominant.
\end{itemize}
The Truncated Newton Polygon $ \mathcal{N} $ and the Truncated Dominant Newton Polygon $\textrm{Dom}{\mathcal N}$ of a $ 1 $-form are obtained in the same way we do it in the case of functions: considering the positive convex hull in $ \mathbb{R}^2 $ of the points $(\nu_{\mathcal A}(\omega_s),s)$ (all of them of $ \mathcal{N} $, only the $ (\gamma-s\nu(z)) $-final dominant ones for $\textrm{Dom}{\mathcal N}$) together with $(\gamma,0)$ and $(0,\gamma/\nu(z))$. In the same way, we say that $ \omega $ is $ \gamma $-prepared if
$$  \mathcal N=\textrm{Dom}{\mathcal N} $$
and moreover $(\nu_{\mathcal A}(\omega_s),s)$ is a interior point of $\mathcal N$ for all $ s $ such that $\omega_s$ is non-dominant.

As in the case of functions, the first step in the process of $ \gamma $-preparation of a $ 1 $-form consists in performing a $ \ell $-nested transformation in order to get the maximum number of $ (\gamma-s\nu(z)) $-final dominant levels $ \omega_s $. This step do not use the hypothesis induction. Furthermore, performing a $ \ell $-nested transformation does not mix the levels between them, and inside each level there are no interferences between the differential part $ \eta_s $ and the functional part $ f_s $.

After completing this first step we have that $ \textrm{Dom}{\mathcal N} $ is stable, it means, $ \textrm{Dom}{\mathcal N} $ will be the same even if we perform more $ \ell $-nested transformations. We have $ {\mathcal N} \subset \textrm{Dom}{\mathcal N} $ and we must to determine a transformation which make both polygons equal (the remaining condition to reach a $ \gamma $-prepared situation is easy to obtain once both polygons are equal).

In order to apply $ T_5(\ell) $ to the levels $ \omega_s $ we need to know $ \nu_{\mathcal A}(\eta_s\wedge d\eta_s) $. We have that $ \nu_{\mathcal A}(\omega\wedge d\omega) \geq 2 \gamma $, however, we do not know if the same happens with the $ 3 $-forms $ \eta_s \wedge d\eta_s$. In oder to know ``how integrable'' is each $ \eta_s $ we use the following formula:
$$
	\omega \wedge d \omega = \sum_{m=0}^{\infty} z^m \Big( \Theta_m + \frac{dz}{z} \Delta_m \Big)
$$
where
$$
\Theta_m := \sum_{i+j=m} \eta_i \wedge d \eta_j
$$
and
$$
\Delta_m := \sum_{i+j=m} j \eta_j \wedge \eta_i +f_i d \eta_j + \eta_i \wedge d f_j \ .
$$
Since $ \nu_{\mathcal A}(\omega\wedge d\omega) \geq 2 \gamma $ we have
\begin{equation*}
	\nu_{\mathcal A}(\Theta_m) \geq 2 \gamma \quad \text{and} \quad 	\nu_{\mathcal A}(\Delta_m) \geq 2 \gamma
\end{equation*}
for every $ m \geq 0 $. On the other hand, for all $ 1 $-form $ \sigma $ we have
$$
\nu_{A}(d\sigma) \geq \nu_{A}(\sigma) \ .
$$
Thanks to this observation, we have that $ \nu_{\mathcal A}(\Theta_{2s}) \geq 2 \gamma $ implies
$$
\nu_{\mathcal A}(\eta_s\wedge d\eta_s) \geq \min \big\{ \{ 2\gamma \} \cup \{ \nu_{\mathcal A}(\eta_{s-i})+\nu_{\mathcal A}(\eta_{s+i})\}_{1 \leq i \leq s} \big\} \ . 
$$
This formula allows us to apply $ T_5(\ell) $ starting with the lowest non-dominant level and continuing with the upper levels until reach the highest non-dominant level below $ \gamma/\nu(z) $. Repeating this process we can approach $\mathcal N$ to $\textrm{Dom}{\mathcal N}$ until a prefixed distance as it is explained in Lemma \ref{le:epsilon}. In particular we can get that all the vertices of $\textrm{Dom}{\mathcal N}$ are also vertices of $\mathcal N$. This procedure is essential but it is not enough to complete the $ \gamma $-preparation.

After bringing $\mathcal N$ over $\textrm{Dom}{\mathcal N}$ we complete the $ \gamma $-preparation by using ``truncated proportionality'' statements as the following truncated version of the De Rham-Saito Lemma (\cite{Sai}):
\begin{quote}
	Let $ \eta $ and $ \sigma $ be $ 1 $-forms. If $ \eta $ is log-elementary and $ \nu_{\mathcal{A}}(\eta \wedge \sigma) = \alpha $ then there is a function $ f $ and a $ 1 $-form $ \bar{\sigma} $ with $ \nu_{\mathcal{A}}(\sigma) > \alpha $ such that $ \sigma = f \eta + \bar{\sigma} $.
\end{quote}

\begin{center}
 - - -
\end{center}

Once we have completed the $ \gamma $-preparation process, we consider, as we did in the case of functions, the critical height $ \chi $ of $\mathcal N$ as the main control invariant. We must to study the behavior of $ \chi $ after performing a $ z $-Puiseux's package, with rational contact function $ z^d/ \boldsymbol{x^p} $, followed by a new $ \gamma $-preparation. Unless one of the \emph{resonant conditions} (R1) or (R2) defined in Chapter \ref{ch:getting_final_forms} is satisfied, the critical height drops. If at one moment we have $ \chi=0 $, one more $ z $-Puiseux's package followed by a $ \gamma $-preparation will produce a $ \gamma $-final situation.

Condition (R1) requires the ramification exponent $ d $ to be strictly greater than $ 1 $. As we will show, if this condition is satisfied it can not be satisfied again while the critical height remains stable. On the other hand, condition (R2) requires $ d=1 $, but unlike (R1), it can be satisfied infinitely many times. 

Therefore, it only remains to consider the case in which condition (R1) is satisfied indefinitely. In this situation, each $ z $-Puiseux's package increases the explicit value of $ \omega $. If it becomes greater than $ \gamma $, we will reach a parameterized regular local model in which $ \omega $ is $ \gamma $-final recessive. However, as in the case of functions, it may happen that the explicit value ``accumulates''. To avoid this phenomenon, we perform a Tschirnhausen change of coordinates. Such a change of variables will be determined thanks to truncated proportionality properties similar to the ones used in the $ \gamma $-preparation process.

\section{Preliminaries}

\subsection{Codimension one foliations}
Let $ k $ be an algebraically closed field of characteristic zero and consider a finitely generated field extension $ K/k $. Let $ \Omega_{K/k} $ be the module of K\"{a}hler differentials. Let $ d: K \rightarrow  \Omega_{K/k} $ be the exterior derivative. We have that a finite subset $ \{ \alpha_1, \alpha_2,\dots,\alpha_n \} \subset K $ is a transcendence basis of $ K/k $ if and only if $ \Omega_{K/k} $ is a free $ K $-module generated by $ \{ d\alpha_1, d\alpha_2,\dots, d\alpha_n \} $ (see \cite{Ku}, Corollary 5.4). In particular we have that $ \Omega_{K/k} $ is a $ K $-vector space of dimension
$$
\dim_{K}(\Omega_{K/k}) = \operatorname{tr.deg}(K/k) \ .
$$
\begin{defi}
	A \emph{rational codimension one foliation of $ K/k $} is a one-dimensional $ K $-vector subspace $ \mathcal{F} \subset \Omega_{K/k} $ such that for any $ \omega \in \mathcal{F} $ the \emph{integrability condition}
	$$
	\omega \wedge d \omega = 0
	$$
	is satisfied.
\end{defi}
A projective model of $ K/k $ is a projective $ k $-variety $ M $, in the sense of scheme theory, such that $ K = K(M) $ is its field of rational functions. Let $ n $ be the dimension of the variety $ M $. We have that $ n := \dim(M) = \operatorname{tr.deg}(K/k) $. Let $ P \in M $ be a point and denote by $ \mathcal{O}_{M,P} $ and $ \mathfrak{m}_{M,P} $ its local ring and its maximal ideal respectively. Suppose that $ P $ is $ k $-rational, it means, the residue field $ \kappa_{M,P} := \mathcal{O}_{M,P} / \mathfrak{m}_{M,P} $ is isomorphic to the ground field $ k $ (so in particular $ P $ is a closed point). The point $ P \in M $ is regular if $ \mathcal{O}_{M,P} $ is a regular local ring of Krull dimension $ n $. Let $ \Omega_{\mathcal{O}/k} $ be the $ \mathcal{O} $-module of K\"{a}hler differentials. The \emph{Jacobian Criterion} (see \cite{Ku}, Theorem 7.2) states that the point $ P $ is regular if and only if $ \Omega_{\mathcal{O}/k} $ is a free $ \mathcal{O} $-module of rank $ n $.

Fix a regular point $ P \in M $ and denote its local ring $ \mathcal{O}_{M,P} $ by $ \mathcal{O} $. Since $ \operatorname{Frac}(\mathcal{O}) = K $ we have an inclusion of $ \mathcal{O} $-modules
\begin{equation*} 
	\Omega_{\mathcal{O}/k} \hookrightarrow \Omega_{K/k} = \Omega_{\mathcal{O}/k} \otimes_{\mathcal{O}} K \ .
\end{equation*}
Given a rational codimension one foliation $ \mathcal{F} \subset \Omega_{K/k} $ and a point $ P \in M $ define
$$ \mathcal{F}_{M,P} := \mathcal{F} \cap \Omega_{\mathcal{O}/k} \ . $$
We have that $ \mathcal{F}_{M,P} $ is a free rank one $ \mathcal{O} $-submodule of $ \Omega_{\mathcal{O}/k} $. Let $ \{z_1,z_2,\dots,z_n\} $ be a regular system of parameters of its local ring $ \mathcal{O} $. The $ z_i $'s are algebraically independent over $ k $ so they are a transcendence basis of $ K/k $, thus we have
$$  \Omega_{K/k} \simeq \bigoplus_{i=1}^n dz_i K \quad \text{ and } \quad \Omega_{\mathcal{O}/k} \simeq \bigoplus_{i=1}^n dz_i \mathcal{O} . $$
Take an element $ \omega = \sum a_i dz_i $ of $ \mathcal{F}_{M,P} $. We have that $ \omega $ generates $ \mathcal{F}_{M,P} $ as $ \mathcal{O} $-module if and only if $ a_1,a_2,\dots,a_n $ are coprime elements of $ \mathcal{O} $. In fact, let $ d \in \mathcal{O} $ be the greatest common divisor of the coefficients $ a_i $. Denote $ \tilde{a}_i := d^{-1} a_i $. Since $ \tilde{a}_i \in \mathcal{O} $, the $ 1 $-form $ \tilde{\omega} = \sum \tilde{a}_i dz_i $ is an element of $ \mathcal{F}_{M,P} $. If $ \omega $ is a generator of $ \mathcal{F}_{M,P} $ then $ d $ has to be a unit since $ \tilde{\omega} = d^{-1} \omega $. On the other hand, if $ a_1,a_2,\dots,a_n $ are not coprimes $ d^{-1} \notin \mathcal{O} $, so $ \omega $ is not a generator.

\begin{prop}\label{pr:regular_conditions}
	The following are equivalent:
	\begin{enumerate}
		\item $ \Omega_{\mathcal{O}/k} / \mathcal{F}_{M,P} $ is a free $ \mathcal{O} $-module of rank $ n - 1 $;
		\item there is a decomposition $ \Omega_{\mathcal{O}/k} = \mathcal{F}_{M,P} \oplus \mathcal{J} $ with $ \mathcal{J} $ a free $ \mathcal{O} $-module of rank $ n - 1 $;
		\item there exists an element $ \omega = \sum a_i dz_i \in \mathcal{F}_{M,P} $ with $ (a_1,a_2,\dots,a_n) = \mathcal{O} $.
	\end{enumerate}
	\begin{proof}
		The equivalence $ 1) \Leftrightarrow 2) $ is direct. For $ 2) \Rightarrow 3) $ note that $ (a_1,a_2,\dots,a_n) = \mathcal{O} $ implies that the coefficients $ a_i $ are coprimes, hence $ \omega $  is a generator of $ \mathcal{F}_{M,P} $. Let $ i_0 $ be an index such that $ a_{i_0} $ is a unit of $ \mathcal{O} $. Taking $ \mathcal{J} $ as $ \mathcal{J} = \oplus_{i \neq i_0} dz_i \mathcal{O}  $ we obtain $ 2) $. Finally for $ 1) \Rightarrow 3) $ suppose $ 3) $ is false. We have that the classes of $ dz_i $ modulo $ \mathcal{F}_{M,P} $ are $ \mathcal{O} $-independents, so the rank of $ \Omega_{\mathcal{O}/k} / \mathcal{F}_{M,P} $ is at least $ n $.
	\end{proof}
\end{prop}
\begin{defi}
	A rational codimension one foliation $ \mathcal{F} $ is \emph{regular at a point $ P \in M $} if $ P $ is a non-singular point of $ M $ and the equivalent conditions of Proposition \ref{pr:regular_conditions} are satisfied.
\end{defi}
\begin{rem}
	Given a projective model $ M $ of $ K $ the set $ \operatorname{Reg}_{\mathcal{F}}(M) $ of points where $ \mathcal{F} $ is regular is a non-empty open subset of $ M $.
\end{rem}
Let $ \boldsymbol{x}=(x_1,x_2,\dots,x_r) $ be the first $ r $ elements of the regular system of parameters $ \{z_1,z_2,\dots,z_n\} $ and let $ \boldsymbol{y}=(y_1,y_2,\dots,y_{n-r}) $ be the remaining ones. Let $ \Omega_{\mathcal{O}/k}(\log \boldsymbol{x}) $ be the free $ \mathcal{O} $-module
$$ \Omega_{\mathcal{O}/k}(\log \boldsymbol{x}) := \bigg( \bigoplus_{i=1}^r \frac{dx_i}{x_i} \mathcal{O} \bigg) \oplus \bigg( \bigoplus_{j=1}^{n-r} dy_i \mathcal{O} \bigg) . $$
We have $ \mathcal{O} $-module monomorphisms
\begin{eqnarray}\label{eq:inclusion_modules_differentials}
	\Omega_{\mathcal{O}/k} & \hookrightarrow & \Omega_{\mathcal{O}/k}(\log \boldsymbol{x}) \\
	\omega = \sum_{i=1}^n a_i dz_i  & \mapsto & \sum_{i=1}^r x_i a_i \frac{dx_i}{x_i} + \sum_{j=r+1}^n a_i dy_j \nonumber
\end{eqnarray}
and
\begin{eqnarray*}
	\Omega_{\mathcal{O}/k}(\log \boldsymbol{x}) & \hookrightarrow & \Omega_{K/k} \\
	\sum_{i=1}^r a_i \frac{dx_i}{x_i} + \sum_{j=r+1}^n a_i dy_j & \mapsto & \sum_{i=1}^r \frac{a_i}{x_i} dx_i + \sum_{j=r+1}^n a_i dy_j .
\end{eqnarray*}
Given a foliation $ \mathcal{F} \subset \Omega_{K/k} $ and a point $ P \in M $ denote
$$ \mathcal{F}_{M,P}(\log \boldsymbol{x}) := \mathcal{F} \cap \Omega_{\mathcal{O}/k}(\log \boldsymbol{x}) \ . $$
We have that $ \mathcal{F}_{M,P}(\log \boldsymbol{x}) $ is a rank one free $ \mathcal{O} $-submodule of $ \Omega_{\mathcal{O}/k}(\log \boldsymbol{x}) $. Take an element $ \omega = \sum a_i \frac{dx_i}{x_i} + \sum a_j dy_j \in \mathcal{F}_{M,P}(\log \boldsymbol{x}) $. We have that $ \omega $ generates $ \mathcal{F}_{M,P}(\log \boldsymbol{x}) $ as $ \mathcal{O} $-module if and only if $ a_1,a_2,\dots,a_n $ are coprime elements of $ \mathcal{O} $. 
\begin{defi}
	Let $ P \in M $ be a closed point. Let $ (\boldsymbol{x},\boldsymbol{y}) $ be a regular system of parameters of $ \mathcal{O}_{M,P} $. A foliation $ \mathcal{F} $ is \emph{$ \boldsymbol{x} $-log-final at $ P $} if $ P $ is a non-singular point of $ M $ and a generator of $ \mathcal{F}_{M,P}(\log \boldsymbol{x}) $
	$$ \omega = \sum_{i=1}^r a_i \frac{dx_1}{x_1} + \sum_{j=r+1}^n a_j dy_j,$$
	satisfies one of the following conditions:
	\begin{enumerate}
		\item $ (a_1,a_2,\dots,a_r) = {\mathcal O} $ ;
		\item $ (a_1,a_2,\dots,a_r)\subset {\mathfrak m}$ and in addition
			$$ (a_1,a_2,\dots,a_r) \not \subset (x_1,x_2,\dots,x_r)+ \mathfrak{m}^2 \ . $$
	\end{enumerate}
	Points satisfying the first condition are called \emph{$ \boldsymbol{x} $-log-elementary} and the ones satisfying the second condition are called \emph{$ \boldsymbol{x} $-log-canonical}.
\end{defi}
\begin{defi}
	A foliation $ \mathcal{F} $ is \emph{log-final at $ P $} if there exists a regular system of parameters $ (\boldsymbol{x},\boldsymbol{y}) $ of $ \mathcal{O}_{M,P} $ such that $ \mathcal{F} $ is $ \boldsymbol{x} $-log-final at $ P $.
\end{defi}
\begin{rem}
	Given a projective model $ M $ of $ K $ the set $ \operatorname{Log-Final}_{\mathcal{F}}(M) $ of points where $ \mathcal{F} $ is log-final is a non-empty open subset of $ M $.
\end{rem}

\subsection{Valuations}
We collect now some classical definitions and results from valuation Theory. We omit the proofs of many assertions, which can be found in \cite{ZarSa}, Chapter 6.
\begin{defi}
	Let $ K/k $ be a field extension. A subring $ R \subset K $ is a \emph{valuation ring of $ K / k$} if $ \operatorname{Frac}(R)=K $, $ k \subset R $ and the following property holds:
	$$ \forall x \in K , \quad x \notin R \Rightarrow x^{-1} \in R .$$
\end{defi}
It follows from the definition that $ R $ is a local ring with maximal ideal
$$ \mathfrak{m} = \{ x \in R \ |\ x^{-1} \notin R \}. $$
\begin{defi}
	Let $ K $ be an extension field of $ k $ and let $ \Gamma $ be an additive abelian totally ordered group. A valuation of $ K/k $ with values in $ \Gamma $ is a surjective mapping
	$$ \nu : K^* \longrightarrow \Gamma $$
	such that the following conditions are satisfied:
	\begin{itemize}
		\item $ \nu(xy) = \nu(x) + \nu(y) $,
		\item $ \nu(x+y) \geq \min \{ \nu(x),\nu(y) \} $,
		\item $ \nu(\alpha) = 0 $ for every $ \alpha \in k^* $.
	\end{itemize}
\end{defi}
It is usual to add formally the element $ + \infty $ to the group $ \Gamma $ with the usual arithmetic rules ($ \alpha + \infty = \infty, \ \beta < \infty \ \forall \alpha, \beta \in \Gamma $) and consider the valuation $ \nu : K \longrightarrow \Gamma \cup \{ \infty \} $ with $ \nu(0) = \infty $. We will use this convention.

Given a valuation $ \nu $ of $ K/k $ the set
$$ R_{\nu} := \big\{ x\in K \ |\ \nu(x)\geq 0 \big\}$$
is a valuation ring of $ K/k $ with maximal ideal
$$ \mathfrak{m}_{\nu} := \{ x \in K \ |\ \nu(x)> 0\} .$$
The ring $ R_{\nu} $ is \emph{the valuation ring of $ \nu $}. Its quotient field $ \kappa_{\nu} := R_{\nu} / \mathfrak{m}_{\nu} $ is \emph{the residue field of $ \nu $}.

Conversely, given a valuation ring $ R $ of $ K/k $ we can construct a valuation $ \nu_R $ of $ K/k $ such that $ R = R_{\nu_R} $. Since $ R $ is a subring of $ K $, its invertible elements form a subgroup $ R^*=R \setminus \mathfrak{m} $ of the multiplicative group $ K^* $. Let $ \Gamma $ be the quotient group $ K^* / R^* $. It is an abelian totally ordered group whose order relation is given by the divisibility in R:
$$ x  R^* \leq y  R^* \Leftrightarrow x \text{ divides } y \text{ in } R .$$
This is clearly an order relation on $ \Gamma $. Since $ R $ is a valuation ring we have that it is a total order:
$$ x  R^* \not\leq y  R^* \Rightarrow x \text{ does not divide } y \text{ in } R \Rightarrow \frac{y}{x} \notin R $$
$$ \Rightarrow \frac{x}{y} \in R \Rightarrow y \text{ divides } x \text{ in } R \Rightarrow y  R^* \leq x  R^* .$$
The valuation is the natural group homomorphism $ \nu : K^* \rightarrow \Gamma = K^* / R^* $. The positive part $ \Gamma_{+} $ is the image of the maximal ideal $ \mathfrak{m} $.

A totally ordered group $ G $ is \emph{archimedean} if it satisfies the \emph{archimedean property}:
$$ \forall x,y \in G_{>0} ,\ \exists n \in \mathbb{N} \ |\ y \leq x . $$
It is well-known that a totally ordered group is archimedean if and only if it is isomorphic as totally ordered group to some subgroup of $ (\mathbb{R},+) $.
\begin{defi}
	A valuation $ \nu $ of $ K/k $ is \emph{archimedean} if its value group $ \Gamma_{\nu} $ is archimedean.
\end{defi}
The \emph{rank of $ \nu $} is defined by
$$ \operatorname{rk}(\nu) := \dim_{Krull}R_{\nu}. $$
This number coincides with the rank of the ordered group $ \Gamma $. We have that a valuation $ \nu $ is archimedean if and only if $ \operatorname{rk}(\nu) = 1 $.

Given a valuation $ \nu $ of $ K/k $ the residue field $ \kappa_{\nu} := R_{\nu} / \mathfrak{m}_{\nu} $ is the \emph{residue field of $ \nu $}. It follows from the third property of the definition of valuation of $ K/k $ that $ \kappa_{\nu} $ is an extension field of $ k $. We define the \emph{dimension of $ \nu $} by
$$ \dim(\nu) := tr.deg (\kappa_{\nu} / k) . $$
\begin{defi}
	A valuation $ \nu $ of $ K/k $ is \emph{$ k $-rational} if $ \kappa_{\nu} \simeq k $.
\end{defi}
\begin{rem}
	Note that in the case of an algebraically closed ground field $ k $ a valuation $ \nu $ of $ K/k $ is $ k $-rational if and only if $ \dim(\nu) = 0 $.
\end{rem}
If $ \nu $ is a $ k $-rational valuation of $ K/k $, we have that for each $ \phi \in K $ with $ \nu(\phi) \geq 0 $ there exists a unique $ \lambda \in k $ such that $ \nu(\phi- \lambda) > 0 $. The existence of such a $ \lambda \in k $ follows from the fact that $ \kappa_{\nu} \simeq k $. Suppose that there are $ \lambda_1, \lambda_2 \in k $ such that $ \nu(\phi- \lambda_i) > 0 $ for $ i=1,2 $. It follows that $ \nu(\lambda_1 - \lambda_2) = \nu((\phi- \lambda_2) - (\phi - \lambda_1)) >0$ hence $ \lambda_1 =  \lambda_2 $.

The largest number of elements of $ K $ with $ \mathbb{Z} $-independent values is the \emph{rational rank of $ \nu $}
$$ \operatorname{rat.rk}(\nu) := \dim_{\mathbb{Q}}(\Gamma \otimes_{\mathbb{Z}} \mathbb{Q}) \ . $$
By the Abhyankar's Inequality we have that
$$ 0 \leq  \operatorname{rk}(\nu) \leq \operatorname{rat.rk}(\nu) \leq \operatorname{tr.deg}(K/k) = n \ . $$

Let $ (A,m) $ and $ (B,n) $ be two local rings. We say that $ A $ is dominated by $ B $ if $ A \subset B $ and $ m = A \cap n $. The relation of domination is denoted by $ A \preceq B $.
\begin{defi}
	Let $ X $ be an algebraic variety over $ k $ whose function field is $ K $. A point $ P \in X $ is called \emph{the center of $ \nu $ at $ X $} if $ \mathcal{O}_{X,P} \preceq R_{\nu} $.
\end{defi}
We will work with projective models of a fixed function field $ K $, i. e., algebraic projective varieties with function field $ K $. The following proposition, whose proof can be found in \cite{Va}, guarantees the existence and uniqueness of the center of any valuation of $ K/k $ in such models.
\begin{prop}
	If $ X $ is a complete variety over a field $ k $, any valuation of $ L/k $, where $ L/K(X) $ is an extension of the function field $ K(X) $ of $ X $, has a unique center on $ X $.
\end{prop}
If $ P $ is a point of a variety $ X $, its residue field is by definition $ \kappa_{X,P} := \mathcal{O}_{X,P} / \mathfrak{m}_{X,P} $, which is an extension field of $ k $. If we have $ \kappa_P = k $ we say that $ P $ is a $ k $-rational point.
If $ P $ is the center of $ \nu $ at $ X $ there are field extensions
$$ k \subset \kappa_P \subset \kappa_{\nu} . $$
In particular we have $ \dim(P) \leq \dim(\nu) $, where $ \dim(P) := tr.deg (\kappa_P / k) $. Let us note that if $ \nu $ is $ k $-rational then the three fields are the same and the center of $ \nu $ in each projective model has dimension $ 0 $ and it is a $ k $-rational point.

Let $ f : X' \rightarrow X $ a birational morphism. If $ P' $ and $ P $ are the centers of $ \nu $ at $ X' $ and $ X $ respectively, we have $ f(P') = P $ and $ f $ induces a domination of local rings $ \mathcal{O}_{X,P} \preceq \mathcal{O}_{X',P'} $. As a consequence we obtain $ \dim(P) \leq \dim(P') $. The proof of the next statement can also be found in \cite{Va}.
\begin{prop}
	Given a projective model $X$ of $K$, there is a birational morphism $ \pi : X' \rightarrow X $ with $\dim P' = \dim \nu $, where $ P' $ is the center of $ \nu $ in  $ X' $.
\end{prop}
Given a valuation $ \nu $ of $ K/k $, the Local Uniformization Problem consists in determine a projective model $ M $ of $ K $ such that the center of $ \nu $ in $ M $ is a regular point. This problem for varieties over a ground field of characteristic zero was stated and solved by Zariski (see \cite{Zar}). In this work, instead of regularity at the center of the valuation, we require that a given rational codimension one foliation is log-final at the center of the valuation. The precise statement we prove in this work is the following refinement of Theorem \ref{th:I}:
\begin{theo}\label{th:1}
	Let $ k $ be a field of characteristic zero and let $ K/k $ be a finitely generated field extension. Let $ \mathcal{F} $ be a rational codimension one foliation of $ K/k $. Given a projective model $ M $ of $ K/k $ and a $ k $-rational archimedean valuation $ \nu $ of $ K/k $, there is a finite composition of blow-ups with codimension two centers
	$$
	\tilde{M} \rightarrow M
	$$
	such that $ \mathcal{F} $ is log-final at the center of $ \nu $ in $ 	\tilde{M} $.
\end{theo}
Given a function field $ K/k $, we can invoke Hironaka's Resolution of Singularities \cite{Hi1,Hi2} or Zariski's Local Uniformization \cite{Zar} in order to obtain a projective model of $ K/k $ regular at the center oh the valuation $ \nu $. In that situation we have that Theorem \ref{th:1} implies Theorem \ref{th:I}.

\section{Transformations adapted to a valuation}\label{ch:prlm}

\subsection{Parameterized regular local models}
Let $ K/k $ be a finitely generated field extension and let $ \nu $ be a $ k $-rational archimedean valuation of $ K/k $ of rational rank $ r $.
\begin{defi}
	A \emph{parameterized regular local model} for $K/k,\nu$ is a pair $ \mathcal{A} = \big( \mathcal{O} ; (\boldsymbol{x},\boldsymbol{y}) \big)$ such that
	\begin{itemize}
		\item $ \mathcal{O} \subset K $ is the regular local ring of the center of $ \nu $ in some projective model of $ K $;
		\item $ (\boldsymbol{x},\boldsymbol{y})=(x_1,x_2,\dots,x_r,y_1,y_2,\dots,y_{n-r})$ is a regular system of parameters of $\mathcal O$ such that $ \{ \nu(x_1),\nu(x_2),\dots,\nu(x_r) \} \subset \Gamma $ is a basis of $ \Gamma \otimes \mathbb{Q} $ . We call $ \boldsymbol{x} $ the \emph{independent variables} and $ \boldsymbol{y} $ the \emph{dependent variables}.
	\end{itemize}
\end{defi}
The following proposition guarantees the existence of parameterized regular local models.
\begin{prop}\label{pr:prml_existence}
	Given a projective model $ M_0 $ of $ K $, there is a morphism $ M \longrightarrow M_0 $ which is the composition of blow-ups with non-singular centers, such that the center $ P $ of $ \nu $ at $ M $ provides a local ring $\mathcal{O}=\mathcal{O}_{M,P}$ for a parameterized regular local model $ \mathcal{A} $ for $K/k,\nu$.
\begin{proof}
	By Zariski's Local Uniformization \cite{Zar40} we get a projective model $ M' $ of $ K $ non-singular at the center $ P' $ of $ \nu $, jointly with a birational morphism $ M' \rightarrow M_0 $ that is the composition of a finite sequence of blow-ups with non-singular centers.
	
	Take $ r $ elements $ g_1 ,g_2, \dots , g_r \in K $ such that $ \nu(g_1),\nu(g_2),\dots,\nu(g_r) $ are $ \mathbb{Z} $-linearly independent. Since $ K = \operatorname{Frac}(\mathcal{O}_{M',P'}) $, multiplying by the common denominator if necessary we obtain $ r $ elements $ f_1 , f_2,\dots , f_r \in \mathcal{O}_{M',P'} $ with independent values. Another application of Zariski's Local Uniformization gives a birational morphism $ M \rightarrow M' $, that is also a composition of a finite sequence of blow-ups with non-singular centers, such that each $ f_i $ is a monomial (times a unit) in a suitable regular system of parameters of $ \mathcal{O}_{M,P} $, where $ P $ is the center of $ \nu $ in $ M $. Let $ \boldsymbol{z} = (z_1, z_2,\dots , z_n ) $ be such a regular system of parameters. We have
	$$ f_i = U_i \boldsymbol{z}^{\boldsymbol{m_i}} \quad , U_i \in \mathcal{O}_{M,P} \setminus \mathfrak{m}_{M,P}  \text{ for } i=1,2,\dots,r , $$
	where $ \boldsymbol{m_i} \in \mathbb{Z}_{\geq 0}^n $. In terms of values, we have $ \nu(f_i) = \sum m_{ij} \nu(z_j) $. This implies that there are $ r $ variables among the $ z_j $'s whose values are $ \mathbb{Z} $-linearly independent.
\end{proof}
\end{prop}

\subsection{Transformations of parameterized regular local models}
Parameterized regular local models are the ambient spaces in which we will work. A transformation between parameterized regular local models is formed by an inclusion of regular local rings and a specific selection of a regular system of parameters of the new ring. We consider two elementary operations: blow-ups and change of coordinates. Certain composition of these operations, called nested transformations, are the transformations between parameterized regular local models which we allow. Given such a transformation
$$
\pi : \mathcal{A} \longrightarrow \mathcal{A}' \ ,
$$
we denote with the same symbol the corresponding inclusion of local rings
$$
\pi : \mathcal{O} \longrightarrow \mathcal{O}' \ ,
$$
and also the induced $ \mathcal{O} $-module homomorphism
$$
\pi : \Omega_{\mathcal{O}/k}(\log \boldsymbol{x}) \longrightarrow \Omega_{\mathcal{O}'/k}(\log \boldsymbol{x}') \ .
$$
\subsubsection{Blowing-up parameterized regular local models}
Let $ \mathcal{A} = \big( \mathcal{O} ; (\boldsymbol{x},\boldsymbol{y}) \big) $ be a parametrized regular local model for $ K,\nu $. As center of blow-up we use only the ideals $ I_{ij}, I_{i}^{j} \subset \mathcal{O} $ defined by
\begin{eqnarray*}
	I_{ij} := (x_i,x_j) \mathcal{O} & \text{for} & 1 \leq i < j \leq r \ ; \\
	I_{i}^{j} := (x_i,y_j) \mathcal{O} & \text{for} & 1 \leq i \leq r , \ 1 \leq j \leq n-r \ .
\end{eqnarray*}
\emph{The blow-up of $ \mathcal{A} $ at $ I_{ij} $} is
$$ \theta_{ij}(\mathcal{A}) : \mathcal{A} \longrightarrow \mathcal{A}' \ , $$
where $ \mathcal{A}' = \big( \mathcal{O}' ; (\boldsymbol{x}',\boldsymbol{y}') \big) $ is defined by:
\begin{itemize}
	\item if $ 1 \leq i < j \leq r $ and $ \nu(x_i) < \nu(x_j) $: put $ x_j' := x_j / x_i $, $ x_{k}'=x_{k} $ for $ k \neq j $ and $ \boldsymbol{y}' := \boldsymbol{y} $;
	\item if $ 1 \leq i < j \leq r $ and $ \nu(x_i) > \nu(x_j) $: put $ x_i' := x_i / x_j $, $ x_{k}'=x_{k} $ for $ k \neq i $ and $ \boldsymbol{y}' := \boldsymbol{y} $;
\end{itemize}
and
$$ \mathcal{O}'=\mathcal{O}[\boldsymbol{x}',\boldsymbol{y}']_{(\boldsymbol{x}',\boldsymbol{y}')} \ .$$
\emph{The blow-up of $ \mathcal{A} $ at $ I_{i}^{j} $} is
$$ \theta_{i}^{j}(\mathcal{A}) : \mathcal{A} \longrightarrow \mathcal{A}' \ , $$
where $ \mathcal{A}' = \big( \mathcal{O}' ; (\boldsymbol{x}',\boldsymbol{y}') \big) $ is defined by:
\begin{itemize}
	\item if $ \nu(x_i) < \nu(y_{j}) $: put $ y_{j}' := y_{j} / x_i $, $ y_j'=y_j $ for $ j \neq l $ and $ \boldsymbol{x}' := \boldsymbol{x} $;
	\item if $ \nu(x_i) > \nu(y_{j}) $: put $ x_i' := x_i / y_{j} $, $ x_i'=x_i $ for $ i \neq k $ and $ \boldsymbol{y}' := \boldsymbol{y} $;
	\item if $ \nu(x_i) = \nu(y_j) $: since $ \kappa_{\nu} = k $, there is $ \lambda \in k $ with $ \nu(y_{j} / x_i - \lambda ) > 0 $. Put $ y_{j}' =  y_{j} / x_i - \lambda $, $ y_k'=y_k $ for $ k \neq j $ and $ \boldsymbol{x}' := \boldsymbol{x} $;
\end{itemize}
and
$$ \mathcal{O}'=\mathcal{O}[\boldsymbol{x}',\boldsymbol{y}']_{(\boldsymbol{x}',\boldsymbol{y}')} \ .$$
The first four cases above are called \emph{combinatorial blow-ups} and the fifth is a \emph{blow-up with translation}.
\begin{rem}
	As we said $ \mathcal{O} $ is the local ring of a point $ P $, the center of $ \nu $ in some projective model $ M $. The ring $ \mathcal{O}' $ is just the local ring of the center of $ \nu $ in the variety obtained after blowing-up $ M $ at the subvariety defined locally at $ P $ by $ I_{ij} $. We have that
	$$ \mathcal{O} \preceq \mathcal{O}' \preceq R_{\nu} \ . $$
\end{rem}

\subsubsection{Ordered change of coordinates}
A change of coordinates does not modify the local ring, it just changes the selection of regular parameters. Let $ \mathcal{A} = \big( \mathcal{O} ; (\boldsymbol{x},\boldsymbol{y}) \big) $ be a parameterized regular local model and let $ y_{\ell} $ be a dependent variable.
\begin{defi}
	An \textit{ordered change of the $ \ell $-th coordinate} is a transformation of parametrized regular local models
	$$ T : \mathcal{A} \longrightarrow \tilde{\mathcal{A}} $$
	where $ \tilde{\mathcal{A}} = \big( \tilde{\mathcal{O}} ; (\tilde{\boldsymbol{x}},\tilde{\boldsymbol{y})} \big) $ is given by
	\begin{itemize}
		\item $ \tilde{\mathcal{O}} := \mathcal{O} \, $;
		\item $ \tilde{x}_i := x_i $ for $ 1\leq i\leq r \, $;
		\item $ \tilde{y}_j := y_j $ for $ 1 \leq j \leq n-r $, $ j \neq \ell$;
		\item $ \tilde{y}_{\ell} := y_{\ell} + \psi $ where $ \psi \in \mathfrak{m} \cap k[\boldsymbol{x},y_{1},y_2,\dots,y_{l-1}] $ is a polynomial such that if we write
		$$ \psi = \sum_{I}\boldsymbol{x}^{I} \psi_{I}(y_1,y_2,\dots,y_{\ell-1}) $$
		we have
		$$ \nu(\boldsymbol{x}^{I}) < \nu(y_{\ell}) \Rightarrow \psi_{I}(y_1,y_2,\dots,y_{\ell-1}) \equiv 0 \ . $$
	\end{itemize}
\end{defi}
Note that we have
$$ \nu(\tilde{y}_{\ell}) \geq \nu(y_{\ell}) \ . $$
Taking differentials in the equations of the coordinate change we obtain explicit equations for the $ \mathcal{O} $-homomorphism between the modules of differentials
\begin{eqnarray*}
	\Omega_{\mathcal{O}/k}(\log \boldsymbol{x}) & \longrightarrow & \Omega_{\tilde{\mathcal{O}}/k}(\log \tilde{\boldsymbol{x}}) \\
	\frac{dx_i}{x_i} & \longmapsto & \frac{d\tilde{x}_i}{\tilde{x}_i} \ , \quad i=1,\dots,r \ ; \\
	dy_j & \longmapsto & d\tilde{y}_j \ , \quad 1 \leq j \leq n-r \ , \ j \neq \ell \ ; \\
	dy_{\ell} & \longmapsto & d \tilde{y}_{\ell} + \sum_{i=1}^{r} x_i \frac{\partial \psi}{\partial x_i} \frac{dx_i}{x_i} + \sum_{j=1}^{\ell-1}  \frac{\partial \psi}{\partial y_j} dy_j \ .
\end{eqnarray*}

\subsubsection{Puiseux's packages}\label{se:Puiseux's_packages}
Let $ \mathcal{A} = \big( \mathcal{O} ; (\boldsymbol{x},\boldsymbol{y}) \big)$ be a parameterized regular local model for $K/k,\nu$. Given a dependent variable $ y_{\ell} $ there is a relation
\begin{equation}\label{eq:dependence_relation}
	d \nu(y_{\ell}) = p_1 \nu(x_1) + \cdots + p_r \nu(x_r).
\end{equation}
Requiring $d>0$ and $\gcd(d,p_1,\ldots,p_r)=1 $ the integers of the above expression are uniquely determined. Denoting by $ \boldsymbol{p}=(p_1,\ldots,p_r)\in \mathbb{Z}^r \setminus \{\boldsymbol{0}\}$, Equation \eqref{eq:dependence_relation} is equivalent to
\begin{equation}\label{eq:valorative_dependence_relation}
	\nu(\frac{y_{\ell}^d}{\boldsymbol{x}^{\boldsymbol{p}}})=0 .
\end{equation}
The rational function $ \phi = y_{\ell}^d / \boldsymbol{x}^{\boldsymbol{p}} $ is called the $ l $\emph{-th contact rational function} and $ d = d(\ell;\mathcal{A})$ is the \emph{$ l $-ramification index}.
\begin{rem}
	Let $ \phi $ be the $ l $-th contact rational function and perform a blow-up $ \mathcal{A} \rightarrow \mathcal{B} $. If $ \pi $ is combinatorial then the new $ l $-th contact rational function is the strict transform of $ \phi $.
\end{rem}

Recall that there exists a unique constant $ \xi \in k^* $ such that $ \nu(\phi_j - \xi) > 0 $.
\begin{defi}
	A \emph{$ \ell $-Puiseux's package} is a finite sequence of blow-ups
	$$	\xymatrix{ \mathcal{A} \ar[r]^{\pi_0} & \mathcal{A}_1 \ar[r]^{\pi_1} & \cdots \ar[r]^{\pi_N} & \mathcal{A}' } $$
	where
	\begin{itemize}
		\item $ \pi_t $ is a combinatorial blow-up $ \pi_t=\theta_{ij}(\mathcal{A}_t) $ with $ 1\leq i<j\leq r $, or $ \pi_t=\theta_{i}^{\ell}(\mathcal{A}_t) $ with $ 1\leq i \leq r $, for $ 0 \leq t \leq N-1 $;
		\item $ \pi_N = \theta_{i}^{l}(\mathcal{A}_N) $ is a blow-up with translation.
	\end{itemize}	
	In the special case in which the ramification index is $ d(\ell;\mathcal{A}_t) = 1 $ and $ \nu(y_{t,\ell}) < \nu(x_{t,i}) $, where $ \mathcal{A}_t = (\mathcal{O}_t,(\boldsymbol{x}_t,\boldsymbol{y}_t)) $, we require the combinatorial blow-up $ \pi_t $ not to be $ \theta_{i}^{\ell}(\mathcal{A}_t) $.
\end{defi}
The last condition in the definition is not necessary. We put it because it makes easier some calculations which will appear later in the text. 
\begin{rem}
Let $ \pi_N = \theta_{i}^{\ell}(\mathcal{A}_N) : \mathcal{A}_N \rightarrow \mathcal{A}' $ be the last blow-up of a $ \ell $-Puiseux's package. The $ l $-th contact rational function in $ \mathcal{A}_N $ has to be necessarily $ \phi_{\ell} = y_{\ell} / x_{i} $ and then after the transformation we obtain $ y'_{\ell} = \phi_{\ell} - \xi $.
\end{rem}
Note that a $ \ell $-Puiseux's package gives a local uniformization of the hypersurface $ \boldsymbol{x^q}y_{\ell}^d - \xi \boldsymbol{x^t} = 0 $. We will prove the existence of Puiseux's packages in Proposition \ref{pr:existence_puiseux_packages}.

\subparagraph{{\bf Equations of a Puiseux's package}}
We collect here some specific calculations about Puiseux's packages for further references. Let $  \mathcal{A}' = (\mathcal{O}',(\boldsymbol{x}',\boldsymbol{y}') \big) $ be the parameterized regular local model obtained from $ \mathcal{A} = \big( \mathcal{O} ; (\boldsymbol{x},\boldsymbol{y}) \big) $ by means of a $ \ell $-Puiseux's package $ \pi : \mathcal{A} \rightarrow \mathcal{A}' $. We have
\begin{eqnarray}\label{eq:variables_after_Puiseux_package}
	y_{\ell}' & = & \phi_{\ell} - \xi ,\\ \nonumber
	x_i & = & \boldsymbol{x}'^{\boldsymbol{\alpha}_i} (y_{\ell}' + \xi)^{\beta_i} ,\\ \nonumber
	y_{\ell} & = & \boldsymbol{x}'^{\boldsymbol{\alpha}_0} (y_{\ell}' + \xi)^{\beta_0} ,\\ \nonumber
	y_j & = & y_j' \text{ if } j \neq \ell ,
\end{eqnarray}
where $\boldsymbol{\alpha_0},\boldsymbol{\alpha_i} \in \mathbb{Z}_{\geq 0}^r$ and $\beta_0,\beta_i \in \mathbb{Z}_{\geq 0}$. The relation \eqref{eq:valorative_dependence_relation} gives
\begin{equation}\label{eq:dependence_alpha_p}
	p_1  \boldsymbol{\alpha_1} + \cdots + p_r  \boldsymbol{\alpha_r} - d \boldsymbol{\alpha_0} = 0,
\end{equation}
\begin{equation}\label{eq:dependence_beta_p}
	p_1  \beta_1 + \cdots + p_r \beta_r + 1 =  d \beta_0 .
\end{equation}
It follows from the construction that the matrices $ \check{H}_{\pi} $ and $ H_{\pi} $ defined by
\begin{equation}\label{eq:Puiseux_package_matrix_1}
	\text{\Large $\check{H}_{\pi}$} := \left(
	\begin{array}{ccc}
		\alpha_{11} & \cdots & \alpha_{1r}  \\
		\vdots & \ddots & \vdots  \\
		\alpha_{r1} & \cdots & \alpha_{rr}
	\end{array}
	\right),
\end{equation}
and
\begin{equation}\label{eq:Puiseux_package_matrix_2}
	\text{\Large $H_{\pi}$} := \left(
	\begin{array}{ccc|c}
		&  &  & \beta_1 \\
		& \text{\Large $\check{H}$} &  & \vdots \\
		&  &  & \beta_r \\
		\hline
		\alpha_{01} & \cdots & \alpha_{0r} & \beta_0
	\end{array}
	\right)
\end{equation}
are invertible with non-negative integers coefficients. Using matrix notation Equality \eqref{eq:dependence_alpha_p} can be written as
\begin{equation}\label{eq:dependence_alpha_p_1}
	\boldsymbol{p} \text{\Large $\check{H}_{\pi}$} = -d \boldsymbol{\alpha}_0 \ .
\end{equation}
\noindent Taking differentials in the expressions \eqref{eq:variables_after_Puiseux_package} we obtain
\begin{equation}\label{eq:differentials_after_Puiseux_package}
	\left(
	\begin{array}{c}
	\frac{dx_1}{x_1} \\
	\vdots \\
	\frac{dx_r}{x_r} \\
	\frac{dy_{\ell}}{y_{\ell}}
	\end{array}
	\right)
	= \text{\Large $ H $}
	\left(
	\begin{array}{c}
	\frac{dx'_1}{x'_1} \\
	\vdots \\
	\frac{dx'_r}{x'_r} \\
	\frac{d\phi_{\ell}}{\phi_{\ell}}
	\end{array}
	\right)
\end{equation}
where $ \frac{d\phi_{\ell}}{\phi_{\ell}} = y_{\ell}' \phi_{\ell}^{-1} \frac{dy_{\ell}'}{y_{\ell}'} $. The equality \eqref{eq:differentials_after_Puiseux_package} provides explicit equations for the $ \mathcal{O} $-homomorphism between the modules of differentials
\begin{eqnarray*}
	\Omega_{\mathcal{O}/k}(\log \boldsymbol{x}) & \longrightarrow & \Omega_{\mathcal{O}'/k}(\log \boldsymbol{x}' ) \\
	\frac{dx_i}{x_i} & \longmapsto & \sum_{k=1}^r \alpha_{ik}\frac{dx'_k}{x'_k} + \phi_{\ell}^{-1} \beta_i dy'_{\ell} \ , \quad i=1,\dots,r \ ; \\
	dy_j & \longmapsto & d y'_j \ , \quad 1 \leq j \leq n-r \ , \ j \neq \ell \ ; \\
	dy_{\ell} & \longmapsto & \boldsymbol{x}'^{\boldsymbol{\alpha}_0} \phi_{\ell}^{\beta_0} \left( \sum_{k=1}^r \alpha_{0k}\frac{dx'_k}{x'_k} + \phi_{\ell}^{-1} \beta_0 dy'_{\ell} \right) \ .
\end{eqnarray*}.
\begin{rem}\label{re:monomial_after_Puiseux_package}
	Let $ \boldsymbol{x}^{\boldsymbol{q}_0}y_{\ell}^{e_0} $ be a monomial with integer exponents and let $ \gamma_0 \in \Gamma $ be its value. We have that all the monomials in the variables $ \boldsymbol{x} $ and $ y_{\ell} $ with value $ \gamma_0 $ are those of the form
	$$ \boldsymbol{x}^{\boldsymbol{q}_0}y_{\ell}^{e_0} \phi_{\ell}^t = \boldsymbol{x}^{\boldsymbol{q}_0 - t \boldsymbol{p} } y_{\ell}^{e_0 + td} .$$
	After a $ \ell $-Puiseux's package such a monomial becomes a polynomial with the same value
	$$  \boldsymbol{x}^{\boldsymbol{q}_0}y_{\ell}^{e_0} \phi_{\ell}^t = \boldsymbol{x}'^{\boldsymbol{q}'_0} (y_{\ell}' + \xi)^t $$
	where the exponent $ \boldsymbol{q}'_0 $ is determined by \eqref{eq:variables_after_Puiseux_package} and it does not depend on the parameter $ t $.
\end{rem}

\subparagraph{{ \bf Puiseux's packages without ramification}}
	One notable case is when $ d=1 $ and $ \boldsymbol{p} \in \mathbb{Z}_{\geq 0}^r $. In this situation we always can determine a $ \ell $-Puiseux's package
	\begin{displaymath}
		\xymatrix{
		\mathcal{A}_0 \ar[r]^{\theta_{i_0 \ell}} & \mathcal{A}_1 \ar[r]^{\theta_{i_1 \ell}} & \cdots \ar[r]^{\theta_{i_{N-1} \ell}} & \mathcal{A}_N }
	\end{displaymath}
	such that $ \check{H}_{\pi} = I_r $ and
	$$
	\text{\LARGE $H_{\pi}$} = \left(
	\begin{array}{ccc|c}
	&  &  & 0 \\
	& \text{\LARGE $I_r$} &  & \vdots \\
	&  &  & 0 \\
	\hline
	p_1 & \cdots & p_r & 1
	\end{array}
	\right) .
	$$
	If we have $ d=1 $ but $ \boldsymbol{p} \notin \mathbb{Z}_{\geq 0}^r $, as we will see in Lemma \ref{le:simple_finite_list}, we can determine a finite composition of blow-ups with centers of the kind $ I_{ij} $ to reach the previous case. If $ C $ is the $ r\times r $ matrix related to that transformation, and $ \boldsymbol{p}' \in \mathbb{Z}_{\geq 0}^r $ is the new exponent (in fact $ \boldsymbol{p}' = \boldsymbol{p} C $), we have that the matrix of the complete Puiseux's package is
	$$
	\text{\LARGE $H_{\pi}$} = \left(
	\begin{array}{ccc|c}
		&  &  & 0 \\
		& \text{\LARGE $C$} &  & \vdots \\
		&  &  & 0 \\
		\hline
		0 & \cdots & 0 & 1
	\end{array}
	\right) \times \left(
	\begin{array}{ccc|c}
			&  &  & 0 \\
			& \text{\LARGE $I_r$} &  & \vdots \\
			&  &  & 0 \\
			\hline
			p'_1 & \cdots & p'_r & 1
	\end{array}
	\right) = \left(
	\begin{array}{ccc|c}
			&  &  & 0 \\
			& \text{\LARGE $C$} &  & \vdots \\
			&  &  & 0 \\
			\hline
			p'_1 & \cdots & p'_r & 1
	\end{array}
	\right)
	$$
	\begin{rem}
		In the case $ d=1 $ and $ \boldsymbol{p} \in \mathbb{Z}_{\geq 0}^r $ there is a coordinate change related to the Puiseux's package. Let $ T_0: \mathcal{A}_0 \rightarrow \tilde{\mathcal{A}}_0 $ be an ordered change of the $ \ell $-th variable given by $ \tilde{y}_{\ell} := y_{\ell} - \xi \boldsymbol{x^p} $. We have the following commutative diagram		
		\begin{displaymath}
			\xymatrix{
			\mathcal{A}_0 \ar[r]^{\theta_{i_0 \ell}} \ar[d]^{T_0} & \mathcal{A}_1 \ar[r]^{\theta_{i_1 \ell}} \ar[d]^{T_1} & \cdots \ar[r]^{\theta_{i_{N-2} \ell}} & \mathcal{A}_{N-1} \ar[r]^{\theta_{i_{N-1} \ell}} \ar[d]^{T_{N-1}} & \mathcal{A}_N \ar[d]^{T_N=\operatorname{id}} \\
			\tilde{\mathcal{A}}_0 \ar[r]^{\theta_{i_0 \ell}}  & \tilde{\mathcal{A}}_1 \ar[r]^{\theta_{i_1 \ell}}  & \cdots \ar[r]^{\theta_{i_{N-2} \ell}} & \tilde{\mathcal{A}}_{N-1} \ar[r]^{\theta_{i_{N-1} \ell}} & \tilde{\mathcal{A}}_N }
		\end{displaymath}
		where the upper horizontal row is a $ \ell $-th Puiseux's package, the vertical arrows $ T_s : \mathcal{A}_s \rightarrow \tilde{\mathcal{A}}_s $ for $ s=0,1,\dots,N-1 $ are ordered changes of the $ \ell $-th coordinate and the last vertical arrow is the identity map. Moreover, all the horizontal arrows are combinatorial blow-ups except $ \theta_{i_{N-1} \ell}:\mathcal{A}_{N-1} \rightarrow \mathcal{A}_{N} $ which is a blow-up with translation. 
	\end{rem}

\subsubsection{Nested transformations}
The transformations of parameterized regular models that we have introduced \textit{respect} the relative ordering in the dependent variables. This fact is a key feature of our induction treatment.
\begin{defi}
	A \textit{$ 0 $-nested transformation} is a composition of transformations of parameterized regular local models
	\begin{displaymath}
		\xymatrix{ \mathcal{A} = \mathcal{A}_0 \ar[r]^{\tau_0} & \mathcal{A}_1 \ar[r]^{\tau_1} & \cdots \ar[r]^{\tau_{t-1}} & \mathcal{A}_t = \mathcal{A}' }
	\end{displaymath}
	where each
	\begin{displaymath}
		\xymatrix{ \mathcal{A}_k \ar[r]^{\tau_k} & \mathcal{A}_{k+1} }
	\end{displaymath}
	is a combinatorial blow-up $ \tau_k = \theta_{i_k j_k}(\mathcal{A}_k) $ with $ 1 \leq i_k < j_k \leq r $.

	A \textit{$ \ell $-nested transformation} is a composition of transformations of parameterized regular local models
	\begin{displaymath}
			\xymatrix{ \mathcal{A} = \mathcal{A}_0 \ar[r]^{\tau_0} & \mathcal{A}_1 \ar[r]^{\tau_1} & \cdots \ar[r]^{\tau_{t-1}} & \mathcal{A}_t = \mathcal{A}' }
	\end{displaymath}
	where each
	\begin{displaymath}
		\xymatrix{ \mathcal{A}_k \ar[r]^{\tau_k} & \mathcal{A}_{k+1} }
	\end{displaymath}
	is a $ (l-1) $-nested transformation, a $ \ell $-Puiseux's package or an ordered change of the $ l $-th coordinate.
\end{defi}
\begin{rem}
	Note that a $ \ell $-nested transformations is also a $ l' $-nested transformation for $ l' > l $, and in particular a $ (n-r) $-nested transformation. We will refer to $ (n-r) $-nested transformations simply by \emph{nested transformations}.
\end{rem}

\subsection{Statements in terms of parameterized regular local models}
Let $ K/k $ be a finitely generated field extension and let $ \nu $ be a $ k $-rational valuation. Let $ \mathcal{A} = \big( \mathcal{O} ; (\boldsymbol{x},\boldsymbol{y}) \big) $ be a parameterized regular local model of $ K,\nu $. Consider a codimension one rational foliation $ \mathcal{F} $ of $ K/k $. Denote by $ \mathcal{F}_{\mathcal{A}} $ the submodule of $ \Omega_{\mathcal{O}/k}(\log \boldsymbol{x}) $ given by
$$
	\mathcal{F}_{\mathcal{A}} := \mathcal{F} \cap  \Omega_{\mathcal{O}/k}(\log \boldsymbol{x}) \ .
$$
\begin{defi}
	The codimension one rational foliation $ \mathcal{F} $ is \emph{$ \mathcal{A} $-final} if $ \mathcal{F}_{\mathcal{A}} $ is $ \boldsymbol{x} $-log-final,
\end{defi}
\begin{theo}\label{th:2}
	Let $ k $ be a field of characteristic zero and let $ K/k $ be a finitely generated field extension. Let $ \mathcal{F} $ be a rational codimension one foliation of $ K/k $. Given a $ k $-rational archimedean valuation $ \nu $ of $ K/k $ and a parameterized regular local model $ \mathcal{A} $ of $K/k, \nu$, there is a nested transformation
	$$
		{\mathcal A} \longrightarrow {\mathcal B}
	$$
	such that ${\mathcal F}$ is ${\mathcal B}$-final.  
\end{theo}
This statement is a refinement of Theorem \ref{th:II}, in which we specify the kind of transformations we allow. Thanks to Proposition \ref{pr:prml_existence} we have that Theorem \ref{th:2} implies Theorem \ref{th:1}.
 
\subsection{Formal completion}
Although the result we want to show is about ``convergent'' foliations, in order to prove it we have to consider formal functions and $ 1 $-forms with formal coefficients.

Let $ \mathcal{A} = \big( \mathcal{O} ; (\boldsymbol{x},\boldsymbol{y}) \big)$ be a parameterized regular local model for $K/k,\nu$. Let $ \hat{\mathcal{O}} $ be the $ \mathfrak{m} $-adic completion of $ \mathcal{O} $. By Cohen's Structure Theorem we know that
$$ \hat{\mathcal{O}} \cong \kappa[[\boldsymbol{x},\boldsymbol{y}]] $$
where $ \kappa := \mathcal{O} / \mathfrak{m} $ is the residue field of the center of the valuation. Since we are dealing with $ k $-rational valuations we have that $ \kappa = k $ so in our case
$$ \hat{\mathcal{O}} = k[[\boldsymbol{x},\boldsymbol{y}]] \ . $$

Let $ R_{\mathcal{A}}^{\ell} $ be the subrings of $ \hat{\mathcal{O}} $ defined by
$$ R_{\mathcal{A}}^0:= \kappa[[\boldsymbol{x}]] \quad \text{and} \quad R_{\mathcal{A}}^{\ell} := \kappa[[\boldsymbol{x},y_1,y_2,\dots,y_{\ell}]] $$
for $ 1 \leq \ell \leq n-r $. All of them are local rings with maximal ideal $ R_{\mathcal{A}}^{\ell} \cap \hat{\mathfrak{m}} $. We have
$$ R_{\mathcal{A}}^0 \subset R_{\mathcal{A}}^{1} \subset \cdots \subset R_{\mathcal{A}}^{n-r} = \hat{\mathcal{O}}, $$
where each inclusion is in fact a relation of domination of local rings.

Consider now a $ \ell $-nested transformation $ \pi : \mathcal{A} \longrightarrow \mathcal{A}' $. Tensoring the inclusion of local rings $ \pi : \mathcal{O} \rightarrow \mathcal{O}' $ by $ \hat{\mathcal{O}} $ we obtain an inclusion of complete local rings
$$
\pi : \hat{\mathcal{O}} \hookrightarrow \hat{\mathcal{O}}'
$$
which we denote with the same symbol $ \pi $. Such an inclusion is compatible with the decomposition in subrings of $ \hat{\mathcal{O}} $ in the following way:
$$
\pi(R_{\mathcal{A}}^{j}) \subset R_{\mathcal{A}'}^{\ell} \quad \text{for} \quad 0 \leq j \leq \ell \ ,
$$
and
$$
\pi(R_{\mathcal{A}}^{k}) \subset R_{\mathcal{A}'}^{k} \quad \text{for} \quad \ell + 1 \leq k \leq n-r \ .
$$
In fact, we have that
$$
\pi_{|R_{\mathcal{A}}^{\ell}} : R_{\mathcal{A}}^{\ell} \rightarrow R_{\mathcal{A}'}^{\ell}
$$
is an injective $ R_{\mathcal{A}}^{\ell} $-homomorphism.

We develop now a similar construction for the modules of differentials. Let $ \hat{\Omega}_{\mathcal{O}/k} $ be the free $ \hat{\mathcal{O}} $-module generated by
$$
\{dx_1,dx_2,\dots,dx_r,dy_1,dy_2,\dots,dy_{n-r}\} \ .
$$
We have that
$$
\hat{\Omega}_{\mathcal{O}/k} \simeq \Omega_{\mathcal{O}/k} \otimes_{\mathcal{O}} \hat{\mathcal{O}} \ .
$$
Let $ \hat{\Omega}_{\mathcal{O}/k}(\log \boldsymbol{x}) $ be the free $ \hat{\mathcal{O}} $-module generated by
$$
\{\frac{dx_1}{x_1},\frac{dx_2}{x_2},\dots,\frac{dx_r}{x_r},dy_1,dy_2,\dots,dy_{n-r}\} \ .
$$
We have that
$$
\hat{\Omega}_{\mathcal{O}/k}(\log \boldsymbol{x}) \simeq \Omega_{\mathcal{O}/k}(\log \boldsymbol{x}) \otimes_{\mathcal{O}} \hat{\mathcal{O}} \ .
$$
Tensoring the inclusion homomorphism \eqref{eq:inclusion_modules_differentials} by $ \hat{\mathcal{O}} $ we obtain an injective $ \hat{\mathcal{O}} $-module homomorphism
\begin{eqnarray*}
	\hat{\Omega}_{\mathcal{O}/k} & \rightarrow & \hat{\Omega}_{\mathcal{O}/k}(\log \boldsymbol{x}) \\
	\omega = \sum_{i=1}^n a_i dz_i  & \mapsto & \sum_{i=1}^r x_i a_i \frac{dx_i}{x_i} + \sum_{j=r+1}^n a_i dy_j .
\end{eqnarray*}
For each index index $ \ell $, $ 1 \leq \ell \leq n-r $, let $ N_{\mathcal{A}}^{\ell} $ be the free $ R_{\mathcal{A}}^{\ell} $-module generated by
$$
\{\frac{dx_1}{x_1},\frac{dx_2}{x_2},\dots,\frac{dx_r}{x_r},dy_1,dy_2,\dots,dy_{\ell}\} \ .
$$
Note that all these modules are subsets of $ \hat{\Omega}_{\mathcal{O}/k}(\log \boldsymbol{x}) $, and we have
$$
N_{\mathcal{A}}^0 \subset N_{\mathcal{A}}^1 \subset \cdots \subset N_{\mathcal{A}}^{n-r} \cong \hat{\Omega}_{\mathcal{O}/k}(\log \boldsymbol{x}) \ ,
$$
where each inclusion is just an inclusion of subsets (not a module monomorphism).

Consider now a $ \ell $-nested transformation $ \pi : \mathcal{A} \longrightarrow \mathcal{A}' $. The inclusion of $ \hat{\mathcal{O}} $-modules $ \pi : \hat{\Omega}_{\mathcal{O}/k}(\log \boldsymbol{x}) \hookrightarrow \hat{\Omega}_{\mathcal{O}'/k}(\log \boldsymbol{x}') $ satisfies
$$
\pi(N_{\mathcal{A}}^{j}) \subset N_{\mathcal{A}'}^{\ell} \quad \text{for} \quad 0 \leq j \leq \ell \ ,
$$
and
$$
\pi(N_{\mathcal{A}}^{k}) \subset N_{\mathcal{A}'}^{k} \quad \text{for} \quad \ell + 1 \leq k \leq n-r \ .
$$
In fact, we have that
$$
{\pi}_{|N_{\mathcal{A}}^{\ell}} : N_{\mathcal{A}}^{\ell} \rightarrow N_{\mathcal{A}'}^{\ell}
$$
is a $ R_{\mathcal{A}}^{\ell} $-module monomorphism.

\section{Maximal rational rank: the combinatorial case}\label{ch:max_rat_rank}
In this chapter we treat the case of valuations of maximal rational rank $ r = \operatorname{rat.rk}(\nu) = \operatorname{tr.deg}(K/k) $. Let $ \mathcal{A} = (\mathcal{O},\boldsymbol{x}) $ be a parameterized regular local model for $ K,\nu $. Recall that in this case we have
$$
\hat{\mathcal{O}} \simeq R_{\mathcal{A}}^0 \quad \text{and} \quad \hat{\Omega}_{\mathcal{O}/k}(\log \boldsymbol{x}) \otimes_{\mathcal{O}} \hat{\mathcal{O}} \simeq N_{\mathcal{A}}^0 \ .
$$
The transformations of parameterized regular local models allowed in this case are the $ 0 $-nested transformations. Such a transformation is a finite composition
\begin{displaymath}
	\xymatrix{ \mathcal{A} = \mathcal{A}_0 \ar[r]^{\tau_0} & \mathcal{A}_1 \ar[r]^{\tau_1} & \cdots \ar[r]^{\tau_{t-1}} & \mathcal{A}_t = \mathcal{A}' }
\end{displaymath}
where each
\begin{displaymath}
	\xymatrix{ \mathcal{A}_k \ar[r]^{\tau_k} & \mathcal{A}_{k+1} }
\end{displaymath}
is a combinatorial blow-up $ \tau_k = \theta_{i_k j_k} $ with $ 1 \leq i_k < j_k \leq r $. The inclusion of local rings is given by
\begin{eqnarray}\label{eq:0_nested_transformations_differentials}
	R_{\mathcal{A}}^{0} & \longrightarrow  & R_{\mathcal{A}'}^{0} \\ \nonumber
	x_i & \longmapsto & {\boldsymbol{x}'}^{\boldsymbol{c}_i} = {x'}_1^{c_{i1}} \cdots {x'}_r^{c_{ir}} \ ,
\end{eqnarray}	
where the matrix $ C := (c_{ij}) $ is an invertible matrix of non-negative integers with determinant $ 1 $. Note that this homomorphism preserves the value of the monomials, it means
$$
\nu(\boldsymbol{x^a}) = \nu( \boldsymbol{x}'^{\boldsymbol{a}'}) \ ,
$$
where $ \boldsymbol{a}' = \boldsymbol{a} C $. The divisibility relation is also maintained:
$$
\boldsymbol{x^a} | \boldsymbol{x^b} \Rightarrow \boldsymbol{x}'^{\boldsymbol{a}'} | \boldsymbol{x}'^{\boldsymbol{b}'} \ .
$$
Taking differentials in \eqref{eq:0_nested_transformations_differentials} we obtain
$$ \frac{dx_i}{x_i} = c_{i1} \frac{dx'_1}{x'_1} + \cdots + c_{ir} \frac{dx'_r}{x'_r} $$
for $ 1\leq i \leq r $. Therefore we have that the $ R_{\mathcal{A}}^{0} $-homomorphism between the modules of differentials is given by
\begin{eqnarray*}
	N_{\mathcal{A}}^{0} & \longrightarrow  & N_{\mathcal{A}'}^{0} \\
	\sum a_i \frac{dx_i}{x_i} & \longmapsto & \sum a'_i \frac{dx'_i}{x'_i} \ ,
\end{eqnarray*}
where
$$
a'_i = c_{i1} a_1 + \cdots + c_{ir} a_r 
$$
for $ 1\leq i \leq r $. It is in fact a monomorphism since the matrix $ C $ is invertible.

Before prove Theorem \ref{th:2} in this case, we present two useful lemmas which we will use frequently. Let $ \mathcal{L}=\{ F_i \ |\ i \in I \} \subset R_{\mathcal{A}}^0 $ be a list of elements of $ R_{\mathcal{A}}^0 $. We say the list $ \mathcal{L} $ is \emph{simple in $ \mathcal{A} $} if for all pair of elements $ F_i,F_j \in \mathcal{L} $ we have that $ \nu(F_i) \leq \nu(F_j) $ if and only if $ F_i $ divides $ F_j $.
\begin{lem}\label{le:simple_finite_monomial_list}
	Let $ \mathcal{L}\subset R_{\mathcal{A}}^0 $ be a finite list of monomials. There is a $ 0 $-nested transformation $ \mathcal{A} \rightarrow \mathcal{A}' $ such that $ \mathcal{L} $ is simple in $ \mathcal{A}' $.
	\begin{proof}
		Since the divisibility relation remains stable after performing a combinatorial blow-up in the independent variables, it is enough to prove the statement for lists with just two monomials. Consider $ \mathcal{L} = \{ \boldsymbol{x}^{\boldsymbol{a}},\boldsymbol{x}^{\boldsymbol{b}} \} $ and put $ \boldsymbol{c} = \boldsymbol{a} - \boldsymbol{b} $. We use the following invariants
		$$ M:= \max \big\{ 0,c_1,\dots,c_r \big\} \ , \quad m:= \min \big\{ 0,c_1,\dots,c_r \big\}, $$
		$$ T:= \# \big\{ i \ :\ c_i = M \big\} \ , \quad t:= \# \big\{ i \ :\ c_i = m \big\} \ , \quad \delta:= \big(-Mm,T + t \big) . $$
		If the first coordinate of $ \delta $ is $ 0 $ the list $ \mathcal{L} $ is simple. If it is not the case perform a blow-up $ \theta_{ij} $ such that $ c_i = m $ and $ c_j = M $. Without loss of generality we can suppose $ \nu(x_i) < \nu(x_j) $. At the center of $ \nu $ in the new variety obtained we have a system of coordinates $ \boldsymbol{x}' $ which only differs from $ \boldsymbol{x} $ in the $ j $-th variable, which is $ x_j' =x_j /x_i $. The exponents $ \boldsymbol{a} $ and $ \boldsymbol{b} $ becomes $ \boldsymbol{a}' $ and $ \boldsymbol{b}' $ respectively, where only the $ i $-th coordinate is modified. They are $ a'_i = a_i + a_j $ and $ b'_i = b_i + b_j $. The same thing happens with $ \boldsymbol{c}' $. We have $ m=c_i<0<c_j=M $, thus $ m < c'_i = c_i + c_j < M $. Then $ \delta' <_{lex} \delta $, so iterating we get the desired result.
	\end{proof}
\end{lem}
For lists with infinitely many monomials we have the following statement:
\begin{lem}\label{le:simple_infinite_monomial_list}
	Let $ \mathcal{L} \subset R_{\mathcal{A}}^0 $ be an infinite list of monomials. There is a $ 0 $-nested transformation $ \mathcal{A} \rightarrow \mathcal{A}' $ such that in $ \mathcal{A}' $ every monomial of $ \mathcal{L} $ is divided by the monomial with lowest value.
	\begin{proof}
		Consider the ideal $ I \subset \mathcal{O} $ generated by the elements of $ \mathcal{L} $. Since $ \mathcal{O} $ is Noetherian, $ I $ is finitely generated. It is enough to apply Lemma \ref{le:simple_finite_list} to $ \mathcal{L}' $, where $ \mathcal{L}' $ is the list formed by a finite system of generators of $ I $.
	\end{proof}
\end{lem}
The Local Uniformization of formal functions in the maximal rational case is a corollary of Lemma \ref{le:simple_infinite_monomial_list}. Let us see in detail. Take a formal function
$$ F = \sum_{\boldsymbol{a}\in \mathbb{Z}_{\geq 0}^r} F_{\boldsymbol{a}} \boldsymbol{x^a} \in R_{\mathcal{A}}^0 $$
and let $ \mathcal{L}_{F} $ be the list formed by its monomials
$$ \mathcal{L}_{F} := \big\{ \boldsymbol{x^a} \ |\ F_{\boldsymbol{a}} \neq 0 \big\} \ . $$
Applying Lemma \ref{le:simple_infinite_list} to $ \mathcal{L}_{F} $ we obtain a new parametrized regular local model $ \mathcal{B} $ in which $ F $ has the form
$$ F = \boldsymbol{x}'^{\boldsymbol{t}} U \in R_{\mathcal{B}}^0 \ , $$
where $ U \in R_{\mathcal{B}}^0 $ is a unit. We say that we have \emph{monomialized} the formal function $ F $.

Now we can improve the statements of Lemmas \ref{le:simple_finite_monomial_list} and \ref{le:simple_infinite_monomial_list}.
\begin{lem}\label{le:simple_finite_list}
	Let $ \mathcal{L}\subset R_{\mathcal{A}}^0 $ be a finite list. There is a $ 0 $-nested transformation $ \mathcal{A} \rightarrow \mathcal{A}' $ such that $ \mathcal{L} $ is simple in $ \mathcal{A}' $.
	\begin{proof}
		First, we monomialize each element of the list using Lemma \ref{le:simple_infinite_monomial_list}, so each element of the list becomes a monomial times a unit. Then we apply Lemma \ref{le:simple_finite_monomial_list} to the list of such monomials.
	\end{proof}
\end{lem}
\begin{lem}\label{le:simple_infinite_list}
	Let $ \mathcal{L} \subset R_{\mathcal{A}}^0 $ be an infinite list. There is a $ 0 $-nested transformation $ \mathcal{A} \rightarrow \mathcal{A}' $ such that in $ \mathcal{A}' $ every monomial of $ \mathcal{L} $ is divided by the monomial with lowest value.
	\begin{proof}
		We just need to apply Lemma \ref{le:simple_infinite_monomial_list} to the list formed by all the monomials appearing in the elements of $ \mathcal{L} $.
	\end{proof}
\end{lem}
Using these lemmas we are able to prove Theorem \ref{th:2} in the maximal rational case.
	
Let $ \mathcal{F} $ be a rational codimension one foliation of $ K/k $ and take a $ 1 $-form
$$ \omega = \sum a_i \frac{dx_i}{x_i} \in F_{\mathcal{A}} \subset N_{\mathcal{A}}^{0} \ . $$
Consider the list
$$ \mathcal{L}_{\omega,\mathcal{A}}:= \big\{ a_1,a_2,\dots,a_r \big\} \subset R_{\mathcal{A}}^0 \ . $$
Using Lemma \ref{le:simple_finite_list} we can determine a $ 0 $-nested transformation $ \mathcal{A} \rightarrow \mathcal{B} $ such that $ \mathcal{L}_{\omega,\mathcal{A}} $ is simple in $ \mathcal{B} $. As we have just explained, in $ \mathcal{B} $ we have
$$
	\omega = \sum a'_i \frac{dx'_i}{x'_i} \in \mathcal{F}_{\mathcal{B}} \subset N_{\mathcal{B}}^{0}
$$
where the coefficients $ a'_i \in R_{\mathcal{B}}^{0} $ are an invertible linear combination of the coefficients $ a_i \in R_{\mathcal{A}}^{0} $. Therefore the list
$$ \mathcal{L}_{\omega,\mathcal{B}} = \big\{ a'_1,a'_2,\dots,a'_r \big\} $$
is also simple, so we can factorize the coefficient with lower value and obtain an expression
$$
	\omega = \boldsymbol{x}'^{\boldsymbol{t}} \sum_{i=1}^{r} \tilde{a}_{i} \frac{dx'_i}{x'_i} \in \mathcal{F}_{\mathcal{B}} \subset N_{\mathcal{B}}^{0}
$$
where at least one of the coefficients $ \tilde{a}_i $ is a unit. We have that the $ 1 $-form
$$
\tilde{\omega} := \frac{1}{\boldsymbol{x}'^{\boldsymbol{t}}} \omega = \sum_{i=1}^{r} \tilde{a}_{i} \frac{dx'_i}{x'_i}
$$
belongs to $ \mathcal{F}_{\mathcal{B}} $ and it is $ \boldsymbol{x}' $-log-elementary, hence $ \mathcal{F} $ is $ \mathcal{B} $-final.
\begin{rem}
	Note that we have not use neither the integrability condition nor the algebraic nature of the coefficients. It means that in the maximal rational case we have proved a more general result:
	\begin{quote}
		``Given a $ 1 $-form $ \omega \in N_{\mathcal{A}}^0 $ there is a $ 0 $-nested transformation $ \mathcal{A} \rightarrow \mathcal{B} $ such that in $ \mathcal{B} $ we have $ \omega =  \boldsymbol{x}'^{\boldsymbol{t}}\tilde{\omega} $ being $ \tilde{\omega} $ log-elementary.''
	\end{quote}
\end{rem}
\begin{rem}
	Note that the results in this section can be used in the case $ r = \operatorname{rat.rk}(\nu) < \operatorname{tr.deg}(K/k) = n $ if we restrict ourselves to elements of $ R_{\mathcal{B}}^{0} \subsetneq R_{\mathcal{B}}^{n-r} $ and $ N_{\mathcal{B}}^{0} \subsetneq N_{\mathcal{B}}^{n-r} $.
\end{rem}

\subsection{Existence of Puiseux's packages}
Using the results of this chapter we are now able to prove the existence of Puiseux's packages.
\begin{prop}\label{pr:existence_puiseux_packages}
	Let $ \mathcal{A} $ be a parameterized regular local model of $ K,\nu $. Given a dependent variable $ y_{\ell} $ there are $ \ell $-Puiseux's packages.
	\begin{proof}
		Consider the $ l $-th contact rational function
		$$ \phi_{\ell} = \frac{\boldsymbol{x^q}y_{\ell}^d }{\boldsymbol{x^t}} .$$
		Applying Lemma \ref{le:simple_finite_list} to the list $ \{\boldsymbol{x^q},\boldsymbol{x^t}\} $, we can reduce to the case $ \boldsymbol{q}=0 $. We use as invariant $ \delta = (d,\sum t_i) $.
		
		Suppose $ \delta = (1,1) $ and let $ i_0 $ be the only index such that $ t_{i_0} \neq 0 $. We have that $ \nu(y_{\ell}) = \nu(x_{i_0}) $ so the blow-up $ \theta_{i_0,s} $ is a blow up with translation hence
		$$ \theta_{i_0,s} : \mathcal{A}_0 \longrightarrow \mathcal{A}_1 $$		
		is a $ \ell $-Puiseux's package.
		
		Suppose that $ \delta = (1,M) $ with $ M>1 $ and let $ i_0 $ be an index such that $ t_{i_0} \neq 0 $. In this case $ \theta_{i_0,s} $ is combinatorial. In the new coordinates we have $ \delta = (1,M-1) $ so iterating we reach the previous situation.
		
		Finally suppose that $ d > 1 $. If there is an index $ i_0 $ such that $ t_{i_0} \neq 0 $ and $ \nu(y_{\ell}) < \nu(x_{i_0}) $ then after the combinatorial blow-up $ \theta_{i_0,s} $ the invariant becomes $ \delta=(d-p_{i_0},M) $. On the other hand, if $ \nu(y_{\ell}) > \nu(x_i) $ for all index $ i $ with $ t_i \neq 0 $ let $ i_0 $ be the first index such that $ d \leq \sum_{i=1}^{i_0} t_i $. Perform the blow-up with center $ (x_1,x_{i_0}) $, then the one with center $ (x_2,x_{i_0}) $ and continue until $ (x_{i_0 - 1},x_{i_0}) $ (exclude the blow-ups corresponding with independent variables with $ t_i = 0 $). Perform the blow-up with center $ (x_{i_0},y_{\ell}) $. After this sequence of combinatorial blow-ups the invariant is $ \delta = (d,M') $ with $ M'<M $ since the new exponent of the variable $ x_{i_0} $ is $ 0 < t'_{i_0} = \sum_{i=1}^{i_0} t_i - d < t_{i_0} $. We observe that in both cases the invariant $ \delta $ decreases for the lexicographic order, so iterating we reach the case $ d=1 $.	
	\end{proof}
\end{prop}

\section{Explicit value and truncated statements. Induction structure}
In this chapter introduce the notions of $ \gamma $-final formal functions and $ 1 $-forms and state local uniformization theorems in a value-truncated version. The local uniformization of foliations will be a consequence of this truncated version.

Let $ \mathcal{A} = \big( \mathcal{O} ; (\boldsymbol{x},\boldsymbol{y}) \big) $ be a parameterized regular local model for $ K , \nu $ and fix an integer $ \ell $, $ 0 \leq \ell \leq n-r$. 

\subsection{Explicit value and $ \gamma $-final $ 1 $-forms}
For each $ \alpha \in \Gamma_{\geq 0} $ denote by $ \mathcal{Q}_{\mathcal{A}}^{\ell}(\alpha) \subset R_{\mathcal{A}}^{\ell} $ the ideal generated by all the monomials $ \boldsymbol{x}^{\boldsymbol{r}} $ with $ \nu(\boldsymbol{x}^{\boldsymbol{r}}) \geq \alpha $.
\begin{rem}\label{re:value_ideals}
	Note that we have defined the ideals $ \mathcal{Q}_{\mathcal{A}}^{\ell}(\alpha) $ only for values $ \alpha \in \Gamma_{\geq 0} $ and not for any non-negative real number. So, given such an ideal $ \mathcal{Q}_{\mathcal{A}}^{\ell}(\alpha) $ there exists a unique integer exponent $ \boldsymbol{q}_{\alpha} \in \mathbb{Z}^r $ such that $ \nu(\boldsymbol{x}^{\boldsymbol{q}_{\alpha}}) = \alpha $. Therefore, there is at most one monomial in the ideal with value $ \alpha $.
	
	Except in the case $ \operatorname{rat.rk}(\nu)=1 $ these ideals are never principal. However, we always can determine a $ 0 $-nested transformation $ \pi : \mathcal{A} \rightarrow \mathcal{B} $ (Lemma \ref{le:simple_infinite_list}) such that $ \pi ( \mathcal{Q}_{\mathcal{A}}^{\ell}(\alpha) ) \subset R_{\mathcal{A}}^{\ell} $ is a principal ideal.
	
	Note also that these ideals ``grow'' by means of $ \ell $-nested transformations. Let $ \pi : \mathcal{A} \rightarrow \mathcal{B} $ be a $ \ell $-nested transformation. We have
	$$ \pi \left(  \mathcal{Q}_{\mathcal{A}}^{\ell}(\alpha) \right) \subset  \mathcal{Q}_{\mathcal{B}}^{\ell}(\alpha) \ . $$
	In fact, the ideal $ \mathcal{Q}_{\mathcal{A}}^{\ell}(\alpha) $ is the contraction of $ \mathcal{Q}_{\mathcal{B}}^{\ell}(\alpha) $ to $ R_{\mathcal{A}}^{\ell} $.
\end{rem}
Given a function $ F \in R_{\mathcal{A}}^{\ell} $ we say that
$$ \nu_{\mathcal{A}}(F) := \max \left\{ \lambda \in \Gamma \ |\ F \in \mathcal{Q}_{\mathcal{A}}^{\ell}(\lambda) \right\} $$
is \emph{the explicit value of $ F $}. We establish $ \nu_{\mathcal{A}}(0) := \infty $. Note that the explicit value is well defined. In fact, if we write $ F $ as a series in the independent variables
$$ F = \sum_{I \in \mathbb{Z}^{r}_{\geq0}} F_{I}(\boldsymbol{y}) \boldsymbol{x}^{I} \ , $$
we have that
$$ \max \left\{ \lambda \in \Gamma \ |\ F \in \mathcal{Q}_{\mathcal{A}}^{\ell}(\lambda) \right\} = 
 \min \left\{ \nu(\boldsymbol{x}^{I}) \ |\ F_{I}(\boldsymbol{y}) \neq 0 \right\} \ , $$
and the minimum in the right term is always attained since the set of values of monomials in the independent variables is a well ordered set.
 
\begin{defi}\label{de:gamma_final_functions}
	Given $ \gamma \in \Gamma $ a formal function $ F \in R_{\mathcal{A}}^{\ell} $ is \emph{$ \gamma $-final} if one of the following properties is satisfied:
	\begin{enumerate}
		\item $ \nu_{\mathcal{A}}(F) \leq \gamma $ and $ F $  can be written as
				$$ F = \boldsymbol{x}^{\boldsymbol{t}} \tilde{F} + \bar{F} $$
				where $ \tilde{F} $ is a unit and $ \nu_{\mathcal{A}}(\bar{F}) > \nu_{\mathcal{A}}(F) = \nu(\boldsymbol{x}^{\boldsymbol{t}}) $. In this case we say \emph{$ F $ is $ \gamma $-final dominant};
		\item $ \nu_{\mathcal{A}}(F) > \gamma $. In this case we say \emph{$ F $ is $ \gamma $-final recessive}.
	\end{enumerate}
\end{defi}
\begin{rem}
	Let $ F  $ be a $ \gamma $-final dominant formal function and consider a valuation $ \hat{\nu} $ defined on $ \hat{\mathcal{O}} $ extending $ \nu $. We have that $ \hat{\nu}(F) = \nu_{\mathcal{A}}(F) = \nu(\boldsymbol{x}^{\boldsymbol{t}}) $.
\end{rem}
While the value $ \nu(\phi) $ of a rational function $ \phi \in K $ is stable under birational transformations, the explicit value can increase by means of a $ \ell $-nested transformation:
\begin{lem}\label{le:gamma_final_functions_stable}
	Let $ F \in R_{\mathcal{A}}^{\ell} $ be a formal function and let $ \pi : \mathcal{A} \longrightarrow \mathcal{B} $ be a $ \ell $-nested transformation. We have that
	$$ \nu_{\mathcal{A}}(F) \leq \nu_{\mathcal{B}}(F) \ . $$
	In particular we have
	\begin{itemize}
		\item if $ F $ is $ \gamma $-final dominant in $ \mathcal{A} $, it is also $ \gamma $-final dominant in $ \mathcal{B} $ and $ \nu_{\mathcal{A}}(F) = \nu_{\mathcal{B}}(F) $.
		\item if $ F $ is $ \gamma $-final recessive in $ \mathcal{A} $, it is also $ \gamma $-final recessive in $ \mathcal{B} $.
	\end{itemize}
\end{lem}
In the module $ N_{\mathcal{A}}^{\ell} $ the \emph{log-elementary} forms are the equivalents of the units in $ R_{\mathcal{A}}^{\ell} $:
\begin{defi}\label{de:log-elementary_1_forms}
	A $ 1 $-form $ \omega \in N_{\mathcal{A}}^{\ell} $ is \emph{log-elementary} if
	$$ \omega = \sum_{i=1}^r a_i \frac{dx_i}{x_i} + \sum_{j=r+1}^n a_j dy_j $$
	with $ \operatorname{ord}(a_1,a_2,\dots,a_r) = 0 $.
\end{defi}
For each $ \alpha \in \Gamma $ denote by $ \mathcal{M}_{\mathcal{A}}^{\ell}(\alpha) \subset N_{\mathcal{A}}^{\ell} $ the submodule
$$
	\mathcal{M}_{\mathcal{A}}^{\ell}(\alpha) := \mathcal{Q}_{\mathcal{A}}^{\ell}(\alpha) \, N_{\mathcal{A}}^{\ell} \ .
$$
Given a $ 1 $-form $ \omega \in N_{\mathcal{A}}^{\ell} $ we say that
$$ \nu_{\mathcal{A}}(\omega) := \max \left\{ \lambda \ |\ \omega \in \mathcal{M}_{\mathcal{A}}^{\ell}(\lambda) \right\} $$
is \emph{the explicit value of $ \omega $}. We establish $ \nu_{\mathcal{A}}(0) := \infty $. As in the case of functions, writing 
$$ \omega = \sum_{I \in \mathbb{Z}^{r}_{\geq0}} \omega_{I}(\boldsymbol{y}) \boldsymbol{x}^{I} \ , $$
where each $ \omega_I(\boldsymbol{y}) $ is a $ 1 $-form whose coefficients are series in the dependent variables, we have that
$$ \max \left\{ \lambda \in \Gamma \ |\ \omega \in \mathcal{Q}_{\mathcal{A}}^{\ell}(\lambda) \right\} = 
 \min \left\{ \nu(\boldsymbol{x}^{I}) \ |\ \omega_{I}(\boldsymbol{y}) \neq 0 \right\} \ , $$
so the explicit value is well defined.

In the same way we define the explicit value for elements of $ \wedge^2 N_{\mathcal{A}}^{\ell} $ and $ \wedge^3 N_{\mathcal{A}}^{\ell} $, where $ \wedge^p $ denotes the $ p $-th exterior power.
\begin{defi}\label{de:gamma_final_1_forms}
	Given $ \gamma \in \Gamma $ a $ 1 $-form $ \omega \in N_{\mathcal{A}}^{\ell} $ is \emph{$ \gamma $-final} if one of the following properties is satisfied
	\begin{enumerate}
		\item $ \nu_{\mathcal{A}}(\omega) \leq \gamma $ and $ \omega$ can be written as
		$$ \omega = \boldsymbol{x}^{\boldsymbol{t}} \tilde{\omega} + \bar{\omega} $$
		where $ \tilde{\omega} $ is log-elementary and $ \nu_{\mathcal{A}}(\bar{\omega}) > \nu_{\mathcal{A}}(\omega) = \nu(\boldsymbol{x}^{\boldsymbol{t}}) $. In this case we say \emph{$ \omega $ is $ \gamma $-final dominant};
		\item $ \nu_{\mathcal{A}}(\omega) > \gamma $. In this case we say \emph{$ \omega $ is $ \gamma $-final recessive}.
	\end{enumerate}
\end{defi}
As in the case of functions, the explicit value of a $ 1 $-form can increase by means of a $ \ell $-nested transformation:
\begin{lem}\label{le:gamma_final_1_forms_stable}
	Let $ \omega \in N_{\mathcal{A}}^{\ell} $ be a formal $ 1 $-form and let $ \pi : \mathcal{A} \longrightarrow \mathcal{B} $ be a $ \ell $-nested transformation. We have that
	$$ \nu_{\mathcal{A}}(\omega) \leq \nu_{\mathcal{B}}(\omega) \ . $$
	In particular we have
	\begin{enumerate}
		\item if $ \omega $ is $ \gamma $-final dominant in $ \mathcal{A} $, it is also $ \gamma $-final dominant in $ \mathcal{B} $ and $ \nu_{\mathcal{A}}(\omega) = \nu_{\mathcal{B}}(\omega) $.
		\item if $ \omega $ is $ \gamma $-final recessive in $ \mathcal{A} $, it is also $ \gamma $-final recessive in $ \mathcal{B} $.
	\end{enumerate}
\end{lem}
This lemma is a consequence of the definitions. We do not detail the proof, by the key is the following remark:
\begin{rem}\label{re:nested_transformation_matrix}
	Let $ \omega \in N_{\mathcal{A}}^{\ell} $ be a log-elementary $ 1 $-form
	$$ \omega = \sum_{i=1}^r a_i \frac{dx_i}{x_i} + \sum_{j=1}^{\ell} b_j dy_j \ . $$
	For each index $ 1 \leq i \leq r $ denote $ \alpha_i = a_i(\boldsymbol{0}) \in k $. Consider a $ \ell $-nested transformation $ \pi:\mathcal{A} \longrightarrow \mathcal{B} $. Let $ a'_i,b'_j \in R_{\mathcal{B}}^{\ell} $ the coefficients of $ \omega $ in $ \mathcal{B} $ and denote $ \alpha'_i = a'_i(\boldsymbol{0}) $. We have that
	$$ (\alpha'_1,\alpha'_2,\dots,\alpha'_r) = (\alpha_1,\alpha_2,\dots,\alpha_r) \text{\large{$C$}}_{\pi} $$
	where $ C_{\pi} $ is a non-zero matrix of non-negative integers.
\end{rem}
Let
$$
	d : \mathcal{O} \longrightarrow \Omega_{\mathcal{O}/k}(\log \boldsymbol{x})
$$
be the map obtained by composition of the exterior derivative $ \mathcal{O} \rightarrow \Omega_{\mathcal{O}/k} $ with the inclusion $ \Omega_{\mathcal{O}/k} \rightarrow \Omega_{\mathcal{O}/k}(\log \boldsymbol{x}) $ given in \eqref{eq:inclusion_modules_differentials}. Tensoring by $ \hat{\mathcal{O}} $ we obtain
$$
d \otimes_{\mathcal{O}} 1 : \mathcal{O} \otimes_{\mathcal{O}} \hat{\mathcal{O}} \longrightarrow \Omega_{\mathcal{O}/k}(\log \boldsymbol{x}) \otimes_{\mathcal{O}} \hat{\mathcal{O}} \ .
$$
By abuse of notation we denote the map $ d \otimes_{\mathcal{O}} 1 $ just by $ d $. We have just defined the map
\begin{eqnarray*}
	d : R_{\mathcal{A}}^{n-r} & \longrightarrow & N_{\mathcal{A}}^{n-r} \\
	F & \longmapsto & \sum_{i=1}^{r} x_i \frac{\partial F}{\partial x_i} \frac{dx_i}{x_i} + \sum_{j=1}^{n-r} \frac{\partial F}{\partial y_j} dy_j 
\end{eqnarray*}
Note that for any index $ \ell $, $ 0 \leq \ell \leq n-r $, we have
$$
	d(R_{\mathcal{A}}^{\ell}) \subset N_{\mathcal{A}}^{\ell} \ .
$$
\begin{prop}\label{pr:monomialize_differentials_of_functions}
	Let $ F \in R_{\mathcal{A}}^{\ell} $ be formal function which is not a unit. We have that
	$$ \nu_{\mathcal{A}}(F) = \nu_{\mathcal{A}}(dF) \ . $$
	In addition, fixed a value $ \gamma \in \Gamma $, we have that $ F $ is $ \gamma $-final dominant (recessive) if and only if $ dF \in N_{\mathcal{A}}^{\ell} $ is $ \gamma $-final dominant (recessive).
	\begin{proof}
		Write $ F $ as a series in monomials in the independent variables:
		$$ F = \sum_{T \in \mathbb{Z}_{\geq 0}^r} F_{\boldsymbol{T}} (\boldsymbol{y}) \boldsymbol{x^T} .$$
		We have
		$$ dF = \sum_{T \in \mathbb{Z}_{\geq 0}^r} \boldsymbol{x^T} \bigg( F_{\boldsymbol{T}}(\boldsymbol{y}) \sum_{i=1}^r T_i \frac{dx_i}{x_i} + \sum_{i=r+1}^n \frac{\partial F_{\boldsymbol{T}}}{\partial y_i}(\boldsymbol{y}) dy_i \bigg) \ . $$
		The result follows from the equivalences
		$$ F_{\boldsymbol{T}}(\boldsymbol{y}) = 0 \Longleftrightarrow  F_{\boldsymbol{T}}(\boldsymbol{y}) \sum_{i=1}^r T_i \frac{dx_i}{x_i} + \sum_{i=r+1}^n \frac{\partial F_{\boldsymbol{T}}}{\partial y_i}(\boldsymbol{y}) dy_i = 0 $$
		and
		$$ F_{\boldsymbol{T}}(\boldsymbol{y}) \text{ is a unit } \Longleftrightarrow  F_{\boldsymbol{T}}(\boldsymbol{y}) \sum_{i=1}^r T_i \frac{dx_i}{x_i} + \sum_{i=r+1}^n \frac{\partial F_{\boldsymbol{T}}}{\partial y_i}(\boldsymbol{y}) dy_i \text{ is log-elementary.} $$
	\end{proof}
\end{prop}
Proceeding as before we define the map
$$
\Omega_{\mathcal{O}/k}(\log \boldsymbol{x}) \otimes_{\mathcal{O}} \hat{\mathcal{O}} \rightarrow \wedge^2 \ \Omega_{\mathcal{O}/k}(\log \boldsymbol{x}) \otimes_{\mathcal{O}} \hat{\mathcal{O}} \ ,
$$
and we denote it also by $ d $. Again, note that for any index $ \ell $, $ 0 \leq \ell \leq n-r $ we have
$$
	d(N_{\mathcal{A}}^{\ell}) \subset  \wedge^2 N_{\mathcal{A}}^{\ell} \ .
$$
A direct computation shows that for $ \omega \in N_{\mathcal{A}}^{\ell} $ we have
$$
\nu_{\mathcal{A}}(\omega) \leq \nu_{\mathcal{A}}(d\omega) \ .
$$
We state now the corresponding notions for pairs $ (\omega,F) \in N_{\mathcal{A}}^{\ell} \times R_{\mathcal{A}}^{\ell} $. Note that this Cartesian product has naturally structure of $ R_{\mathcal{A}}^{\ell} $-module. Given a pair $ (\omega,F) \in N_{\mathcal{A}}^{\ell} \times R_{\mathcal{A}}^{\ell} $ we say that
$$ \nu_{\mathcal{A}}(\omega,F) := \max \left\{ \lambda \ |\ (\omega,F) \in \mathcal{M}_{\mathcal{A}}^{\ell}(\lambda) \times \mathcal{Q}_{\mathcal{A}}^{\ell}(\lambda) \right\} $$
is \emph{the explicit value of $ (\omega,F) $}. We establish $ \nu_{\mathcal{A}}(0,0) := \infty $. As in the case of functions and $ 1 $-forms the explicit value of a pair is well defined.
\begin{defi}\label{de:gamma_final_1_pairs}
	Given $ \gamma \in \Gamma $ a pair $ (\omega,F) \in N_{\mathcal{A}}^{\ell} \times R_{\mathcal{A}}^{\ell} $ is \emph{$ \gamma $-final} if one of the following properties is satisfied:
	\begin{enumerate}
		\item $ \nu_{\mathcal{A}}(\omega,F) \leq \gamma $ and $ \omega $ and $ F $ are $ \nu_{\mathcal{A}}(\omega,F) $-final. In this case we say \emph{$ (\omega,F) $ is $ \gamma $-final dominant};
		\item $ \nu_{\mathcal{A}}(\omega,F) > \gamma $. In this case we say \emph{$ (\omega,F) $ is $ \gamma $-final recessive}.
	\end{enumerate}
\end{defi}
\begin{rem}
	Note that the definition of a $ \gamma $-final dominant pair is slightly different that the corresponding to functions or $ 1 $-forms. The following is an equivalent definition:
	\begin{quote}
		 A pair $ (\omega,F) \in N_{\mathcal{A}}^{\ell} \times R_{\mathcal{A}}^{\ell} $ is $ \gamma $-final dominant if $ \nu_{\mathcal{A}}(\omega,F) \leq \gamma $ and it can be written as
			$$ (\omega,F) = \boldsymbol{x}^{\boldsymbol{t}} (\tilde{\omega},\tilde{F}) + (\bar{\omega},\bar{f}) $$
			where $ \nu_{\mathcal{A}}(\bar{\omega},\bar{f}) > \nu_{\mathcal{A}}(\omega,F) = \nu(\boldsymbol{x}^{\boldsymbol{t}}) $ and one of the following options is satisfied:
			\begin{enumerate}
				\item $ \tilde{\omega} $ is log-elementary and $ \tilde{F} $ is a unit;
				\item $ \tilde{\omega} $ is log-elementary and $ \tilde{F} = 0 $;
				\item $ \tilde{\omega} = 0 $ and $ \tilde{F} $ is a unit.
			\end{enumerate}
	\end{quote}
\end{rem}
\begin{rem}
	Note that if $ \omega \in N_{\mathcal{A}}^{\ell} $ and $ F \in R_{\mathcal{A}}^{\ell} $ are both $ \gamma $-final then the pair $ (\omega,F) \in N_{\mathcal{A}}^{\ell} \times R_{\mathcal{A}}^{\ell} $ is also $ \gamma $-final. However, the opposite is not necessarily true: it happens when $ \nu_{\mathcal{A}}(\omega,F) = \nu_{\mathcal{A}}(\omega) = \nu_{\mathcal{A}}(F) < \gamma $ but only one of the terms of the pair is $ \nu_{\mathcal{A}}(\omega,F) $-final dominant.
\end{rem}
Given a $ \ell $-nested transformation $ \mathcal{A} \rightarrow \mathcal{A}' $ it is well defined the $ R_{\mathcal{A}}^{\ell} $-module monomorphism
$$ N_{\mathcal{A}}^{\ell} \times R_{\mathcal{A}}^{\ell} \longrightarrow N_{\mathcal{A}'}^{\ell} \times R_{\mathcal{A}'}^{\ell} \ . $$
As in previous cases, the explicit value of a pair can increase by means of a $ \ell $-nested transformation:
\begin{lem}\label{le:gamma_final_pairs_stable}
	Let $ (\omega,F) \in N_{\mathcal{A}}^{\ell} \times R_{\mathcal{A}}^{\ell} $ be a pair and let $ \pi : \mathcal{A} \longrightarrow \mathcal{B} $ be a $ \ell $-nested transformation. We have that
	$$ \nu_{\mathcal{A}}(\omega,F) \leq \nu_{\mathcal{B}}(\omega,F) \ . $$
	In particular we have
	\begin{enumerate}
		\item if $ (\omega,F) $ is $ \gamma $-final dominant in $ \mathcal{A} $, it is also $ \gamma $-final dominant in $ \mathcal{B} $ and $ \nu_{\mathcal{A}}(\omega,F) = \nu_{\mathcal{B}}(\omega,F) $.
		\item if $ (\omega,F) $ is $ \gamma $-final recessive in $ \mathcal{A} $, it is also $ \gamma $-final recessive in $ \mathcal{B} $.
	\end{enumerate}
\end{lem}
The following lemma provides a simple but powerful tool which allows to ``push right'' non dominant objects:
\begin{lem}\label{le:push_right}
	There is a $ \ell $-nested transformation $ \Psi_{\ell} : \mathcal{A} \rightarrow \mathcal{B} $ such that for any object (function, $ 1 $-form or pair) $ \psi $ we have:
	$$ \psi \text{ is not dominant in } \mathcal{A} \Longrightarrow \nu_{\mathcal{A}}(\psi) < \nu_{\mathcal{B}}(\psi) \ .$$
	\begin{proof}
		We have to perform Puiseux's packages with respect to all dependent variables $ y_1,y_2,\dots,y_{\ell} $. So we can take
		$$ \Psi_{\ell} := \pi_{\ell} \circ \cdots \circ \pi_{1} $$
		where $ \pi_j $ is a $ j $-Puiseux's package. 
	\end{proof}
\end{lem}

\subsection{Truncated Local Uniformization statements}
The following statement is a refinement of Theorem \ref{th:III}. It is the key in the proof of Theorem \ref{th:2}.
\begin{theo}[Truncated Local Uniformization of formal differential $ 1 $-forms]\label{th:gamma_finalization_1_forms}
	Let $ \mathcal{A} $ be a parameterized regular local model for $ K, \nu $ and let $ \ell $ be an index $ 0 \leq \ell \leq n-r$. Given a $ 1 $-form $ \omega \in N_{\mathcal{A}}^{\ell} $ and a value $ \gamma \in \Gamma $, if $  \nu_{\mathcal{A}}(\omega \wedge d \omega) \geq 2 \gamma  $ then there exists a $ \ell $-nested transformation $  \mathcal{A} \rightarrow \mathcal{B} $ such that $ \omega $ is $ \gamma $-final in $ \mathcal{B} $.
\end{theo}
Taking into account that for a formal function $ F $ we have that $ d (d F) \equiv 0 $, Theorem \ref{th:gamma_finalization_functions} and Proposition \ref{pr:monomialize_differentials_of_functions} implies the following statement.
\begin{theo}[Truncated Local Uniformization of formal functions]\label{th:gamma_finalization_functions}
	Let $ \mathcal{A} $ be a parameterized regular local model for $ K, \nu $ and let $ \ell $ be an index $ 0 \leq \ell \leq n-r$. Given a formal function $ F \in R_{\mathcal{A}}^{\ell} $ and a value $ \gamma \in \Gamma $, there exists a $ \ell $-nested transformation $ \mathcal{A} \rightarrow \mathcal{B} $ such that $ F $ is $ \gamma $-final in $ \mathcal{B} $.
\end{theo}
We have also the corresponding statement for pairs:
\begin{theo}\label{th:gamma_preparation_pairs}
	Let $ \mathcal{A} $ be a parameterized regular local model for $ K, \nu $ and let $ \ell $ be an index $ 0 \leq \ell \leq n-r$. Given a pair $ (\omega,F) \in N_{\mathcal{A}}^{\ell} \times R_{\mathcal{A}}^{\ell} $ and a value $ \gamma \in \Gamma $, if $ \nu_{\mathcal{A}}(\omega \wedge d \omega) \geq 2 \gamma $ then there exists a $ \ell $-nested transformation $ \mathcal{A} \rightarrow \mathcal{B} $ such that $ (\omega,F) $ is $ \gamma $-final in $ \mathcal{B} $.
\end{theo}
We have that Theorem \ref{th:gamma_preparation_pairs} is also a consequence of Theorem \ref{th:gamma_finalization_1_forms}. Thanks to Theorems \ref{th:gamma_finalization_1_forms} and \ref{th:gamma_finalization_functions} and Lemmas \ref{le:gamma_final_1_forms_stable} and \ref{le:gamma_final_functions_stable} we can obtain a pair whose both terms ($ 1 $-form and function) are $ \gamma $-final, hence the pair is also $ \gamma $-final.

\subsection{Induction procedure}
In the statement of Theorems \ref{th:gamma_finalization_1_forms}, \ref{th:gamma_finalization_functions} and \ref{th:gamma_preparation_pairs} appears a parameter $ \ell $, $ 0 \leq \ell \leq n-r$. Let us refer to these theorems by $ T_3(\ell) $, $ T_4(\ell) $ and $ T_5(\ell) $ respectively to indicate a fixed parameter $ \ell $. As we said in the previous section we have that
$$
T_3(\ell) \Rightarrow T_4(\ell) \quad \text{and} \quad T_3(\ell) \Rightarrow T_5(\ell) \ .
$$
Note also that for $ i=3,4,5 $ we have
$$
T_i(\ell) \Rightarrow T_i(\ell') \quad \text{for all} \quad 0 \leq \ell' < \ell \leq n-r \ ,
$$
and in particular
$$
T_i(n-r) \Leftrightarrow T_i \ .
$$
Our goal is to prove Theorem \ref{th:gamma_finalization_1_forms}. In Chapter \ref{ch:max_rat_rank} we proved  $ T_3(0) $. In Chapters \ref{ch:truncated_preparation} and \ref{ch:getting_final_forms} we conclude the proof of Theorem \ref{th:gamma_finalization_1_forms} by proving the induction step
$$ T_3(\ell) \Rightarrow T_3(\ell+1) \ , $$
so in particular we will also prove Theorems \ref{th:gamma_finalization_functions} and \ref{th:gamma_preparation_pairs}. However, in Chapter \ref{ch:function_case} we detail the proof of
$$ T_5(\ell) \Rightarrow T_5(\ell+1) \ , $$
since we will use that proof as a guide for the next chapters.

\section{Truncated Local Uniformization of functions}\label{ch:function_case}
Let $ \mathcal{A} = \big( \mathcal{O}, (\boldsymbol{x},\boldsymbol{y}) \big) $ be a parameterized regular local model of $ K,\nu $. Fix an index $ \ell $, $ 0 \leq \ell \leq n-r- 1 $ and a value $ \gamma \in \Gamma $.

In this chapter we consider a function $ F \in R_{\mathcal{A}}^{\ell +1} $ which is not $ \gamma $-final. We assume $ T_5(\ell) $ and we will show $ T_5(\ell+1) $, that is, there is a $ (\ell +1) $-nested transformation $ \mathcal{A} \rightarrow \mathcal{B} $ such that $ F $ is $ \gamma $-final in $ \mathcal{B} $.

\subsection{Truncated preparation of a function}\label{se:truncated_preparation_functions}
Let us denote the dependent variable $ y_{\ell+1} $ by $ z $. Write $ F $ as a power series in the dependent variable $ z $:
$$ F = \sum_{k=0}^{\infty} z^k F_k \ , \quad F_k \in R_{\mathcal{A}}^{\ell} \text{ for } k\geq 0 \ . $$
For each $ k \geq 0  $ denote by $ \phi_k(F;\mathcal{A}) \in \Gamma $ the explicit value
$$ \phi_k(F;\mathcal{A}) := \phi_k(F;\mathcal{A}) \ . $$
The \emph{Cloud of Points of $ F $} is the discrete subset
$$ \operatorname{CL}(F;\mathcal{A}) := \left\{ (\phi_k,k) \ |\ k = 0,1,\dots \right\} \ . $$
Note that
$$ F \neq 0 \Rightarrow \operatorname{CL}(F;\mathcal{A}) \neq \emptyset \ .$$
We also use the \emph{Dominant Cloud of Points of $ F $}
$$ \operatorname{DomCL}(F;\mathcal{A}) := \left\{ (\beta_k,k) \in \operatorname{CL}(F;\mathcal{A}) \ |\ F_k \text{ is dominant} \right\} \ . $$
Note that $ \operatorname{DomCL}(F;\mathcal{A}) $ can be empty. In Figure \ref{fi:clouds_functions} we can see an example in which the points $ (\phi_k,k) $ are represented with black and white circles, corresponding to dominant and non-dominant levels respectively.
\begin{figure}[!ht]
	\centering
	\includegraphics[width=0.8\textwidth]{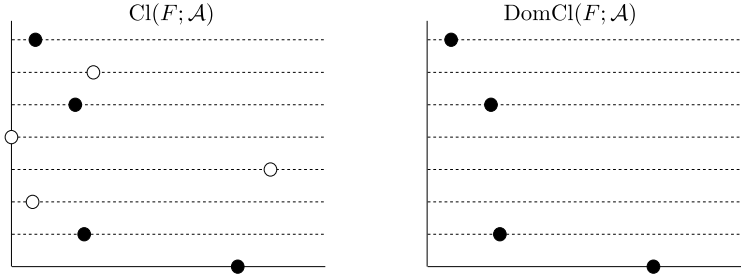}
	\caption{The Cloud of Points and the Dominant Cloud of Points}
	\label{fi:clouds_functions}
\end{figure}
Given a value $ \sigma \in \Gamma $, we define the truncated polygons
$$
\mathcal{N}(F;\mathcal{A},\sigma) \quad \text{and} \quad
\operatorname{Dom}\mathcal{N}(F;\mathcal{A},\sigma)
$$
to be the respective positively convex hulls of
$$
\{(0,\sigma/\nu(z)),(\sigma,0)\} \cup \operatorname{CL}(F;\mathcal{A})
$$
and
$$
\{(0,\sigma/\nu(z)),(\sigma,0)\} \cup \operatorname{DomCL}(F;\mathcal{A}) \ .
$$
Note that for any $ \sigma \in \Gamma $ we have that
$$ \mathcal{N}(F;\mathcal{A},\sigma) \supset \operatorname{Dom}\mathcal{N}(F;\mathcal{A},\sigma) \ . $$
In Figure \ref{fi:truncated_polygons_functions} we can see the truncated polygons corresponding to the cloud of points represented in Figure \ref{fi:clouds_functions}. Note that in this example $ \operatorname{Dom}\mathcal{N}(F;\mathcal{A},\gamma) $ has the vertex $ (0,\gamma/\nu(z)) $ which does not correspond to any level.
\begin{figure}[!ht]
	\centering
	\includegraphics[width=0.8\textwidth]{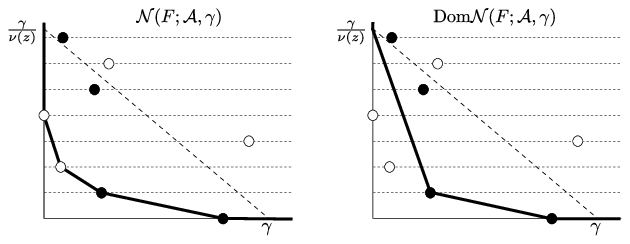}
	\caption{The Truncated Newton Polygon and the Dominant Truncated Newton Polygon}
	\label{fi:truncated_polygons_functions}
\end{figure}

For each $ k \geq 0 $ let us consider the real number
$$
\tau_k(F;\mathcal{A},\gamma) = \min \{ u \ |\ (u,k)\in  \operatorname{Dom}\mathcal{N}(F;\mathcal{A},\gamma) \} \ .
$$
Note that $0\leq \tau_k(F;\mathcal{A},\gamma)\leq\max\{0, \gamma-k\nu(z)\}$.
\begin{defi}
	A function $ F \in R_{\mathcal{A}}^{\ell+1} $ is {\em $\gamma$-prepared in $ \mathcal{A} $} if for any $0\leq k\leq \gamma/\nu(z)$, the level $F_k$ is $\tau_k(F;\mathcal{A},\gamma)$-final.
\end{defi}
The example represented in Figures \ref{fi:clouds_functions} and \ref{fi:truncated_polygons_functions} corresponds to a non $ \gamma $-prepared function.
\begin{prop}\label{pr:gammma_preparation_functions}
	Given a function $ F \in R_{\mathcal{A}}^{\ell+1} $ there is a $ \ell $-nested transformation $ \mathcal{A} \rightarrow \mathcal{B} $ such that $ F $ is $\gamma$-prepared in $\mathcal B$.
	\begin{proof}
		Let $ h $ be the integer part of $ \gamma / \nu(z) $. By $ T_4(\ell) $ there is a $ \ell $-nested transformation $ \mathcal{A} \rightarrow \mathcal{A}_1 $ such that $ F_0 $ is $ \gamma $-final dominant in $ \mathcal{A}_1 $. In the same way, there is a $ \ell $-nested transformation $ \mathcal{A}_1 \rightarrow \mathcal{A}_2 $ such that $ F_1 $ is $ (\gamma-\nu(z)) $-final dominant in $ \mathcal{A}_2 $. By Lemma \ref{le:gamma_final_functions_stable} we know that $ F_0 $ is $ \gamma $-final dominant in $ \mathcal{A}_2 $. After performing a finite number of $ \ell $-nested transformation we obtain a parameterized regular local model $ \mathcal{A}^* $ in which $ F_t $ is $ (\gamma-t\nu(z)) $-final for $ t=0,1,\dots,h $. Finally, performing the $ \ell $-nested transformation $ \Phi_{\ell} : \mathcal{A}^* \rightarrow \mathcal{B} $ given by Lemma \ref{le:push_right}, all the levels $ F_k $ with $ k>h $ becomes $ 0 $-final.
	\end{proof}
\end{prop}
A $ \ell $-nested transformation $ \mathcal{A} \rightarrow \mathcal{B} $ such that $ F $ is $ \gamma $-prepared in $ \mathcal{B} $ is a \emph{$ \gamma $-preparation for $ F $}.
\begin{rem}
	Note that thanks to Lemma \ref{le:gamma_final_functions_stable} given a $ \gamma $-preparation for $ F $
	$$ \mathcal{A} \rightarrow \mathcal{B} $$
	and any $ \ell $-nested transformation
	$$ \mathcal{B} \rightarrow \mathcal{C} $$
	then the composition
	$$ \mathcal{A} \rightarrow \mathcal{C} $$
	of both $ \ell $-nested transformations is also a $ \gamma $-preparation for $ F $.
\end{rem}
Note that if $ F \in R_{\mathcal{A}}^{\ell+1} $ is $ \gamma $-prepared then have that $ \mathcal{N}(F;\mathcal{A},\gamma) = \operatorname{Dom} \mathcal{N}(F;\mathcal{A},\gamma) $.

\subsection{The critical height of a $ \gamma $-prepared function}\label{se:critical_height_function}
Let $ F \in R_{\mathcal{A}}^{\ell+1}$ be a $ \gamma $-prepared function. Recall that in this situation we have that $ \mathcal{N}(F;\mathcal{A},\gamma) = \operatorname{Dom} \mathcal{N}(F;\mathcal{A},\gamma) $.

The \emph{critical value $ \delta(F;\mathcal{A},\gamma) $} is defined by
$$
	\delta(F;\mathcal{A},\gamma) := \min_{k \geq 0} \left\{ \tau_k(F;\mathcal{A},\gamma) + k \nu(z) \right\} \ .
$$
Note that $ \delta(F;\mathcal{A},\gamma) \leq \gamma $ since $ (0,\gamma) \in \mathcal{N}(F;\mathcal{A},\gamma) $. The critical value can be determined graphically:
$$ \delta(F;\mathcal{A},\gamma) = \min \left\{ \alpha \in \Gamma \ | \mathcal{N}(F;\mathcal{A},\gamma) \cap L_{\nu(z)}(\alpha) \neq \emptyset \right\} $$
where $ L_{\nu(z)}(\alpha) $ stands for the line passing though the point $ (\alpha,0) $ with slope $ -1/\nu(z) $. If no confusion arises we denote the critical value by $ \delta $. 

In the case $ \delta < \gamma $ we say that $ \mathcal{N}(F;\mathcal{A},\gamma) \cap L_{\nu(z)}(\delta) $ is \emph{the critical segment of $ \mathcal{N}(F;\mathcal{A},\gamma) $}. The highest vertex of the critical segment is the \emph{critical vertex}. The height of the critical vertex is the \emph{critical height of $ \mathcal{N}(F;\mathcal{A},\gamma)$} and is denoted by $ \chi(F;\mathcal{A},\gamma) $. This integer number is our main control invariant. It satisfies
$$  0 \leq \chi(F;\mathcal{A},\gamma) \leq \frac{\delta}{\nu(z)} < \frac{\gamma}{\nu(z)} \ . $$
If no confusion arises we denote the critical height by $ \chi $.

Note that if $ \delta(F;\mathcal{A},\gamma) < \gamma $ we have
$$ \delta(F;\mathcal{A},\gamma) \geq \nu_{\mathcal{A}}(F) + \chi(F;\mathcal{A},\gamma) \nu(z) \ , $$
where we have equality if and only if $ \nu_{\mathcal{A}}(F) $ is the abscissa of the critical vertex.
\begin{figure}[!ht]
	\centering
	\includegraphics[width=0.4\textwidth]{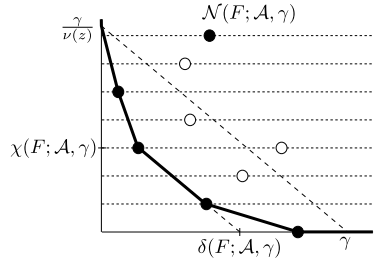}
	\caption{The critical value and the critical height}
	\label{fi:invariants_function}
\end{figure}

\subsection{Pre-$ \gamma $-final functions}\label{se:pre_final_functions}
\begin{defi}
	A $ \gamma $-prepared function $ F \in R_{\mathcal{A}}^{\ell+1} $ is \emph{pre-$ \gamma $-final} if
	$$ \delta(F;\mathcal{A},\gamma) = \gamma $$
	or
	$$ \delta(F;\mathcal{A},\gamma) < \gamma \quad \text{and} \quad \chi(F;\mathcal{A},\gamma) = 0 \ . $$
\end{defi}
Pre-$ \gamma $-final functions are easily recognizable by its Truncated Newton Polygon as it is represented in Figure \ref{fi:pre_gamma_final_function}
\begin{figure}[!ht]
	\centering
	\includegraphics[width=0.8\textwidth]{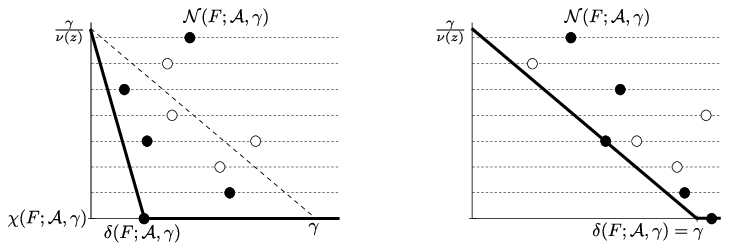}
	\caption{The two pre-$ \gamma $-final situations}
	\label{fi:pre_gamma_final_function}
\end{figure}
Let $ \Psi_{\ell+1} $ be the $ (\ell+1) $-nested transformation given in Lemma \ref{le:push_right}.
\begin{prop}
	Let $ F \in R_{\mathcal{A}}^{\ell+1}$ be a pre-$ \gamma $-final function. Consider the $ (\ell+1) $-nested transformation
	\begin{equation*}
		\xymatrix{\mathcal{A} \ar[r]^{\pi} &  \mathcal{B} \ar[r]^{\Psi_{\ell+1}}  & \mathcal{C}}
	\end{equation*}
	where $ \pi : \mathcal{A} \rightarrow  \mathcal{B} $ is a $ (\ell+1) $-Puiseux's package. Then $ F $ is $ \gamma $-final in $ \mathcal{C} $.
	\begin{proof}
		First, suppose we have $ \delta = \gamma $. In this situation for each index $ k \geq 0 $ we have
		$$ \nu_{\mathcal{A}}(F_k) \geq \gamma - k \nu(z) \ . $$
		Let $ \boldsymbol{x}' $ and $ z' $ be the variables in the parameterized regular local model $ \mathcal{B} $ obtained after perform the $ (\ell+1) $-Puiseux's package. From Equations \eqref{eq:variables_after_Puiseux_package} we know that
		$$ z = \boldsymbol{x}'^{\boldsymbol{\alpha}_0} (z' + \xi)^{\beta_0} \ , \quad \text{with} \ \nu(\boldsymbol{x}'^{\boldsymbol{\alpha}_0}) = \nu(z) \ , $$
		hence
		$$ \nu_{\mathcal{B}}(z^k) = \nu_{\mathcal{B}} \left(\boldsymbol{x}'^{k\boldsymbol{\alpha}_0} (z' + \xi)^{k\beta_0} \right) = k \nu(z) \ . $$
		Therefore, for each $ k \geq 0 $ we have
		$$ \nu_{\mathcal{B}}(z^k F_k) = \nu_{\mathcal{B}}(z^k) + \nu_{\mathcal{B}}(F_k) \geq \gamma \ . $$
		It follows that
		$$ \nu_{\mathcal{B}}(F) \geq \gamma \ . $$
		If $ \nu_{\mathcal{B}}(F) > \gamma $ then $ F $ is $ \gamma $-final recessive in $  \mathcal{B} $, so the same holds in $ \mathcal{C} $ (Lemma \ref{le:gamma_final_functions_stable}). On the other hand, if $ \nu_{\mathcal{B}}(F) = \gamma $, then $ F $ is $ \gamma $-final (dominant or recessive) in $ \mathcal{C} $ (see Lemma \ref{le:push_right}).
		
		Now, suppose $ \delta < \gamma $ and $ \chi = 0 $. In this situation for each index $ k \geq 1 $ we have
		$$ \nu_{\mathcal{A}}(F_k) > \delta - k \nu(z) \ , $$
		hence
		$$ \nu_{\mathcal{B}}(z^k F_k) > \delta \ . $$
		As a consequence we have that
		$$ \nu_{\mathcal{B}}(F-F_0) = \nu_{\mathcal{B}}(F- \sum_{k\geq1}F_k) > \delta \ , $$
		and therefore
		$$ \nu_{\mathcal{C}}(F-F_0) > \delta \ . $$	
		On the other hand we have that $ F_0 $ is dominant in $ \mathcal{A} $ with explicit value $ \delta $, hence
		$$ \nu_{\mathcal{B}}(F_0) =  \nu_{\mathcal{A}}(F_0) = \delta \ . $$
		By Lemma \ref{le:push_right} we know that $ F_0 $ is $ \gamma $-final dominant in $ \mathcal{C} $ with value $ \delta $. These facts imply that $ F $ is also $ \gamma $-final dominant in $ \mathcal{C} $ with value $ \delta $ (note that $ F = F_0 + (F-F_0) $).
	\end{proof}
\end{prop}

\subsection{Getting $ \gamma $-final functions}\label{se:getting_gamma_final_functions}
In this section we will complete the proof of $ T_4(\ell+1) $ by reductio ad absurdum: we suppose that we have a function $ F \in R_{\mathcal{A}}^{\ell+1} $ such that there is no $ (\ell+1) $-nested transformation $ \mathcal{A} \rightarrow \mathcal{B} $ such that $ F $ is pre-$ \gamma $-final in $ \mathcal{B} $ and we will get a contradiction.

Let $ \mathcal{A} $ be a parameterized regular local model and let $ F \in R_{\mathcal{A}}^{\ell+1} $ be a function. Assume
\begin{enumerate}
	\item $ F $ is $ \gamma $-prepared;
	\item for any $ (\ell+1) $-nested transformation $ \mathcal{A} \rightarrow \mathcal{B} $ we have that $ F $ is not pre $ \gamma $-final in $ \mathcal{B} $.
\end{enumerate}
For each index $ k \geq 0 $ let us write
$$ F_k = \boldsymbol{x}^{\boldsymbol{q}_k} \tilde{F}_k + \bar{F}_k \ , \quad \nu_{\mathcal{A}}(\bar{F}_k) > \nu(\boldsymbol{x}^{\boldsymbol{q}_k})  \ , $$
where we require $ \tilde{F}_k \in k[[y_1,y_2,\dots,y_{\ell}]] $. We have
$$ \nu_{\mathcal{A}}(F_{\chi}) = \nu(\boldsymbol{x}^{\boldsymbol{q}_{\chi}}) = \delta - \chi \nu(z) \ . $$ 

Let $ \phi \in K $ be the $ (\ell+1) $-contact rational function $ \phi_{\ell+1} = z^{d} / \boldsymbol{x^p} $, where $ d $ is the ramification index $ d(z;\mathcal{A}) $ (see Section \ref{se:Puiseux's_packages}).

Now, consider a level $ F_k $ which gives a point in the critical segment. We have that
$$ \nu_{\mathcal{A}}(F_k) = \nu(\boldsymbol{x}^{\boldsymbol{q}_{k}}) = \delta - k \nu(z) = \nu(\boldsymbol{x}^{\boldsymbol{q}_{\chi}}) + (\chi-k) \nu(z) \ . $$
Therefore, the index $ k $ must be of the form $ k = \chi - td $ for some integer $ 0 \leq t \leq \chi/d $. Following Remark \ref{re:monomial_after_Puiseux_package}, we have that
$$ \boldsymbol{x}^{\boldsymbol{q}_{\chi - td}} = \boldsymbol{x}^{\boldsymbol{q}_{\chi} + t \boldsymbol{p}} $$
hence
$$ z^{\chi-td}\tilde{F}_{\chi - td} = \boldsymbol{x}^{\boldsymbol{q}_{\chi}} z^{\chi} \phi^{-t} \tilde{F}_{\chi -td} \ . $$
Denote by $ M $ the integer part of $ \chi/d $. For $ 0 \leq t \leq M $ define the functions $ G_t \in k[[y_1,y_2,\dots,y_{\ell}]] $ given by
$$ G_t = \left\{
	\begin{array}{ll}
		\tilde{F}_{\chi - td} & \text{if } F_{\chi - td} \text{ gives a point in the critical segment} ; \\
		0 & \text{otherwise} \ .
	\end{array} \right. $$
Note that if $ G_t \neq 0 $ then it is a unit. For $ t=0,1,\dots,M $ write
$$ G_t = \tilde{G}_t + \bar{G}_t \ , \quad \tilde{G}_t \in k \ , \bar{G}_t \in \mathfrak{m} k[[y_1,y_2,\dots,y_{\ell}]] \ . $$

Let $ \tilde{F}_{\text{crit}} $  and $ \bar{F}_{\text{crit}} $ be the functions given by
$$ \tilde{F}_{\text{crit}} := \boldsymbol{x}^{\boldsymbol{q}_{\chi}} z^{\chi} \sum_{t=0}^{M} \phi^{-t} \tilde{G}_t
\quad \text{and} \quad
\bar{F}_{\text{crit}} := \boldsymbol{x}^{\boldsymbol{q}_{\chi}} z^{\chi} \sum_{t=0}^{M} \phi^{-t} \bar{G}_t \ . $$
Note that we have
\begin{equation}\label{eq:critical_tail_dependent_variables_1}
	\bar{F}_{\text{crit}} \in (y_1,y_2,\dots,y_{\ell}) R_{\mathcal{A}}^{\ell+1} \ .
\end{equation}
Denote by $ \breve{F} $ the function
$$ \breve{F} := F - \tilde{F}_{\text{crit}} - \bar{F}_{\text{crit}} \ . $$
Now we will study the behavior of $ F $ after performing a $ (\ell+1) $-nested transformation of the kind
$$ \xymatrix{\mathcal{A} \ar[r]^{\pi} &  \tilde{\mathcal{B}} \ar[r]^{\tau}  & \mathcal{B}} $$
where $ \pi:\mathcal{A} \rightarrow \tilde{\mathcal{B}} $ is a $ (\ell+1) $-Puiseux's package and $ \tau : \tilde{\mathcal{B}} \rightarrow \mathcal{B} $ is a $ \gamma $-preparation.

Perform a $ (\ell+1) $-Puiseux's package $ \mathcal{A} \rightarrow \tilde{\mathcal{B}} $ and let $ (\tilde{\boldsymbol{x}},\tilde{\boldsymbol{y}},\tilde{z}) $ be the coordinates in the parameterized regular local model $ \tilde{\mathcal{B}} $. We have
$$ \tilde{F}_{\text{crit}} = \tilde{\boldsymbol{x}}^{\boldsymbol{r}} \phi^e \sum_{t = 0}^{M} \phi^{-t} \tilde{G}_t
\quad \text{and} \quad
\bar{F}_{\text{crit}} = \tilde{\boldsymbol{x}}^{\boldsymbol{r}} \phi^e \sum_{t = 0}^{M} \phi^{-t} \bar{G}_t  \ , $$
where $ \nu(\tilde{\boldsymbol{x}}^{\boldsymbol{r}}) = \delta $. The exponents $ \boldsymbol{r} \in \mathbb{Z}_{\geq 0}^r $ and $ e \in \mathbb{Z}_{> 0} $ can be determined using the equalities given in \eqref{eq:variables_after_Puiseux_package}. Recall that $ \phi = \tilde{z} + \xi $ is a unit in $ R_{\tilde{\mathcal{B}}}^{\ell + 1} $. We can rewrite the above expressions as
$$ \tilde{F}_{\text{crit}} = \tilde{\boldsymbol{x}}^{\boldsymbol{r}} U \sum_{t = 0}^{M} (\tilde{z} + \xi)^{M-t} \tilde{G}_t
\quad \text{and} \quad
\bar{F}_{\text{crit}} = \tilde{\boldsymbol{x}}^{\boldsymbol{r}} U \sum_{t = 0}^{M} (\tilde{z} + \xi)^{M-t} \bar{G}_t  \ , $$
where $ U = \phi^{e-M} $ is a unit. Note that we have
\begin{equation}\label{eq:new_explicit_value_function_1}
	\nu_{\tilde{\mathcal{B}}}(\tilde{F}_{\text{crit}}) = \nu_{\tilde{\mathcal{B}}}(\bar{F}_{\text{crit}}) = \delta \ .
\end{equation}
On the other hand, it follows by construction that
\begin{equation}\label{eq:new_explicit_value_function_2}
	\nu_{\tilde{\mathcal{B}}}(\breve{F}) > \delta \ .
\end{equation}
Note that Equation \eqref{eq:critical_tail_dependent_variables_1} gives
\begin{equation}\label{eq:critical_tail_dependent_variables_2}
	\bar{F}_{\text{crit}} \in (\tilde{y}_1,\tilde{y}_2,\dots,\tilde{y}_{\ell}) R_{\tilde{\mathcal{A}}}^{\ell+1} \ .
\end{equation}
Let $ Q \in k[\tilde{z}] $ be the polynomial
$$ Q = \sum_{t = 0}^{M} \tilde{G}_t (\tilde{z} + \xi)^{M-t} \ , $$
and denote by $ \hbar $ its order. Note that $ \hbar \leq M \leq \chi $ and
\begin{eqnarray}\label{eq:resonance_condition_functions}
	\hbar = M & \Longleftrightarrow & Q = \tilde{G}_0 \tilde{z}^M  \nonumber \\
	& \Longleftrightarrow & \tilde{G}_t = (-1)^t \xi^t \binom{\chi}{t} \mu_0 \quad \text{for } 1 \leq t \leq M \ .
\end{eqnarray}
From Equations \eqref{eq:new_explicit_value_function_1}, \eqref{eq:new_explicit_value_function_2} and \eqref{eq:critical_tail_dependent_variables_1} we have that the $ \hbar $-level of $ F $ in $ \tilde{\mathcal{B}} $ is dominant.

Now, perform a $ \gamma $-preparation $ \tilde{\mathcal{B}} \rightarrow \mathcal{B} $. Let $ \delta' $ be the critical value of $ F $ in $ \mathcal{B} $. By assumption we have $ \delta' < \gamma $. Let $ \chi' $ be the new critical height.

Since the $ \hbar $-level of $ F $ in $ \tilde{\mathcal{B}} $ is dominant, the same happens in $ \mathcal{B} $ (Lemma \ref{le:gamma_final_functions_stable}). We also have that
$$ \nu_{\mathcal{B}}(F) = \nu_{\tilde{\mathcal{B}}}(F) = \nu_{\tilde{\mathcal{B}}}(\tilde{F}_{\text{crit}} + \bar{F}_{\text{crit}} + \breve{F}) = \delta \ . $$
We conclude that
\begin{equation}\label{eq:new_critical_height_functions}
	\chi' \leq \hbar \leq M = \left[\frac{\chi}{d}\right] \leq \chi \ .
\end{equation}
Inequality \eqref{eq:new_critical_height_functions} gives
$$ \chi' < \chi $$
except in the case when $ d = 1 $ and the condition about the coefficients of $ \tilde{F}_{\text{crit}} $ given in \eqref{eq:resonance_condition_functions} is satisfied. Note that we have
\begin{equation}\label{eq:critical_main_vertex_functions}
	\chi = \chi' \Longrightarrow \nu_{\mathcal{B}}(F)=\delta = \delta' - \chi \nu(z') \ ,
\end{equation}
where $ z' $ is the $ (\ell + 1) $-th dependent variable in $ \mathcal{B} $.

Suppose that $ \chi' = \chi $. In this situation instead of performing the $ (\ell+1) $-nested transformation
$$ \xymatrix{\mathcal{A} \ar[r]^{\pi} &  \tilde{\mathcal{B}} \ar[r]^{\tau}  & \mathcal{B}} $$
we will make an ordered change of the variable $ z $.

So we have a parameterized regular local model $ \mathcal{A} $ with $ d = d(z;\mathcal{A}) = 1 $ and a function $ F \in R_{\mathcal{A}}^{\ell+1} $ with critical height $ \chi $ and such that the coefficients of $ \tilde{F}_{\text{crit}} $ satisfy the condition given in \eqref{eq:resonance_condition_functions}. Furthermore, following Equation \eqref{eq:critical_main_vertex_functions}, after performing a $ (\ell+1) $-Puiseux's package and a $ \gamma $-preparation if necessary, we can assume that
$$ \nu_{\mathcal{A}}(F) = \nu_{\mathcal{A}}(F_{\chi}) \ . $$
Moreover, after performing a $ 0 $-nested transformation given by Lemma \ref{le:simple_infinite_list} if necessary, we can suppose that $ F_{\chi} $ divides $ F $. So, after factoring $ F_{\chi} $, we can assume that
$$ \nu_{\mathcal{A}}(F) = \nu_{\mathcal{A}}(F_{\chi}) = 1 \quad \text{and} \quad F_{\chi} = 1 \ . $$ 
Since $ F $ is $ \gamma $-prepared, the level at height $ (\chi-1) $ has the form
$$ F_{\chi-1} = \boldsymbol{x}^{\boldsymbol{p}} \tilde{F}_{\chi-1} + \bar{F}_{\chi-1} \ , \quad \nu_{\mathcal{A}}(\bar{F}_{\chi-1})> \nu(\boldsymbol{x}^{\boldsymbol{p}}) \ , $$
where $ \tilde{F}_{\chi-1} $ is a unit which does not depend on the independent variables $ \boldsymbol{x} $. Let us write $ \tilde{F}_{\chi-1} $ as a power series
$$ \tilde{F}_{\chi-1} = \sum_{(I,J)\in \mathbb{Z}^{r+\ell}_{\geq 0}} f_{IJ} \boldsymbol{x}^I \boldsymbol{y}^J \ , \quad f_{IJ} \in k \ . $$
Denote
$$ \tilde{F}_{\chi-1} = G + H $$
where $ G \in k[\boldsymbol{x},\boldsymbol{y}] \subset R_{\mathcal{A}}^{\ell} $ is the polynomial
$$ G = \sum_{\substack{(I,J)\in \mathbb{Z}^{r+\ell}_{\geq 0} \\ \nu(\boldsymbol{x}^I \boldsymbol{y}^J) \leq 2 \nu(z) }} f_{IJ} \boldsymbol{x}^I\boldsymbol{y}^J \  \ . $$
Since the coefficients of $ F $ satisfy the conditions given in \eqref{eq:resonance_condition_functions} we have
$$ G = - \xi \chi \boldsymbol{x}^{\boldsymbol{p}} + \cdots \ . $$
Note that
\begin{equation}\label{eq:explicit_value_Tschirhausen_functions}
	\nu_{\mathcal{A}}(\tilde{F}_{\chi-1}) = \nu_{\mathcal{A}}(G) = \nu(z) \leq \nu_{\mathcal{A}}(H) \ .
\end{equation}
Now consider the ordered change of coordinates
$$ \tilde{z} = z - \phi \ , \quad \text{where} \quad \phi:= \frac{-1}{\chi}G \ , $$
and let $ \tilde{\mathcal{A}} $ be the parameterized regular local model obtained. Note that
$$ \nu(\tilde{z}) \geq \nu(z) \ . $$
We have
$$ F = \sum_{k=0}^{\infty} z^k F_k =\sum_{k=0}^{\infty} \left( \tilde{z} + \phi \right)^k F_k = \sum_{k=0}^{\infty} \tilde{z}^k F'_k \ , $$
where
$$ F'_k = F_k + \sum_{i=1}^{\infty} \binom{k+i}{i}  \phi^i F_{k+i} \ . $$
So the $ (\chi-1) $-level of $ F $ in $ \tilde{\mathcal{A}} $ is
\begin{equation}\label{eq:coefficient_after_Tschirhausen_functions}
	F'_{\chi-1} = F_{\chi-1} + \chi \phi F_{\chi} + \phi^2 (\cdots) = G + H - G + \phi^2 (\cdots) = H + \phi^2 (\cdots) \ .
\end{equation}
In $ \tilde{\mathcal{A}} $ the function $ F $ is not necessarily $ \gamma $-prepared so let $ \tilde{\mathcal{A}} \rightarrow \mathcal{A}_1 $ a $ \gamma $-preparation.

It follows from the definition of $ H $ and Equations \eqref{eq:explicit_value_Tschirhausen_functions} and \eqref{eq:coefficient_after_Tschirhausen_functions} that
$$ \nu_{\mathcal{A}_1}(F'_{\chi-1}) \geq 2 \nu(z) \ . $$
Note also that we still have
$$ \nu_{\mathcal{A}_1}(F) = \nu_{\mathcal{A}_1}(F'_{\chi}) = 0 \ . $$
Let $ z_1 $ be the $ (\ell+1) $-th dependent variable in $ \mathcal{A}_1 $. We have that
$$ \chi(F;\mathcal{A}_1,\gamma) \leq \chi(F;\mathcal{A},\gamma) \ . $$
Furthermore, we have
\begin{equation}\label{eq:value_variable_after_Tschirhausen_functions}
	\chi(F;\mathcal{A}_1,\gamma) = \chi(F;\mathcal{A},\gamma) \Longrightarrow \nu(z_1) = \nu_{\mathcal{A}_1}(F'_{\chi-1}) \geq 2 \nu(z) \ .
\end{equation}
Now, we can perform a $ z_1 $-Puiseux's package. If the critical height does not drop, instead of performing a $ z_1 $-Puiseux's package we make an ordered change of coordinates as above. We iterate this procedure while the critical height does not drop. At each step we obtain a parameterized regular local model $ \mathcal{A}_i $. By Equation \eqref{eq:value_variable_after_Tschirhausen_functions} we know that the $ (\ell+1) $-th dependent variable $ z_i $ satisfies
$$ \nu(z_i) \geq 2^i \nu(z) \ . $$
This can not happen infinitely many times, since in a finite number of steps we reach a parameterized regular local model $ \mathcal{A}_{i_0} $ such that
$$ \nu(z_{i_0}) \geq \frac{\gamma - \nu_{\mathcal{A}}(F)}{\chi} = \frac{\gamma - \nu_{\mathcal{A}_{i_0}}(F)}{\chi}\ . $$
The above inequality implies that $ \delta(F;\mathcal{A}_{i_0},\gamma) = \gamma $ which is in contradiction with our assumptions.

Then, after a finite number of ordered changes of the $ (\ell+1) $-th variable and $ \gamma $-preparations of $ F $, necessarily we reach a parameterized regular local model in which the critical height drops by means of a  $ (\ell+1) $-Puiseux's package. Again, this can not happen infinitely many times since we are assuming that the critical height is strictly positive.

We have just proved that our assumptions give a contradiction, so there is always a $ (\ell +1) $-nested transformation which transform a function into a $ \gamma $-final one. Thus, we have prove that
$$ T_4(\ell) \Longrightarrow T_4(\ell+1) \ . $$

\section{Truncated preparation of a $ 1 $-form}\label{ch:truncated_preparation}
In this chapter and the next one we will detail the proof of
$$ T_3(\ell) \Longrightarrow T_3(\ell+1) \ . $$
We will adapt the arguments used in Chapter \ref{ch:function_case} to the case of $ 1 $-forms. As the name of the chapter says, this chapter is the equivalent for $ 1 $-forms of the Section \ref{se:truncated_preparation_functions}. 

Let $ \mathcal{A} = \big( \mathcal{O}, (\boldsymbol{x},\boldsymbol{y}) \big) $ be a parameterized regular local model of $ K,\nu $. Fix an index $ \ell $, $ 0 \leq \ell \leq n-r- 1 $ and a value $ \gamma \in \Gamma $. We consider a 1-form $\omega \in N_{\mathcal{A}}^{l+1}$ such that $ \nu_{\mathcal{A}}(\omega \wedge d \omega) \geq 2 \gamma $. Since we are working by induction on $\ell$, we assume that the statement $T_3(\ell)$ is true (hence $T_4(\ell)$ and $T_5(\ell)$ are also true).

\subsection{Expansions relative to a dependent variable}\label{se:level_expansion}
Let us denote the dependent variable $ y_{l+1} $ by $ z $. Note that by definition
$$ R_{\mathcal{A}}^{l+1} = R_{\mathcal{A}}^{\ell}[[z]] \ . $$
Thus we can expand an element $ F \in R_{\mathcal{A}}^{l+1} $ as a power series in the variable $ z $:
$$ F = \sum_{k=0}^{\infty} F_k z^k \quad , \ F_k \in R_{\mathcal{A}}^{\ell} . $$
Take an element of $ N_{\mathcal{A}}^{l+1} $
$$ \omega = \sum_{i=1}^{r} a_i \frac{dx_i}{x_i} + \sum_{j=1}^{\ell} b_j dy_j + c dz \ . $$
Write
$$ \omega = \sum_{i=1}^{r} a_i \frac{dx_i}{x_i} + \sum_{j=1}^{\ell} b_j dy_j + f \frac{dz}{z} \ , $$
where $ f = zc $. The \textit{decomposition in $ z $-levels of $ \omega $} of consists in writing $ \omega $ as
\begin{equation}\label{eq:levels_decomposition}
	\omega = \sum_{k=0}^{\infty} z^{k} \omega_k = \sum_{k=0}^{\infty} z^{k} \Big( \sum_{i=1}^{r} a_{ik} \frac{dx_i}{x_i} + \sum_{j=1}^{\ell} b_{jk} dy_j + f_k \frac{dz}{z} \Big)  \ .
\end{equation}
where
$$ a_i = \sum_{k=0}^{\infty} a_{ik} z^k \ , \ b_j = \sum_{k=0}^{\infty} b_{jk} z^k \ \text{ and } \ f = \sum_{k=1}^{\infty} f_{0} z^k \ . $$
Note that $ f_0 = 0 $. We say that $ \omega_k $ is \emph{the $ k $-level of $ \omega $}.
\begin{rem}\label{re:levels_structure}
	The coefficients of each $ z $-level $ \omega_k $ are elements of $ R_{\mathcal{A}}^{\ell} $, but $ \omega_k $ itself belongs neither to $ N_{\mathcal{A}}^{\ell} $ nor to $ N_{\mathcal{A}}^{l+1} $. The $ z $-levels $ \omega_k $ belong to the $ R_{\mathcal{A}}^{\ell} $-module $ N_{\mathcal{A}}^{\ell} \oplus R_{\mathcal{A}}^{\ell} \frac{dz}{z} $. We will write
	$$ \omega_k = \eta_k + f_k \frac{dz}{z} \in N_{\mathcal{A}}^{\ell} \oplus R_{\mathcal{A}}^{\ell} \frac{dz}{z} \ , \ \forall k \geq 0 \ , $$
	where we denote by $ \eta_k \in N_{\mathcal{A}}^{\ell} $ the $ 1 $-forms
	$$ \eta_k := \sum_{i=1}^{r} a_{ik} \frac{dx_i}{x_i} + \sum_{j=1}^{\ell} b_{jk} dy_j \ , \ \forall k \geq 0 \ . $$
\end{rem}
\noindent To each level we can attach a pair
$$ \omega_k = \eta_k + f_k \frac{dz}{z} \longmapsto (\eta_k,f_k) \in N_{\mathcal{A}}^{\ell} \times R_{\mathcal{A}}^{\ell} \ . $$
Denote by $ \delta_k(\omega;\mathcal{A}),\phi_k(\omega;\mathcal{A}), \beta_k(\omega;\mathcal{A}) \in \Gamma \cup \{\infty\} $ the explicit values
\begin{eqnarray*}
	\delta_k(\omega;\mathcal{A}) & := & \nu_{\mathcal{A}}(\eta_k) \ , \\ 
	\phi_k(\omega;\mathcal{A}) & := & \nu_{\mathcal{A}}(f_k) \ , \\
	\beta_k(\omega;\mathcal{A}) & := & \nu_{\mathcal{A}}(\eta_k,f_k) = \min \left\{ \phi_k(\omega;\mathcal{A}),\eta_k(\omega;\mathcal{A})  \right\} \ .
\end{eqnarray*}
The value $ \beta_k(\omega;\mathcal{A}) $ is the \emph{explicit value of $ \omega_k $}. If no confusion arises we denote $ \delta_k(\omega;\mathcal{A}) $, $ \phi_k(\omega;\mathcal{A}) $ and $ \beta_k(\omega;\mathcal{A}) $ by $ \delta_k $, $ \phi_k $ and $ \beta_k $ respectively. Given $ \alpha \in \Gamma $ we say that the level $ \omega_k $ is \emph{$ \alpha $-final} (\emph{final dominant}, \emph{final recessive}) if and only if the pair $ (\eta_k,f_k) $ is $ \alpha $-final (final dominant, final recessive). In particular, we say that a level $ \omega_k $ is \emph{log-elementary} if it is $ 0 $-final dominant, and that it is \emph{dominant} if it is $ \beta_k $-final dominant. 

From Lemmas \ref{le:gamma_final_functions_stable}, \ref{le:gamma_final_1_forms_stable} and \ref{le:gamma_final_pairs_stable} we obtain the following property of stability of $ \delta_k $, $ \phi_k $ and $ \beta_k $ under any $ \ell $-nested transformation:
\begin{quote}
	{\bf Property of stability of levels}. For any $ \ell $-nested transformation ${\mathcal A}\rightarrow {\mathcal A}'$  and any $k\geq 0$, we have that $ \delta_k' \geq \delta_k $, $ \phi_k' \geq \phi_k $ and $ \beta_k' \geq \beta_k $.
\end{quote}
In addition we have stability for dominant levels:
\begin{quote}
	{\bf Property of stability for dominant levels}. Given a dominant level $\omega_k$ and any $ \ell $-nested transformation $ {\mathcal A} \rightarrow {\mathcal A}'$, the transformed level $\omega'_k$ is also dominant and  $\beta_k'=\beta_k$.
\end{quote}

\subsection{Truncated Newton polygons and prepared $ 1 $-forms}
Using the values defined in the previous section we define certain subsets of $ \Gamma_{\geq 0} \times \mathbb{Z}_{\geq 0} \subset \mathbb{R}_{\geq 0}^{2} $. The \emph{Cloud of Points of $ \omega $} is the discrete subset
$$ \operatorname{CL}(\omega;\mathcal{A}) := \left\{ (\beta_k,k) \ |\ k = 0,1,\dots \right\} \ . $$
Note that
$$ \omega \neq 0 \Rightarrow \operatorname{CL}(\omega;\mathcal{A}) \neq \emptyset \ .$$
We also use the \emph{Dominant Cloud of Points of $ \omega $}
$$ \operatorname{DomCL}(\omega;\mathcal{A}) := \left\{ (\beta_k,k) \in \operatorname{CL}(\omega;\mathcal{A}) \ |\ \omega_k \text{ is dominant} \right\} \ . $$
Note that $ \operatorname{DomCL}(\omega;\mathcal{A}) $ can be empty. In Figure \ref{fi:clouds} we can see an example in which the points $ (\beta_k,k) $ are represented with black and white circles, corresponding to dominant and non-dominant levels respectively.
\begin{figure}[!ht]
	\centering
	\includegraphics[width=0.8\textwidth]{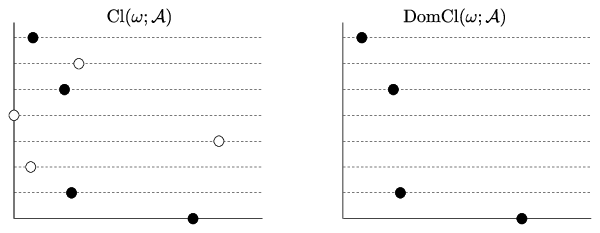}
	\caption{The Cloud of Points and the Dominant Cloud of Points}
	\label{fi:clouds}
\end{figure}

\noindent Given a value $ \sigma \in \Gamma $, we define the truncated polygons
$$
\mathcal{N}(\omega;\mathcal{A},\sigma) \quad \text{and} \quad
\operatorname{Dom}\mathcal{N}(\omega;\mathcal{A},\sigma)
$$
to be the respective positively convex hulls of
$$
\{(0,\sigma/\nu(z)),(\sigma,0)\} \cup \operatorname{CL}(\omega;\mathcal{A})
$$
and
$$
\{(0,\sigma/\nu(z)),(\sigma,0)\} \cup \operatorname{DomCL}(\omega;\mathcal{A}) \ .
$$
Note that for any $ \sigma \in \Gamma $ we have that
$$ \mathcal{N}(\omega;\mathcal{A},\sigma) \supset \operatorname{Dom}\mathcal{N}(\omega;\mathcal{A},\sigma) \ . $$
In Figure \ref{fi:truncated_polygons} we can see the truncated polygons corresponding to the cloud of points represented in Figure \ref{fi:clouds}. Note that in this example $ \operatorname{Dom}\mathcal{N}(\omega;\mathcal{A},\gamma) $ has the vertex $ (0,\gamma/\nu(z)) $ which does not correspond to any level.
\begin{figure}[!ht]
	\centering
	\includegraphics[width=0.8\textwidth]{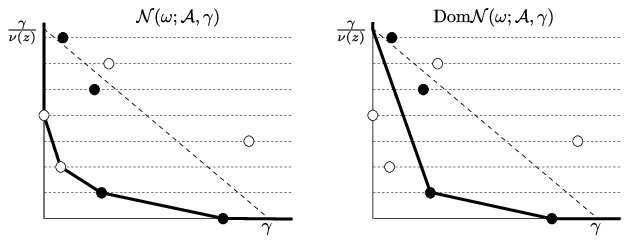}
	\caption{The Truncated Newton Polygon and the Dominant Truncated Newton Polygon}
	\label{fi:truncated_polygons}
\end{figure}

For each $ k \geq 0 $ let us consider the real number
$$
\tau_k(\omega;\mathcal{A},\gamma) := \min \{ u \ |\ (u,k)\in  \operatorname{Dom}\mathcal{N}(\omega;\mathcal{A},\gamma) \} \ .
$$
Note that $0\leq \tau_k(\omega;\mathcal{A},\gamma)\leq\max\{0, \gamma-k\nu(z)\}$.
\begin{defi}
	We say that  $\omega$ is {\em $\gamma$-prepared } in $\mathcal A$ if and only if the level $\omega_k$ is $\tau_k(\omega;\mathcal{A},\gamma)$-final for any $0\leq k\leq \gamma/\nu(z)$.
\end{defi}
The example represented in Figures \ref{fi:clouds} and \ref{fi:truncated_polygons} corresponds to a non $ \gamma $-prepared $ 1 $-form.
\begin{rem}\label{re:gamma_preparation_polygons}
	Note that being $ \gamma $-prepared implies that $ \mathcal{N}(\omega;\mathcal{A},\gamma) = \operatorname{Dom\mathcal{N}} $. Conversely, if we have that $ \mathcal{N}(\omega;\mathcal{A},\gamma) = \operatorname{Dom\mathcal{N}} $ to assure that $\omega$ is $\gamma$-prepared it is enough to guarantee that $ \beta_k > \tau_k $ for any level $\omega_k$ which is not $\tau_k$-dominant. This last condition can be obtained applying Lemma \ref{le:push_right}.
\end{rem}
The objective of this chapter is to prove the following result
\begin{theo}[Existence of $ \gamma $-preparation]\label{th:gammma_preparation}
	Let $ \omega \in N_{\mathcal{A}}^{\ell+1} $ be a $ 1 $-form such that $ \nu_{\mathcal{A}}(\omega \wedge d \omega) \geq 2 \gamma $. There is a $ \ell $-nested transformation $ \mathcal{A} \rightarrow \mathcal{B} $ such that $ \omega $ is $\gamma$-prepared in $\mathcal B$.
\end{theo}
A $ \ell $-nested transformation $ \mathcal{A} \rightarrow \mathcal{B} $ such that $ \omega $ is $ \gamma $-prepared in $ \mathcal{B} $ is called a \emph{$ \gamma $-preparation}.

\subsection{Property of preparation of levels}
Consider and integer number $k\geq 0$ and put
\begin{equation}\label{eq:preparation_levels_triangles}
	\lambda_{k} (\omega;\mathcal{A},\gamma) = \min \{ \gamma - k \nu(z) \} \cup \left \{ \frac{ \delta_{k+s} + \delta_{k-s} }{2} \ |\ \quad s \geq 1 \right \} \ .
\end{equation}
In this section we prove the following Lemma
\begin{lem}\label{le:preparation_levels_triangles}
	There is an $ \ell $-nested transformation $ {\mathcal A} \rightarrow {\mathcal A}'$ such that the $ k $-level $ \omega'_k $ of $ \omega $ with respect to $ {\mathcal A}' $ is $ \lambda_{k} ({\mathcal A};\omega,\gamma) $-final.
\end{lem}
This result is a consequence of the induction hypothesis and the fact that $ \nu_{\mathcal{A}}(\omega \wedge d \omega) \geq 2 \gamma $. Namely, we can write
$$
	\omega \wedge d \omega = \sum_{m=0}^{\infty} z^m \Big( \Theta_m + \frac{dz}{z} \wedge \Delta_m \Big)
$$
where
$$ \Theta_m := \sum_{i+j=m} \eta_i \wedge d \eta_j $$
and
$$ \Delta_m := \sum_{i+j=m} j \eta_j \wedge \eta_i +f_i d \eta_j + \eta_i \wedge d f_j \ . $$
We have that
$$
	\nu_{\mathcal{A}}(\omega \wedge d \omega) \geq 2 \gamma 
$$
is equivalent to
$$ \nu_{\mathcal{A}}(\Theta_m) \geq 2 \gamma \text{ and } \nu_{\mathcal{A}}(\Delta_m) \geq 2 \gamma \quad \forall \, m \geq 0 \ .
$$
The proof of Lemma \ref{le:preparation_levels_triangles} is based on this equivalence. In view of the statement $T_3(\ell)$ it is enough to prove that 
$$
\nu_{\mathcal{A}}(\eta_k\wedge d\eta_k) \geq 2 \lambda_k \ .
$$
Look at $ \Theta_{2 k} $:
$$
\Theta_{2 k} = \eta_{k} \wedge d \eta_{k} + \sum_{\substack{i+j=2 k \\ i,j \neq k}} \eta_i \wedge d \eta_j \ . $$
Recall that $ \nu_{\mathcal{A}}(\Theta_{2k}) \geq 2 \gamma  $. By definition of the values $ \delta_i $, we have that $ \nu_{\mathcal{A}}(\eta_i) \geq \delta_i $, hence $ \nu_{\mathcal{A}}(d \eta_i) \geq \delta_i $.
Writing
$$
\eta_{k} \wedge d \eta_{k} = - \Theta_{2 k} + \sum_{\substack{i+j=2 k \\ i,j \neq k}} \eta_i \wedge d \eta_j \ ,
$$
we conclude that $ \nu_{\mathcal{A}}(\eta_{k} \wedge d \eta_{k} ) \geq 2 \lambda_k $. We end the proof of Lemma \ref{le:preparation_levels_triangles} applying $T_5(\ell)$ to the pair $(\eta_k, f_k)$.

\subsection{Preparation. First reductions}\label{se:preparationreductions}
In order to prove Theorem \ref{th:gammma_preparation} we will first show that we can assume some reductions.
\begin{prop}\label{pr:preparation_reductions}
	Let $ \omega \in N_{\mathcal{A}}^{\ell+1} $ be a $ 1 $-form such that $ \nu_{\mathcal{A}}(\omega \wedge d \omega) \geq 2 \gamma $. Without lost of generality we can assume that the following properties are satisfied:
	\begin{enumerate}
		\item \textbf{Maximality of dominant levels}: For any integer number $k$ with $0\leq k\leq \gamma/\nu(z)$, the level $\omega_k$ is either $(\gamma-k\nu(z))$-final dominant in $ \mathcal{A} $ or there is no $ \ell $-nested transformation $ \mathcal{A} \rightarrow \mathcal{A}' $ such that $\omega_k$ is $(\gamma-k\nu(z))$-final dominant in $ \mathcal{A}' $;
		\item \textbf{Preparation of the functional part}: For any integer number $k$ with $0\leq k\leq \gamma/\nu(z)$, the function $ f_k \in R_{\mathcal{A}}^{\ell}$ is $ (\gamma -k\nu(z)) $-final;
		\item \textbf{Preparation of the $0$-level}: The $0$-level $ \omega_0 $ is $ \gamma $-final;
		\item \textbf{Explicitness of the dominant vertices}: Any vertex of $ \operatorname{Dom}\mathcal{N}(\omega;\mathcal{A},\gamma) $ is also a vertex of $ \mathcal{N}(\omega;\mathcal{A},\gamma) $.
	\end{enumerate}
\end{prop}
First of all, note that the four properties listed in the proposition are stable under further $ \ell $-nested transformations. The first property can be obtained without making use of the induction hypothesis while the remaining ones needs the assumption that $ T_3(\ell) $ is true. In this section we detail how to obtain the first three properties. The following section is devoted to show how to obtain explicit dominant vertices.

\textbf{Maximality of dominant levels}: This property can be obtained thanks to the stability of dominant levels as follows. Take an integer $k$ with $0\leq k\leq \gamma/\nu(z)$. If there is a $ \ell $-nested transformation such that $ \omega_k $ becomes $(\gamma-k\nu(z))$-final dominant, perform it. In this way we perform a finite number of transformations to get the desired maximality property.

\textbf{Preparation of the functional part:} We only have to use $ T_4(\ell) $ finitely many times.

\textbf{Preparation of the $0$-level:} This property is also obtained using the induction hypothesis. We have that $ \nu_{\mathcal{A}}(\omega \wedge d \omega) \geq 2 \gamma $ implies $ \nu_{\mathcal{A}} (\Theta_0) \geq 2 \gamma $. Since $ \Theta_0 = \eta_0 \wedge d \eta_0 $ we can invoke $ T_3(\ell) $ and transform $ \omega_0 $ into a $ \gamma $-final level.

\subsection{Getting explicit dominant vertices}
In this section we complete the proof of Proposition \ref{pr:preparation_reductions} by showing how to obtain explicit dominant vertices. First of all, note that the maximality of dominant vertices property implies the following additional property:
\begin{quote}
	\textbf{Stability of the Truncated Dominant Newton Polygon}: For any $ \ell $-nested transformation $ \mathcal{A} \rightarrow \mathcal{A}' $ we have that
	$$
	\textrm{Dom}\mathcal{N}(\omega;\mathcal{A},\gamma) = \textrm{Dom}\mathcal{N}(\omega;\mathcal{A}',\gamma) \ .
	$$
\end{quote}
In view of this stability property, from now on we will denote the Truncated Dominant Newton Polygon $ \textrm{Dom}\mathcal{N}(\omega;\mathcal{A},\gamma) $ simply by $ \textrm{Dom}\mathcal{N} $, and the values $ \tau_k(\omega;\mathcal{A},\gamma) $ by $ \tau_k $.

For any positive real number $\delta>0$, let us consider the lines
$$ L_\delta (\rho) = \{ (u,v) \ |\ u + \delta v = \rho \} $$
of slope $ -1/\delta $, and the open half-planes
$$ H^{+}_{\delta} (\rho) = \{ (u,v) \ |\ u + \delta v >\rho \} \quad \text{and} \quad H^{-}_{\delta} (\rho) = \{ (u,v) \ |\ u + \delta v <\rho \}   \ . $$
Let $ \rho_{\delta} $ be the real number defined by
$$
\rho_\delta : = \min \{ \rho \ |\ L_\delta(\rho) \cap \operatorname{Dom}\mathcal{N} \neq \emptyset \} = \sup \{ \rho \ |\ H^{+}_{\delta} (\rho)\supset \operatorname{Dom}\mathcal{N} \} \ .
$$
We have that $ L_\delta(\rho_\delta) $ cuts the polygon $ \operatorname{Dom}\mathcal{N} $ in only one vertex or a side joining two vertices.

In order to get explicit dominant vertices, it is enough to prove the following lemma:
\begin{lem}\label{le:epsilon}
	Given $\delta>0$ and a fixed $\epsilon >0$ there is an $ \ell $-nested transformation $ \mathcal{A} \rightarrow \mathcal{A}' $ such that $ \mathcal{N}(\omega;\mathcal{A}',\gamma) \subset H^{+}_{\delta}(\rho_{\delta} - \epsilon) $.
\end{lem}
Note that as usual, the property obtained after applying the lemma is stable under new $ \ell $-nested transformations.

Let us show that Lemma \ref{le:epsilon} allows us to get the explicitness of dominant vertices property. Consider a vertex $ \boldsymbol{v} = (\tau_k,k) $ of $ \operatorname{Dom}\mathcal{N} $. Take two slopes $ 0 < \delta_2 < \delta_1 $ such that both $ L_{{\delta_1}}(\rho_{{\delta}_1}) $ and $ L_{{\delta_2}}(\rho_{{\delta}_2}) $ cut the polygon $ \operatorname{Dom}\mathcal{N} $ only in the vertex $ \boldsymbol{v} $. Consider an slope $ \delta_3 $ with $ \delta_2 < \delta_3 < \delta_1 $. We also have that $L_{\delta_3}(\rho_{\delta_3})$ cuts $ \operatorname{Dom}\mathcal{N} $ only in the vertex ${\boldsymbol{v}}$. By an elementary geometrical argument, we see that there is an $ \epsilon>0 $ satisfying the following property
\begin{quote}
	``For any $ (a,k') \in H^+_{{\delta_1}}(\rho_{{\delta}_1}-\epsilon) \cap H^+_{{\delta_2}}(\rho_{{\delta}_2}-\epsilon)$ such that  $k'\ne k$ we have that $(a,k')\in H^+_{{\delta_3}}(\rho_{{\delta}_3}) $''.
\end{quote}
Applying Lemma \ref{le:epsilon} with respect to $ \epsilon $, we obtain that $ \boldsymbol{v} $ is a vertex of $\mathcal{N}(\omega;\mathcal{A}',\gamma) $. Repeating this argument at all the vertices of $ \operatorname{Dom}\mathcal{N} $ we obtain the explicitness of dominant vertices property.
\begin{figure}[!ht]
	\centering
	\includegraphics[width=0.8\textwidth]{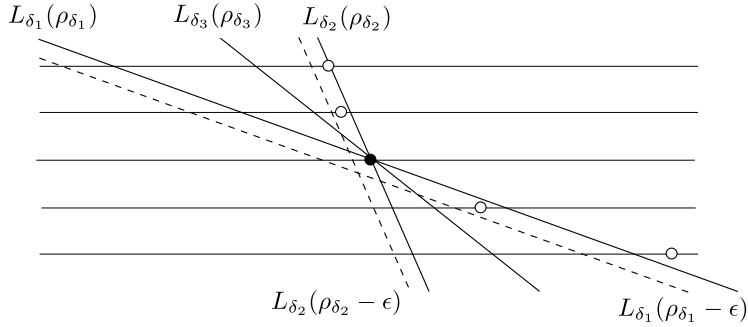}
	\caption{Situation after applying Lemma \ref{le:epsilon}}
	\label{fi:epsilon_lemma}
\end{figure}

So in order to complete the proof of Proposition \ref{pr:preparation_reductions} we have to prove Lemma \ref{le:epsilon}. Denote by $(\alpha_{k},k)$ the point in $ \mathcal{N}(\omega;\mathcal{A},\gamma)
$ with integer ordinate equal to  $k$ and smallest abscissa. Note that
$$
\alpha_k\leq \tau_k \ .
$$
Note also that $\alpha_0\leq \gamma$ and $\alpha_k=0$ for any $k>\gamma/\nu(z)$. We use as a key argument the following property
\begin{quote}
	{\bf Reduction of vertices}: Consider a vertex $\boldsymbol{v}=(\alpha_k,k)$ of the polygon  $ \mathcal{N}(\omega;\mathcal{A},\gamma) $ which is not a vertex of $ \operatorname{Dom}\mathcal{N} $ (in particular $k\geq 1$). There is a $ \ell $-nested transformation $ \mathcal{A} \rightarrow \mathcal{A} $ such that
	$$
	\mathcal{N}(\omega;\mathcal{A}',\lambda) \subset \mathcal{N}'
	$$
	where $ \mathcal{N}' $ is the positively convex polygon generated by $ \{ (\alpha'_s,s) \}_{s\geq 0} $ where
	$$
	\alpha'_s=\left\{
	\begin{array}{lc}
	\alpha_s \ , & \text{ if } s \neq k \ ; \\
	\frac{\alpha_{k-1}+ \alpha_{k+1}}{2}\ , & \text{ if } s=k \ .
	\end{array}
	\right.
	$$
\end{quote}
This property is a direct consequence of Lemma \ref{le:preparation_levels_triangles}. Note that the $ k $-level is never dominant, since $ \boldsymbol{v} $ is a vertex of $ \mathcal{N}(\omega;\mathcal{A},\gamma) $ which is not a vertex of $ \operatorname{Dom}\mathcal{N} $ and the property of maximality of dominant levels holds.

Now, take a positively convex polygon  ${\mathcal N}$ of ${\mathbb R}^2_{\geq 0}$ such that all the vertices have integer ordinates, except maybe the vertex of abscissa $0$. Consider $\delta$ and $\rho_\delta$ as in the statement of Lemma \ref{le:epsilon}. Note that
$$
\gamma/\nu(z)\geq \rho_\delta/\delta \quad \text{and} \quad \gamma \geq \rho{_\delta} \ .
$$
Suppose also that
\begin{enumerate}
	\item Either $(\rho_\delta,0)$ is a vertex of ${\mathcal N}$ or $(\rho_\delta,0)\notin {\mathcal N}$;
	\item Either $(0,\rho_{\delta}/ \delta)$ is a vertex of ${\mathcal N}$ or $(0,\rho_\delta/\delta)\notin {\mathcal N}$;
	\item The points $(0,\gamma/\nu(z))$ and $(\gamma,0)$ are in $\mathcal N$.
\end{enumerate}
For any vertex $\boldsymbol{v}$ of ${\mathcal P}$ denote by $\delta_{l}(\boldsymbol{v})$ and by $\delta_{r}(\boldsymbol{v})$ the real numbers such that $ -1/\delta_{l}(\boldsymbol{v}) $ and $ -1/\delta_{r}(\boldsymbol{v}) $ are the slopes of the two sides of $\mathcal N$ throughout $ \boldsymbol{v} $, with $0\leq \delta_{l}(\boldsymbol{v})<\delta_{r}(\boldsymbol{v})\leq +\infty$. The following Lemma has an elementary proof:
\begin{lem}\label{le:difference_slopes_bound}
	In the above situation, there is a constant $K_\epsilon\geq 0$, not depending on the particular polygon, such that for any $\mathcal N$ with
	$$
	{\mathcal N}\not\subset H_{\delta}^+(\rho_\delta-\epsilon) \ ,
	$$
	there is at least one vertex $\boldsymbol{v}$ of $\mathcal N$ with $\boldsymbol{v}\in{\mathcal N}\cap H_{\delta}^-(\rho_\delta)$ such that
	$$
	\delta_{\ell}({\boldsymbol{v}})<\delta_{r}({\boldsymbol{v}})-K_{\epsilon} \ .
	$$
\end{lem}
\begin{proof}
	Take $ k_{\epsilon} \in \mathbb{R} $ such that
	$$
	k_{\epsilon} < \frac{2 \delta^2 \epsilon }{\rho_{\delta} (\rho_{\delta}+\delta)}  \ .
	$$
	We assert that this constant satisfies the conditions required in the lemma. Suppose that it is false. Consider a polygon $ {\mathcal N}\not\subset H_{\delta}^+(\rho_\delta-\epsilon) $. Since $ \mathcal{N} $ is not contained in $ H_{\delta}^+(\rho_\delta-\epsilon) $ there must be a vertex $ \boldsymbol{v} $ of $ \mathcal{N} $ such that $ \boldsymbol{v} \notin H_{\delta}^+(\rho_\delta-\epsilon) $. If our assumption is false, for every vertex $ \boldsymbol{v}' $ in $ H_{\delta}^- (\rho_\delta) $ (and in particular for $ \boldsymbol{v} $) we have
	$$ \delta_{\ell} (\boldsymbol{v}') \geq \delta_{r} (\boldsymbol{v}') - K_{\epsilon} \ . $$
	But this condition gives that at least one of the points $(0,\rho_{\delta}/ \delta)$ or $(\rho_\delta,0)$ are interior points of $ \mathcal{N} $ in contradiction with the hypothesis about $ \mathcal{N} $.
\end{proof}
\noindent Now, let us apply Lemma \ref{le:difference_slopes_bound} to prove Lemma \ref{le:epsilon}. Assume that
$$
\mathcal{N}(\omega;\mathcal{A},\gamma) \not\subset H_{\delta}^{+}(\rho_0 - \epsilon) \ .
$$
Take one of the vertices ${\boldsymbol v}=(\alpha_k,k)$ of $ \mathcal{N}(\omega;\mathcal{A},\gamma) $ given by  Lemma \ref{le:difference_slopes_bound}.  Note that ${\boldsymbol v}$ is not a vertex of $ \operatorname{Dom}\mathcal{N} $, hence we can apply the reduction of vertices. We obtain that
$$
\mathcal{N}(\omega;\mathcal{A}',\gamma) \subset {\mathcal N}' = {\mathcal N}\setminus \text{interior of }{\mathcal T} \ ,
$$
where ${\mathcal T}$ is the triangle having vertices $\boldsymbol{v}$, $(\alpha_{k-1},k-1)$ and $(\alpha_{k+1},k+1) $. Moreover
$$
\operatorname{area}({\mathcal T})=\frac{\delta_{r}(\boldsymbol{v})-\delta_{l}(\boldsymbol{v})}{2}> K_\epsilon/2 \ .
$$
We deduce that the area of
$$
\mathcal{N}(\omega;\mathcal{A},\gamma) \cap H_{\delta}^-(\rho_\delta)
$$
decreases strictly at least the amount $K_\epsilon/2$. This cannot be repeated infinitely many times, thus we obtain the condition stated in Lemma \ref{le:epsilon}.

\subsection{Elimination of recessive vertices}
In this section we complete the proof of Theorem \ref{th:gammma_preparation}. In view of Remark \ref{re:gamma_preparation_polygons} it is enough with determine a $ \ell $-nested transformation $ \mathcal{A} \rightarrow \mathcal{B} $ such that $\mathcal{N}(\omega;\mathcal{B},\gamma) = \operatorname{Dom}\mathcal{N}(\omega;\mathcal{B},\gamma) = \operatorname{Dom}\mathcal{N} $ and then use Lemma \ref{le:push_right}.

We assume that the properties listed in Proposition \ref{pr:preparation_reductions} are satisfied. Note that this reductions and Lemma \ref{le:epsilon} guarantee that for levels $ \omega_k $ which are not $ \tau_k $-final we have
\begin{equation}\label{eq:non_tau_final_levels}
	\alpha_k \leq \beta_k = \delta_k \leq \tau_k \leq \phi_k \ ,
\end{equation}
where we recall that $ \beta_k=\nu_{\mathcal{A}}(\omega_k) $, $ \delta_k=\nu_{\mathcal{A}}(\eta_k) $ and $ \phi_k=\nu_{\mathcal{A}}(f_k) $, and $ \alpha_k $ and $ \tau_k $ are the minimum values such that $ (\alpha_k,k) $ and $ (\tau_k,k) $ belong to $ \mathcal{N} $ and $ \textrm{Dom}\mathcal{N} $ respectively.

Let us state the induction property we are going to use:
\begin{quote}
	``$ \mathcal{P}_h (\omega;\mathcal{A};\gamma) $: for all $ 0 \leq k \leq h $ the $ k $-level $ \omega_k $ is $ \tau_k $-final.''
\end{quote}
Note that $\omega$ is $\gamma$-prepared if and only if $ \mathcal{P}_{h}(\omega;\mathcal{A};\gamma) $ is true for all $ h \leq \gamma/\nu(z) $.

The starting property $ \mathcal{P}_0 (\omega;\mathcal{A};\gamma) $ is true, since $ \tau_0 \leq \gamma $ hence $ \omega_0 $ is $ \tau_0 $-final. Moreover, the stability properties under $\ell$-nested transformations give that
$$
 \mathcal{P}_h (\omega;\mathcal{A};\gamma) \Rightarrow  \mathcal{P}_h (\omega;\mathcal{A}';\gamma) \ ,
$$
for any $\ell$-nested transformation $ \mathcal{A} \rightarrow \mathcal{A}'$.

Thus, in order to complete the proof of Theorem \ref{th:gammma_preparation} we have to show the following statement:
\begin{quote}
	``For a given integer $ h $ with $ 1 \leq h \leq \gamma/ \nu(z) $, if $ \mathcal{P}_{h-1}(\omega;\mathcal{A};\gamma) $ is true, there is a $\ell$-nested transformation $ \mathcal{A} \rightarrow \mathcal{A}' $ such that  $ \mathcal{P}_{h}(\omega;\mathcal{A}';\gamma) $ is true.''
\end{quote}
We suppose that $ \mathcal{P}_h (\omega;\mathcal{A};\gamma) $ is not true. Note that the level $\omega_h$ is not $(\gamma-h\nu(z))$-final dominant, otherwise it would be $\tau_h$-final.

Let $ L_{\delta}(\rho) $ be the line passing through the point $ (\tau_h,h) $ and containing a side of $ \mathrm{Dom}\mathcal{N} $. We have two possibilities:
\begin{itemize}
	\item [a)] for every $ k < h $ the level $ \omega_k $ is $ (\rho - k \delta) $-final recessive;
	\item [b)] there is at least on index $ k < h $ such that $ \omega_k $ is $ (\rho - k \delta) $-final dominant.
\end{itemize}
Note that in the first case we must have $ \rho = \gamma $. Let us refer to the case a) as the ``totally recessive case'' and to the case b) as "dominant base point case".

\subsubsection{Totally recessive case}
In this case we will use Lemma \ref{le:epsilon} in order to bring $ \mathcal{N} $ close enough to $ L_{\delta}(\rho) $ such that the property of reduction of vertices applied to $ (\beta_h,h) $ gives us the desired result.

For every $ k < h $ we have that $ \tau_k = \rho - k \delta $, thus the real number $ \epsilon $ given by
$$
\epsilon := \min \{ \beta_k - \tau_k \ |\ k=0,1,\dots,h-1 \}
$$
is strictly positive. Let $ \mathcal{A} \rightarrow \mathcal{A}^* $ be a $ \ell $-nested transformation given by Lemma \ref{le:epsilon} with respect to $ L_{\delta}(\rho) $ and $ \epsilon $. Now, the property of reduction of vertices applied to the vertex $ (\beta_h,h) $ of $ \mathcal{N}(\omega;\mathcal{A}^*,\gamma) $ gives a $ \ell $-nested transformation $ \mathcal{A}^* \rightarrow \mathcal{A}' $ such that $ \mathcal{P}_h (\omega;\mathcal{A}';\gamma) $ is true.

\subsubsection{Dominant base point case}\label{subse:dominant_vertex}
Let $ b $ be the lowest index such that $ \omega_b $ is $ (\rho - b \delta) $-final dominant. Note that if $ \rho < \gamma $ the point $ (\tau_b,b) $ is a vertex of $ \mathrm{Dom}\mathcal{N} $ but in the case $ \rho = \gamma $ it is not necessarily a vertex. Taking into account these possibilities, in the case that $ b \geq 1 $ we define a value $ \epsilon_1 >0 $ as
\begin{equation*}
	\epsilon_1 := \left\{
		\begin{array}{lll}
			\tau_{b-1} - (\tau_b + \delta) & \text{if} & \rho < \gamma \ ; \\
			\min \{ \beta_k - (\tau_k + (h-k)\delta) \ |\ k=0,1,\dots,h-1 \} & \text{if} & \rho = \gamma \ .
		\end{array} \right.
\end{equation*}
In the first case $ \epsilon_1 $ is the distance between $ L_{\delta}(\rho) $ and $ \mathrm{Dom}\mathcal{N} $ over the horizontal line at height $ b-1 $. In the second case, $ \beta_k - (\tau_k + (h-k)\delta) $ is the distance between the line $ L_{\delta}(\rho) $ and the point $ (\beta_k,k) $ of $ \mathrm{DomCl}(\omega;\mathcal{A}) $ over the horizontal line at height $ k $, so $ \epsilon_1 $ is the minimum among such distances.

Since $ \omega_b $ is $ \tau_b $-final dominant we have
$$
	\omega_b = \boldsymbol{x}^{\boldsymbol{p}_b} \tilde{\omega}_b + \bar{\omega}_b \ ,
$$
where $ \nu_{\mathcal{A}}(\bar{\omega}_b) > \nu({{\boldsymbol x}^{{\boldsymbol p}_b}})=\tau_b $ and $ \tilde{\omega}_b $ is log-elementary. If we write $ \omega_b = \eta_b + f_b \frac{dz}{z} $ we have
$$
\eta_b = \boldsymbol{x}^{\boldsymbol{p}_b} \tilde{\eta}_b + \bar{\eta}_b \quad \text{and} \quad f_b = \tilde{f}_b \boldsymbol{x}^{\boldsymbol{p}_b} + \bar{f}_b \ ,
$$
where $ \tilde{\eta}_b \in N_{\mathcal{A}}^{\ell} $ is log-elementary or $ \tilde{\eta}_b \equiv 0 $ and $ \nu_{\mathcal{A}}(\bar{\eta}_b) \geq \epsilon $, $ \tilde{f}_b \in R_{\mathcal{A}}^{\ell} $ is a unit or $ \tilde{f}_b \equiv 0 $ $ \nu_{\mathcal{A}}(\bar{f}_b) \geq \epsilon $, and $ (\tilde{\eta}_b,\tilde{f}_b) \neq (0,0) $. Moreover, we can assume that $ \tilde{f}_b $ is a constant, just by taking the non-constant terms and considering them as terms of $ \bar{f}_b $, and using Lemma \ref{le:push_right} if necessary. So from now on we assume that
$$
\omega_b = \boldsymbol{x}^{\boldsymbol{p}_b} \left( \tilde{\eta_b} + \mu \frac{dz}{z} \right) + \bar{\omega}_b \ .
$$
Let $ \epsilon_2 $ be the positive value defined by
$$
\epsilon_2 := \nu_{\mathcal{A}}(\bar{\omega}_b) - \tau_b \ .
$$
Consider the value
\begin{equation*}
	\epsilon := \left\{
		\begin{array}{lll}
			\min \{ \epsilon_1,\epsilon_2\} & \text{if} & b \geq 1 \ ; \\
			\epsilon_2 & \text{if} & b = 0 \ .
		\end{array} \right.
\end{equation*}
After performing a $\ell$-nested transformation given by Lemma \ref{le:epsilon} with respect to $L_{\delta}(\rho)$ and $ \epsilon $ if necessary, we can assume that:
\begin{itemize}
	\item[a)] $\beta_k> \tau_h-(k-h) \delta - \epsilon $, for any $h\leq k\leq  \gamma / \nu(z) $.
	\item[b)] $\beta_k\geq \tau_k = \tau_h-(k-h) \delta $, for any $b\leq k\leq h-1$.
	\item[c)] $\beta_k> \tau_h-(k-h) \delta + \epsilon $, for any $0\leq k\leq b-1$.
\end{itemize}
Recall that $ \nu_{\mathcal{A}}(\Delta_{h+b} ) \geq 2 \gamma $, where this $ 2 $-form is given by
$$ \Delta_{h+b} = \sum_{i+j=h+b} \big( j \eta_j \wedge \eta_i +f_i d \eta_j + \eta_i \wedge d f_j \big) \ . $$
We can write
\begin{eqnarray*}
	(h-b)\eta_h \wedge \eta_b +  f_h d \eta_b + \eta_b \wedge d f_h+ f_b d \eta_h + \eta_h \wedge d f_b
 =\\
 =  \Delta_{h+b}  - \sum_{\stackrel{i+j=h+b}{i,j \neq h,b}} \big( j \eta_j \wedge \eta_i +f_i d \eta_j + \eta_i \wedge d f_j \big)
\end{eqnarray*}
In view of properties a), b) and c), and taking into account that $ \tau_h + \tau_b < 2 \gamma $, we deduce that
\begin{equation}\label{eq:truncations_1}
	\nu_{\mathcal{A}}\big( (h-b) \eta_h \wedge \eta_b +  f_h \, d \eta_b + \eta_b \wedge d f_h+ f_b \, d \eta_h + \eta_h \wedge d f_b \big) \geq \tau_h+\tau_b \ .
\end{equation}
By \eqref{eq:non_tau_final_levels} we know that $ \nu_{\mathcal{A}}(f_h) \geq \tau_h $. On the other hand, since $ \nu_{\mathcal{A}}(\eta_b) \geq \tau_b $ we have that $ \nu_{\mathcal{A}}(d \eta_b) \geq \tau_b $. We deduce that
\begin{equation}\label{eq:truncations_2}
	\nu_{\mathcal{A}}\left( f_h \, d \eta_b + \eta_b \wedge d f_h  \right) \geq \tau_h+\tau_b \ .
\end{equation}
From \eqref{eq:truncations_1} and \eqref{eq:truncations_2} we derive that
\begin{equation}\label{eq:truncations_3}
	\nu_{\mathcal{A}}\big( (h-b)\eta_h \wedge \eta_b +   f_b \, d \eta_h + \eta_h \wedge d f_b \big) \geq \tau_h+\tau_b \ .
\end{equation}
By property a) we have
$$ \nu_{\mathcal{A}} ( \eta_h \wedge \bar{\eta}_b ) \geq \tau_b+ \tau_h \quad \text{and} \quad  \nu_{\mathcal{A}} ( \bar{f}_b \, d \eta_h ) \geq \tau_b+ \tau_h \ , $$
so from Equation \eqref{eq:truncations_3} we deduce
\begin{equation*}
	\nu_{\mathcal{A}} \left( (h-b) \, \eta_h \wedge \boldsymbol{x}^{\boldsymbol{p}_b} \tilde{\eta}_b + \mu \, \boldsymbol{x}^{\boldsymbol{p}_b} \, d \eta_h \right) \geq \tau_h+\tau_b \ .
\end{equation*}
We can rewrite the above expression as
\begin{equation*}
	\nu_{\mathcal{A}} \left( \boldsymbol{x}^{\boldsymbol{p}_b} \left( \eta_h \wedge \left[ (h-b) \tilde{\eta}_b + \mu \frac{d\boldsymbol{x}^{\boldsymbol{p}_b}}{\boldsymbol{x}^{\boldsymbol{p}_b}} \right] + \mu \, d \eta_h  \right) \right) \geq \tau_h+\tau_b \ .
\end{equation*}
Dividing by ${\boldsymbol x}^{{\boldsymbol p}_b}$ we obtain
\begin{equation}\label{eq:truncations_4}
	\nu_{\mathcal{A}} \left( \eta_h \wedge \left[ (h-b) \tilde{\eta}_b + \mu \frac{d\boldsymbol{x}^{\boldsymbol{p}_b}}{\boldsymbol{x}^{\boldsymbol{p}_b}} \right] + \mu \, d \eta_h \right) \geq \tau_h \ .
\end{equation}
Let us denote by $ \sigma \in N_{\mathcal{A}}^{\ell} $ the term in the brackets
$$
\sigma := (h-b) \tilde{\eta}_b + \mu \frac{d\boldsymbol{x}^{\boldsymbol{p}_b}}{\boldsymbol{x}^{\boldsymbol{p}_b}} \ .
$$
Now we study separately two cases depending on whether or not $ \sigma $ is log-elementary.
\paragraph{a) $ \sigma $ is log-elementary.}
We study first two particular cases first and then we treat the general situation. The first particular case is $ \mu = 0 $ and the second one is $ \tilde{\eta}_b = 0$.

\subparagraph{ a1) Case $ \mu = 0 $.} In this situation we have that $ \sigma = (h-b) \tilde{\eta}_b $ so Equation \eqref{eq:truncations_4} gives
\begin{equation}\label{eq:truncations_5}
	\nu_{\mathcal{A}} \left( \eta_h \wedge \tilde{\eta}_b \right) \geq \tau_h \ .
\end{equation}
We need the following lemma:
\begin{lem}[Truncated proportionality]\label{le:truncated_proportionality}
	Let $\tilde\eta\in N_{\mathcal A}^{\ell}$ be a log-elementary $1$-form. Given $\theta\in N_{\mathcal A}^{\ell}$ such that $ \nu_{\mathcal{A}} (\theta \wedge \tilde\eta ) \geq \lambda $, there is $ g \in R_{\mathcal A}^{\ell} $ and $ \bar\theta \in N_{\mathcal A}^{\ell} $ with $ \nu_{\mathcal{A}}(\bar\theta) \geq \lambda $ such that
	$$
	\theta = g \, \tilde{\eta} + \bar{\theta} \ .
	$$
	\begin{proof}
		Let us write
		$$ \tilde{\eta} = \sum a_i \frac{dx_i}{x_i} + \sum b_j dy_j \quad \text{and} \quad \theta = \sum a'_i \frac{dx_i}{x_i} + \sum b'_j dy_j \ ,$$
		where we suppose without lost of generality that $ a_1 $ is a unit. The coefficients of the $ 2 $-form $ \theta \wedge \tilde\eta $ are given by the minors of the matrix
		$$
		\left(
		\begin{array}{cccccc}
			a_1	& \dots & a_r & b_1 & \dots & b_s \\
			a'_1 & \dots & a'_r & b'_1 & \dots & b'_s  \\
		\end{array}
		\right) \ .
		$$
		Since $ \nu_{\mathcal{A}}(\tilde{\eta} \wedge \theta) \geq \lambda $ we have that
		$$ \nu_{\mathcal{A}}( a_1 a'_i - a_i a'_1 ) \geq \lambda \ ,\quad 2\leq i \leq r  \ ,$$
		and
		$$ \nu_{\mathcal{A}}( a_1 b'_j - b_j a'_1 ) \geq \lambda \ ,\quad 1 \leq j \leq n-r \ . $$
		Thus we have
		$$ a'_i = \frac{a'_1}{a_1}  a_i + \bar{a}_i \ ,\quad  \nu_{\mathcal{A}}( \bar{a}_i ) \geq \lambda \ ,\quad2\leq i \leq r \ , $$
		and
		$$ b'_j = \frac{a'_1}{a_1}  b_j + \bar{b}_j \ ,\quad  \nu_{\mathcal{A}}( \bar{b}_j ) \geq \lambda  \ ,\quad 1 \leq j \leq n-r \ . $$
		Therefore we can write
		$$ \theta = g \, \tilde{\eta} + \bar\theta \ , $$
		where
		$$ g = \frac{a'_1}{a_1} \quad \text{and} \quad \bar\theta = \sum \bar{a}_i \frac{dx_i}{x_i} + \sum \bar{b}_j dy_j \ , $$
		and by construction we have $ \nu_{\mathcal{A}}(\bar\theta) \geq \lambda $.
	\end{proof}
\end{lem}
\begin{rem}
	We have detailed a direct proof of Lemma \ref{le:truncated_proportionality}, but it can be obtained as a consequence of the de Rham-Saito Lemma \cite{Sai}. 
\end{rem}
\noindent From Equation \eqref{eq:truncations_5} and Lemma \ref{le:truncated_proportionality} we conclude that there are $ g \in R_{\mathcal{A}}^{\ell} $ and $ \bar{\eta}_h \in N_{\mathcal{A}}^{\ell} $ such that
$$ \eta_h=g \sigma + \bar{\eta}_h \ , $$
where $ \nu_{\mathcal{A}} ( \bar{\eta}_h ) \geq \tau_h $. This expression is stable under further $\ell$-nested transformations, thus we can assume that $ g $ is $ \tau_h $-final. If $ \nu_{\mathcal{A}}(g) < \tau_h $ the level $ \omega_h $ would be $ \tau_h $-final dominant with value 
$ \nu_{\mathcal{A}}(\omega_h) = \nu_{\mathcal{A}}(g) < \tau_h $, in contradiction with the maximality of dominant levels assumption. So we have $ \nu_{\mathcal{A}}(g) \geq \tau_h $, hence $ \nu_{\mathcal{A}}(\omega_h) \geq \tau_h $. Since $ \omega_h $ cannot be dominant, applying Lemma \ref{le:push_right} if necessary we obtain $ \nu_{\mathcal{A}}(\omega_h) > \tau_h $, that is, $ \omega_h $ is $ \tau_h $-final recessive.

\subparagraph{ a2) Case $ \tilde{\eta}_b = 0 $.} In this situation we have $ \mu \neq 0 $. In this situation we have that
$$
\sigma = \mu \frac{d\boldsymbol{x}^{\boldsymbol{p}_b}}{\boldsymbol{x}^{\boldsymbol{p}_b}} \ ,
$$
hence Equation \eqref{eq:truncations_4} gives
\begin{equation}\label{eq:truncations_6}
	\nu_{\mathcal{A}} \left( \eta_h \wedge \frac{d\boldsymbol{x}^{\boldsymbol{p}_b}}{\boldsymbol{x}^{\boldsymbol{p}_b}} + d \eta_h \right) \geq \tau_h \ .
\end{equation}
Let us write $ \eta_h $ as a series in the independent variables
$$
	\eta_h = \sum_{I} \boldsymbol{x}^{I} \eta_{h,I} \ ,
$$
where the coefficients of the $ 1 $-forms $ \eta_{h,I} \in N_{\mathcal{A}}^{\ell} $ are series in the variables $ \boldsymbol{y} $. Let us denote $ \eta_h = \check{\eta}_h + \bar{\eta}_h $ where
\begin{equation}\label{eq:explicit_truncation}
	\check{\eta}_h  = \sum_{\nu(\boldsymbol{x}^{I}) < \tau_h } \boldsymbol{x}^{I} \eta_{h,I} \ .
\end{equation}
Since $ \nu_{\mathcal{A}}(\bar{\eta}_h) \geq \tau_h $, from Equation \eqref{eq:truncations_6} we have that
\begin{equation}\label{eq:truncations_7}
	\nu_{\mathcal{A}} \left( \check{\eta}_h \wedge \frac{d\boldsymbol{x}^{\boldsymbol{p}_b}}{\boldsymbol{x}^{\boldsymbol{p}_b}} + d \check{\eta}_h \right) \geq \tau_h \ .
\end{equation}
Note that this expression is homogeneous with respect to $ \boldsymbol{x} $, it means,
$$
\check{\eta}_h \wedge \frac{d\boldsymbol{x}^{\boldsymbol{p}_b}}{\boldsymbol{x}^{\boldsymbol{p}_b}} + d \check{\eta}_h = \sum_{\nu(\boldsymbol{x}^{I}) < \tau_h } \boldsymbol{x}^{I} \left( \eta_{h,I} \wedge \frac{d\boldsymbol{x}^{\boldsymbol{p}_b}}{\boldsymbol{x}^{\boldsymbol{p}_b}} + \frac{d\boldsymbol{x}^{\boldsymbol{I}}}{\boldsymbol{x}^{\boldsymbol{I}}} \wedge \eta_{h,I} + d \eta_{h,I} \right) \ .
$$
Due to this homogeneity we have that Equation \eqref{eq:truncations_7} is equivalent to
\begin{equation}\label{eq:truncations_8}
	\check{\eta}_h \wedge \frac{d\boldsymbol{x}^{\boldsymbol{p}_b}}{\boldsymbol{x}^{\boldsymbol{p}_b}} + d \check{\eta}_h = 0 \ .
\end{equation}
Multiplying by $ \check{\eta}_h $ the above expression we deduce that $ \check{\eta}_h $ is integrable. By the induction statement $ T_3(\ell) $ there is a $ \ell $-nested transformation $ \mathcal{A} \rightarrow \mathcal{A}' $ such that $ \check{\eta}_h $ is $ \tau_h $-final. If $ \nu_{\mathcal{A}'}(\check{\eta}_h) < \tau_h $ the level $ \omega_h $ would be $ \tau_h $-final dominant with value 
$ \nu_{\mathcal{A}'}(\omega_h) = \nu_{\mathcal{A}'}(\check{\eta}_h) < \tau_h $, in contradiction with the maximality of dominant levels assumption. So we have $ \nu_{\mathcal{A}'}(\check{\eta}_h) \geq \tau_h $, hence $ \nu_{\mathcal{A}}(\omega_h) \geq \tau_h $. Since $ \omega_h $ cannot be dominant, applying Lemma \ref{le:push_right} if necessary we obtain $ \nu_{\mathcal{A}'}(\omega_h) > \tau_h $, that is, $ \omega_h $ is $ \tau_h $-final recessive.
\subparagraph{ a3) General case.} Let us write
$$
\sigma = (h-b) \tilde{\eta}_b + \mu \frac{d\boldsymbol{x}^{\boldsymbol{p}_b}}{\boldsymbol{x}^{\boldsymbol{p}_b}} =  \sum_{i=1}^r \lambda_i \frac{dx_i}{x_i} + \sigma^* \ ,
$$
where $ \boldsymbol{\lambda} \in k^{r} \setminus \{ \boldsymbol{0} \} $ and $ \sigma^* \in N_{\mathcal{A}}^{\ell} $ is not log-elementary. Suppose that we perform a $ \ell $-nested transformation $ \mathcal{A} \rightarrow \mathcal{A}' $. We obtain new coordinates $ (\boldsymbol{x}',\boldsymbol{y}') $ such that for $ i=1,2,\dots,r $ we have
$$
x_i = \boldsymbol{x}'^{\boldsymbol{\alpha}_i} U_i \ ,
$$
where $ \boldsymbol{\alpha}_i \in \mathbb{Z}^{r}_{\geq 0} \setminus \{ \boldsymbol{0} \} $ and $ U_i \in R_{\mathcal{A}'}^{\ell} $ is a unit. We have that
$$
\sum_{i=1}^r \lambda_i \frac{dx_i}{x_i} = \sum_{i=1}^r \lambda_i \frac{d(\boldsymbol{x}'^{\boldsymbol{\alpha}_i} U_i)}{\boldsymbol{x}'^{\boldsymbol{\alpha}_i} U_i} = \sum_{i=1}^r \lambda_i \frac{d\boldsymbol{x}'^{\boldsymbol{\alpha}_i} }{\boldsymbol{x}'^{\boldsymbol{\alpha}_i} } + \sum_{i=1}^r \lambda_i U_i^{-1} d U_i \ .
$$
We can write
$$
\sum_{i=1}^r \lambda_i \frac{d\boldsymbol{x}'^{\boldsymbol{\alpha}_i} }{\boldsymbol{x}'^{\boldsymbol{\alpha}_i} } = \sum_{i=1}^r \lambda'_i \frac{dx'_i}{x'_i}
$$
where $ \boldsymbol{\lambda}' \in k^{r} \setminus \{ \boldsymbol{0} \} $. On the other hand, we have that
$$
d \left( \sum_{i=1}^r \lambda_i U_i^{-1} d U_i \right) = 0 \ , 
$$
so, by Poincare's Lemma we know that
$$
 \sum_{i=1}^r \lambda_i U_i^{-1} d U_i = d G
$$
for certain formal function $ G \in R_{\mathcal{A}'}^{\ell} $. We have just proved that after performing a $ \ell $-nested transformation $ \mathcal{A} \rightarrow \mathcal{A}' $ the $ 1 $-form $ \sigma $ can be written as
$$
\sigma = \sum_{i=1}^r \lambda'_i \frac{dx'_i}{x'_i} + dG + \sigma^* \ .
$$
In view of these considerations and using if necessary Lemmas \ref{le:push_right} and \ref{le:epsilon} we can assume that in $ \mathcal{A} $ we have
$$
\sigma = \sum_{i=1}^r \lambda_i \frac{dx_i}{x_i} + dG + \sigma^* \ ,
$$
where $ \nu_{\mathcal{A}}(\sigma^*) > \tau_h - \nu_{\mathcal{A}}(\eta_h) > 0 $. With this assumption we have that Equation \eqref{eq:truncations_4} gives
\begin{equation}\label{eq:truncations_9}
	\nu_{\mathcal{A}} \left( \eta_h \wedge \left[ \sum_{i=1}^r \lambda_i \frac{dx_i}{x_i} + dG \right] + \mu \, d \eta_h \right) \geq \tau_h \ .
\end{equation}
Moreover, we can assume that $ \lambda_{i_0} = 1 $ for certain index $ 1 \leq i_0 \leq r $. Let us consider the elements $ x^*_1 , x^*_2 , \dots, x^*_r \in R_{\mathcal{A}}^{\ell} $ defined by
$ x^*_i := x_i $ if $ i \neq i_0 $ and $ x^*_{i_0} = x_{i_0} \exp(G) $. Note that we have
$$
\sum_{i=1}^r \lambda_i \frac{dx_i}{x_i} + dG = \sum_{i=1}^r \lambda_i \frac{dx^*_i}{x^*_i} \ .
$$
We have that $ x^*_1 , x^*_2 , \dots, x^*_r,y_1,y_2,\dots,y_{\ell} $ are a regular system of parameters of $ R_{\mathcal{A}}^{\ell} $. Let us consider a new ``explicit value'' $ \nu_{\mathcal{A}}^* $, defined exactly as $ \nu_{\mathcal{A}} $ but considering the power series expansions with respect to the parameters $ \boldsymbol{x}^* $ instead of $ \boldsymbol{x} $. For every exponent $ \boldsymbol{q} \in \mathbb{Z}^{r}_{\geq0} $ it follows that
$$
\nu_{\mathcal{A}}({\boldsymbol{x}^*}^{\boldsymbol{q}}) = \nu(\boldsymbol{x^q}) \ ,
$$
so we have that $ \nu_{\mathcal{A}}^* \equiv \nu_{\mathcal{A}} $. It means, for any object $ \psi $ (formal function or $ p $-form) we have that $ \nu_{\mathcal{A}}^* (\psi) = \nu_{\mathcal{A}} (\psi) $.

Let us write $ \eta_h $ as a power series of the elements $ x^*_1 , x^*_2 , \dots, x^*_r $
$$
\eta_h = \sum_{I} {\boldsymbol{x}^*}^{I} \eta_{h,I}
$$
where the coefficients of the $ 1 $-forms $ \eta_{h,I} \in N_{\mathcal{A}}^{\ell} $ are series in the variables $ \boldsymbol{y} $. Let us denote $ \eta_h = \check{\eta}_h + \bar{\eta}_h $ where
$$
\check{\eta}_h = \sum_{ \nu_{\mathcal{A}}^*({\boldsymbol{x}^*}^{I})  < \tau_h} {\boldsymbol{x}^*}^{I} \eta_{h,I} \ .
$$
We have that $ \check{\eta}_h $ and $ \bar{\eta}_h $ belong to $ N_{\mathcal{A}}^{\ell} $ and $ \nu_{\mathcal{A}}^*(\bar{\eta}_h) \geq \tau_h $. From Equation \eqref{eq:truncations_9} we obtain
$$
	\nu_{\mathcal{A}}^{*} \left( \check{\eta}_h \wedge \sum_{i=1}^r \lambda_i \frac{dx^*_i}{x^*_i} + \mu \, d \check{\eta}_h \right) \geq \tau_h \ .
$$
This expression is homogeneous with respect to $ \boldsymbol{x}^* $ so it is equivalent to
$$
\check{\eta}_h \wedge \sum_{i=1}^r \lambda_i \frac{dx^*_i}{x^*_i} + \mu \, d \check{\eta}_h = 0 \ .
$$
Multiplying this last expression by $ \check{\eta}_h $ we obtain that $ \check{\eta}_h $ is an integrable $ 1 $-form. Recall that both $ \check{\eta}_h $ and $ \bar{\eta}_h $ belong to $ N_{\mathcal{A}}^{\ell} $ and that $ \nu_{\mathcal{A}} (\bar{\eta}_h) \geq \tau_h $. We are in the same situation that in the particular case $ \tilde{\eta}_b \equiv 0 $ detailed previously, so we can end as we did in that case.
\begin{rem}
	Note that we have not performed non-algebraic operations. We have used the formal variables $ \boldsymbol{x}^* $ only for divide in two parts the $ 1 $-form $ \eta_h $.
\end{rem}

\paragraph{b) $ \sigma $ is not log-elementary.}
First of all, note that $ \mu = 0 $ implies that $ \sigma $ is log-elementary, so in this case we have $ \mu \neq 0 $.

In this case, using Lemma \ref{le:push_right} if necessary, we can assume that $ \nu_{\mathcal{A}}(\sigma) > 0 $. In fact, we can suppose that $ \nu_{\mathcal{A}}(\sigma) > \tau_h - \nu_{\mathcal{A}}(\eta_h) $, otherwise we use Lemma \ref{le:epsilon}. With these assumptions Equation \eqref{eq:truncations_4} gives
$$
\nu_{\mathcal{A}}(d \eta_h) \geq \tau_h \ .
$$
Writing $ \eta_h = \check{\eta}_h + \bar{\eta_h} $ as we did in Equation \eqref{eq:explicit_truncation}, where we recall that $ \nu_{\mathcal{A}}(\bar{\eta_h}) \geq \tau_h $, we obtain that
$$
\nu_{\mathcal{A}}(d \check{\eta}_h) \geq \tau_h \ ,
$$
and again due to the homogeneity with respect to the variables $ \boldsymbol{x} $ we have that
$$
d \check{\eta}_h = 0 \ .
$$
We have that $ \check{\eta}_h $ is integrable (indeed, following Poincare's Lemma it is the differential of a function), and we conclude this case as the previous one.

\section{Getting $ \gamma $-final forms}\label{ch:getting_final_forms}
In this chapter we complete the proof of the induction step
$$ T_3(\ell) \Longrightarrow T_3(\ell+1) $$
started in Chapter \ref{ch:truncated_preparation}, thus we also end the proof of Theorem \ref{th:gamma_finalization_1_forms}.

Let $ \mathcal{A} = \big( \mathcal{O}, (\boldsymbol{x},\boldsymbol{y}) \big) $ be a parameterized regular local model for $ K,\nu $. Fix an index $ \ell $, $ 0 \leq \ell \leq n-r- 1 $. Let us recall here the precise statement we want to prove:
\begin{quote}
	$ \boldsymbol{T_3(\ell + 1) }$: Given a $ 1 $-form $ \omega \in N_{\mathcal{A}}^{\ell+1} $ and a value $ \gamma \in \Gamma $, if $ \nu_{\mathcal{A}}(\omega \wedge d\omega) \geq 2 \gamma $ then there exists a $ (\ell+1) $-nested transformation $  \mathcal{A} \rightarrow \mathcal{B} $ such that $ \omega $ is $ \gamma $-final in $ \mathcal{B} $.
\end{quote}
So, during this chapter we fix a value $ \gamma \in \Gamma $ and consider a 1-form $\omega \in N_{\mathcal{A}}^{l+1}$ such that $ \nu_{\mathcal{A}}(\omega \wedge d\omega) \geq 2 \gamma $. Since we are working by induction on $\ell$, we assume that the statement $T_3(\ell)$ is true (hence $T_4(\ell)$ and $T_5(\ell)$ are also true).

As in the previous chapter we denote the dependent variables by $ \boldsymbol{y}=(y_1,y_2,\dots,y_{\ell}) $ and $ z = y_{\ell + 1} $.

\subsection{The critical height of a $ \gamma $-prepared $ 1 $-form}
In this section we assume that $ \omega \in N_{\mathcal{A}}^{\ell+1} $ is $ \gamma $-prepared. Recall that due to the $ \gamma $-prepared assumption in this situation we have that $ \mathcal{N}(\omega;\mathcal{A},\gamma) = \operatorname{Dom} \mathcal{N}(\omega;\mathcal{A},\gamma) $.

The \emph{critical value $ \delta(\omega;\mathcal{A},\gamma) $} is defined by
$$
	\delta(\omega;\mathcal{A},\gamma) := \min \left\{ \rho \ |\ L_{\nu(z)}(\rho) \cap \mathcal{N}(\omega;\mathcal{A},\gamma) \neq \emptyset \right\} \ .
$$
Note that $ \delta(\omega;\mathcal{A},\gamma) \leq \gamma $ since $ (0,\gamma) \in \mathcal{N}(\omega;\mathcal{A},\gamma) $. The critical value satisfies
$$  
	\delta(\omega;\mathcal{A},\gamma) = \min  \left\{ \beta_k(\omega;\mathcal{A}) + k \nu(z) \right\}_{k\geq 0} \cup \left\{ \gamma \right\} \,
$$
where we recall that $ \beta_k(\omega;\mathcal{A}) = \nu_{\mathcal{A}}(\omega_k) $ (see Section \ref{se:level_expansion}). Note that due to the $ \gamma $-preparation assumption in the above equality we can put $ \tau_k $ instead of $ \beta_k $. If no confusion arises we denote the critical value by $ \delta $. We study separately the cases $ \delta < \gamma $ and $ \delta = \gamma $

In the case $ \delta < \gamma $ we say that $ \mathcal{N}(\omega;\mathcal{A},\gamma) \cap L_{\nu(z)}(\delta) $ is \emph{the critical segment of $ \mathcal{N}(\omega;\mathcal{A},\gamma) $}. The \emph{critical height $ \chi(\omega;\mathcal{A},\gamma) $ of $ \mathcal{N}(\omega;\mathcal{A},\gamma)$} is the height of the upper endpoint of the critical segment. This integer number is our main control invariant. It satisfies
$$  0 \leq \chi(\omega;\mathcal{A},\gamma) \leq \frac{\delta}{\nu(z)} < \frac{\gamma}{\nu(z)} \ . $$
If no confusion arises we denote the critical height by $ \chi $. Note that we have
\begin{equation}\label{eq:critical_height_and_value}
	\delta =  \tau_{\chi} + \chi \, \nu(z) = \beta_{\chi} + \chi \, \nu(z)\ .
\end{equation}

Denote by $ \beta(\omega;\mathcal{A}) $ the explicit value $ \nu_{\mathcal{A}}(\omega) $ of $ \omega $ in $ \mathcal{A} $. Note that $ \beta(\omega;\mathcal{A}) $ is the minimum of the values $ \beta_k(\omega;\mathcal{A}) $. If $ \delta(\omega;\mathcal{A},\gamma) < \gamma $, from Equation \eqref{eq:critical_height_and_value} we derive that
\begin{equation}\label{eq:explicit_and_critical_value}
	\delta(\omega;\mathcal{A},\gamma) \geq \beta(\omega;\mathcal{A}) + \chi(\omega;\mathcal{A},\gamma) \nu(z) \ ,
\end{equation}
where we have equality if and only if $ \beta(\omega;\mathcal{A}) $ is the abscissa of the critical vertex. If no confusion arises we denote the explicit value of $ \omega $ by $ \beta $.
\begin{figure}[!ht]
	\centering
	\includegraphics[width=0.4\textwidth]{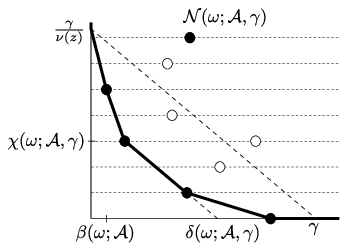}
	\caption{The explicit value, the critical value and the critical height}
	\label{fi:invariants_form}
\end{figure}

XXXXXXXXXXXXXXXXXXXXXXXXXXXXXXXXXXXXXXXXXXXXXXXXXXXX

XXXXXXXXXXXXXXXXXXXXXXXXXXXXXXXXXXXXXXXXXXXXXXXXXXXX

According to our definitions, all the levels $ \omega_k $ with explicit value $ \beta_k = \delta - k \nu(z) $ must be $ \beta_k $-final dominant. In particular, $ \omega_{\chi} $ is dominant with value $ \beta_{\chi} = \tau_{\chi} = \delta - \chi \nu(z) $. We have that
\begin{equation}\label{eq:critical_level}
	\omega_{\chi} = \boldsymbol{x}^{\boldsymbol{q}_{\chi}} \tilde{\omega}_{\chi} + \bar{\omega}_{\chi} \ , \quad \nu(\boldsymbol{x}^{\boldsymbol{q}_{\chi}}) = \tau_{\chi} \ ,
\end{equation}
where $ \tilde{\omega}_{\chi} $ is log-elementary and $ \nu_{\mathcal{A}}(\bar{\omega}_{\chi}) > \tau_{\chi} $.
	
Let $ \phi \in K $ be the $ (\ell+1) $-contact rational function $ \phi_{\ell+1} = z^{d} / \boldsymbol{x^p} $, where $ d $ is the ramification index $ d(z;\mathcal{A}) $ (see Section \ref{se:Puiseux's_packages}).
	
Now, consider a level $ \omega_k $ which gives a point $ (\beta_k,k) = (\tau_k,k) $ in the critical segment. Since $ \tau_k = \delta - k \nu(z) = \tau_{\chi} + (\chi - k)\nu(z) $, the index $ k $ must be of the form $ k = \chi - td $ for some integer $ 0 \leq t \leq \chi / d $. Since $ \omega $ is $ \gamma $-prepared we know that $ \omega_k $ is dominant, so it can be written as
$$ \omega_{k} = \boldsymbol{x}^{\boldsymbol{q}_{k}} \tilde{\omega}_{k} + \bar{\omega}_{k} \ , \quad \nu(\boldsymbol{x}^{\boldsymbol{q}_{k}}) = \tau_{k} \ . $$
	
Following Remark \ref{re:monomial_after_Puiseux_package}, we have that
$$ \boldsymbol{x}^{\boldsymbol{q}_{\chi - td}} = \boldsymbol{x}^{\boldsymbol{q}_{\chi} + t \boldsymbol{p}} $$
hence
\begin{equation}\label{eq:critical_segment_1}
	z^{\chi-td}\tilde{\omega}_{\chi - td} = \boldsymbol{x}^{\boldsymbol{q}_{\chi}} z^{\chi} \phi^{-t} \tilde{\omega}_{\chi -td} \ .
\end{equation}
Moreover, we can choose the forms $ \tilde{\omega}_{\chi - td} $ with coefficients not depending on the variables $ \boldsymbol{x} $. Write
$$ \tilde{\omega}_{\chi -td} = \sum_{i=1}^r \lambda_{t,i} \frac{dx_i}{x_i} + \mu_t \frac{dz}{z} + \omega^*_t \ , \quad (\boldsymbol{\lambda}_t,\mu_t)\in \mathbb{C}^{r+1} \setminus \{ \boldsymbol{0} \} \ , $$
where $ \omega^*_t $ is not log-elementary. Note that in order to simplify the expressions we have change the subindices. Denote by $ \sigma_t $ the closed form
\begin{equation}\label{eq:initial_part_critical_level}
	\sigma_t := \sum_{i=1}^r \lambda_{t,i} \frac{dx_i}{x_i} + \mu_t \frac{dz}{z} \ .
\end{equation}
Denote by $ M $ the integer part of $ \chi / d $. Let $ \omega_{\text{crit}} $ be the $ 1 $-form defined by
$$ \omega_{\text{crit}} := \boldsymbol{x}^{\boldsymbol{q}_{\chi}} \sum_{t = 0}^{M} z^t \boldsymbol{x}^{t\boldsymbol{p}} \sigma_t \ ,  $$
where we set $ \sigma_t = 0 $ if $ \omega_{\chi-td} $ is not a dominant level of the critical segment (note that at least $ \sigma_0 \neq 0 $). By Equation \eqref{eq:critical_segment_1} we have
\begin{equation}\label{eq:critical_segment_2}
	\omega_{\text{crit}} = \boldsymbol{x}^{\boldsymbol{q}_{\chi}} z^{\chi} \sum_{t = 0}^{M} \phi^{-t} \sigma_t  \ .
\end{equation}
Let $ \omega^* $ be the $ 1 $-form defined by
$$ \omega^* := \boldsymbol{x}^{\boldsymbol{q}_{\chi}} z^{\chi} \sum_{t = 0}^{M} \phi^{-t} \omega^*_t \ ,  $$	
where we set $ \omega^*_t = 0 $ if $ \omega_{\chi-td} $ is not a dominant level of the critical segment. Note that none of the levels of $ \omega^* $ is dominant.
	
Finally, let $ \breve{\omega} $ be the $ 1 $-form
$$ \breve{\omega} := \omega - \omega_{\text{crit}} - \omega^* \ . $$

XXXXXXXXXXXXXXXXXXXXXXXXXXXXXXXXXXXXXXXXXXXXXXXXXXXX

XXXXXXXXXXXXXXXXXXXXXXXXXXXXXXXXXXXXXXXXXXXXXXXXXXXXXxxxx

After performing a $ (\ell+1) $-Puiseux's package we obtain a parameterized regular local model in which the $ 1 $-form $ \omega $ not need to be $ \gamma $-prepared. By Theorem \ref{th:gammma_preparation} there is a $ \gamma $-preparation, so we will perform it and compare the new critical value and height with the old ones.
	
\subsection{Pre-$ \gamma $-final $ 1 $-forms}
As we see in Section \ref{se:pre_final_functions} in the case of functions, if the critical value is $ \gamma $ or if the critical height is $ 0 $ we know how to obtain a $ \gamma $-final situation. The same happens when we deal with $ 1 $-forms.
\begin{defi}
	A $ \gamma $-prepared $ 1 $-form $ \omega \in N_{\mathcal{A}}^{\ell+1} $ is \emph{pre-$ \gamma $-final} if
	$$ \delta(\omega;\mathcal{A},\gamma) = \gamma $$
	or
	$$ \delta(\omega;\mathcal{A},\gamma) < \gamma \quad \text{and} \quad \chi(\omega;\mathcal{A},\gamma) = 0 \ . $$
\end{defi}
Pre-$ \gamma $-final functions are easily recognizable by its Truncated Newton Polygon as it is represented in Figure \ref{fi:pre_gamma_final_form}

\begin{figure}[!ht]
	\centering
	\includegraphics[width=0.8\textwidth]{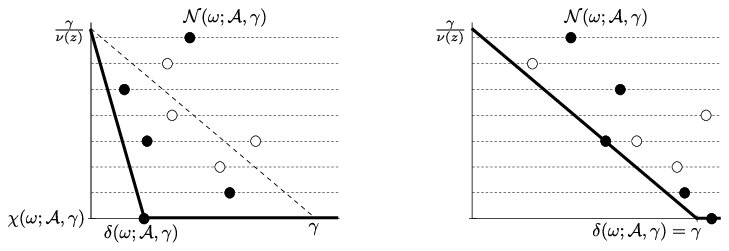}
	\caption{The two pre-$ \gamma $-final situations}
	\label{fi:pre_gamma_final_form}
\end{figure}

Let $ \Psi_{\ell+1} $ be the $ (\ell+1) $-nested transformation given in Lemma \ref{le:push_right}.
\begin{prop}\label{pr:gamma_pre_final}
	Let $ \omega \in N_{\mathcal{A}}^{\ell+1} $ be a pre-$ \gamma $-final $ 1 $-form.
	Consider the $ (\ell+1) $-nested transformation
	\begin{equation*}
		\xymatrix{\mathcal{A} \ar[r]^{\pi} &  \mathcal{A}' \ar[r]^{\Psi_{\ell+1}}  & \mathcal{B}}
	\end{equation*}
	where $ \pi : \mathcal{A} \rightarrow  \mathcal{A}' $ is a $ (\ell+1) $-Puiseux's package. Then $ \omega $ is $ \gamma $-final in $ \mathcal{B} $.
	\begin{proof}
		Consider the decomposition in $ z $-levels of $ \omega $ in $ \mathcal{A} $
		$$ \omega = \sum_{k=0}^{\infty} z^k \omega_k =  \sum_{k=0}^{\infty} z^k \left( \eta_k + f_k \frac{dz}{z} \right) \ . $$
		First, suppose we are in the first case $ \delta(\omega;\mathcal{A},\gamma) = \gamma $. For each index $ k \geq 0 $ we have
		$$ \nu_{\mathcal{A}'}(\omega_k) \geq \nu_{\mathcal{A}}(\omega_k) \geq \gamma - k \nu(z) \ . $$
		From Equations \eqref{eq:variables_after_Puiseux_package} we know that
		$$ z = \boldsymbol{x}'^{\boldsymbol{\alpha}_0} (z' + \xi)^{\beta_0} \ , \quad \text{with} \ \nu(\boldsymbol{x}'^{\boldsymbol{\alpha}_0}) = \nu(z) \ , $$
		hence
		$$ \nu_{\mathcal{A}'}(z^k) = \nu_{\mathcal{A}'} \left(\boldsymbol{x}'^{k\boldsymbol{\alpha}_0} (z' + \xi)^{k\beta_0} \right) = k \nu(z) \ . $$
		Therefore, for each $ k \geq 0 $ we have
		$$ \nu_{\mathcal{A}'}(z^k \omega_k) = \nu_{\mathcal{A}'}(z^k) + \nu_{\mathcal{A}'}(\omega_k) \geq \gamma \ . $$
		It follows that
		$$ \beta(\omega; \mathcal{A}') = \nu_{\mathcal{A}'}(\omega) \geq \gamma \ . $$
		If $ \beta(\omega; \mathcal{A}') > \gamma $ then $ \omega $ is $ \gamma $-final recessive in $  \mathcal{A}' $ so it is also $ \gamma $-final recessive in $ \mathcal{B} $ (Lemma \ref{le:gamma_final_1_forms_stable}). On the other hand, if $ \beta(\omega; \mathcal{A}') = \gamma $, by Lemma \ref{le:push_right} $ \omega $ is $ \gamma $-final in $ \mathcal{B} $.
		
		Now suppose that $ \chi(\omega;\mathcal{A},\gamma) = 0 $. All the levels $ \omega_k $ with $ k > 0 $ do not belong to the critical segment, so we have
		$$ \nu_{\mathcal{A}'}(\omega_k) \geq \nu_{\mathcal{A}}(\omega_k) > \delta(\omega;\mathcal{A},\gamma) - k \nu(z) \ , \quad \forall k \geq 1 \ . $$
		It follows that
		$$ \nu_{\mathcal{A}'}(z^k \omega_k) > \delta(\omega;\mathcal{A},\gamma) \quad \text{for all } k > 0 \ , $$
		hence
		\begin{equation}\label{eq:chi=0_explicit_value_klevels}
			\nu_{\mathcal{A}'}(\omega - \omega_0) = \nu_{\mathcal{A}'}\left(\sum_{k=1}^{\infty} z^k \omega_k \right) > \delta(\omega;\mathcal{A},\gamma) \ .
		\end{equation}
		Thinking in $ \omega_0 $ as a element of $ N_{\mathcal{A}}^{\ell+1} $, it is $ \gamma $-final dominant with explicit value $ \nu_{\mathcal{A}}(\omega_0) = \delta(\omega;\mathcal{A},\gamma) $. By Lemma \ref{le:gamma_final_1_forms_stable} we have that $ \omega_0 $, as a element of $ N_{\mathcal{A}'}^{\ell+1} $, is also $ \gamma $-final dominant with explicit value
		$$ \nu_{\mathcal{A}'}(\omega_0) = \delta(\omega;\mathcal{A},\gamma) \ . $$
		Taking into account Equation \eqref{eq:chi=0_explicit_value_klevels} we have that $ \omega $ is $ \gamma $-final dominant with explicit value
		$$ \nu_{\mathcal{A}'}(\omega) = \nu_{\mathcal{A}'}(\omega_0 + (\omega - \omega_0)) = \delta(\omega;\mathcal{A},\gamma) \ . $$
		Finally, it follows from Lemma \eqref{le:gamma_final_1_forms_stable} that the same happens in $ \mathcal{B} $.
	\end{proof}
\end{prop}

\subsection{Stability of the Critical Height}
In view of Proposition \ref{pr:gamma_pre_final}, in order to complete the proof of $ T_3(\ell+1) $ it is enough with determine a $ (\ell+1) $-nested transformation such that $ \omega $ becomes pre-$ \gamma $-final. In this section we show  that the critical height cannot increase by means of $ (\ell+1) $-nested transformations.
\begin{prop}\label{pr:stability_critical_height_ordered_change}
	Let $ \omega \in N_{\mathcal{A}}^{\ell+1} $ a $ \gamma $-prepared $ 1 $-form. Consider the $ (\ell+1) $-nested transformation
	$$ \xymatrix{ \mathcal{A} \ar[r]^{T} & \tilde{\mathcal{B}} \ar[r]^{\tau}  & \mathcal{B} } $$
	where $ T : \mathcal{A} \rightarrow \tilde{\mathcal{B}} $ is an ordered change of the variable $ z $ and $ \tau : \tilde{\mathcal{B}} \rightarrow \mathcal{B} $ is a $ \gamma $-preparation. Then
	$$ \beta(\mathcal{B}; \omega) = \beta(\mathcal{A}; \omega) \quad \text{and} \quad \delta(\omega;\mathcal{B},\gamma) \geq \delta(\omega;\mathcal{A},\gamma) \ . $$
	In addition, if $ \delta(\omega;\mathcal{B},\gamma) < \gamma $ we have that
	$$ \chi(\omega;\mathcal{B},\gamma) \leq \chi(\omega;\mathcal{A},\gamma) \ . $$
	\begin{proof}
		Consider the decomposition in $ z $-levels of $ \omega $ in $ \mathcal{A} $
		$$ \omega = \sum_{k=0}^{\infty} z^k \omega_k =  \sum_{k=0}^{\infty} z^k \left( \eta_k + f_k \frac{dz}{z} \right) \ . $$
		The ordered change of variables $ T : \mathcal{A} \rightarrow \tilde{\mathcal{B}} $ is given by $ \tilde{z}:= z - \psi $ where $ \psi $ is a polynomial $ \psi \in k[\boldsymbol{x},y_1,y_2,\dots,y_{\ell}] $ such that $ \nu_{\mathcal{A}}(\psi) \geq \nu(z) $. Note that $ \nu_{\mathcal{A}} \equiv \nu_{\tilde{\mathcal{B}}} $. In $ \tilde{\mathcal{B}} $ the decomposition in $ \tilde{z} $-levels is given by
		$$ \omega = \sum_{k=0}^{\infty} z^k \tilde{\omega}_k = \sum_{k=0}^{\infty} z^k \left( \tilde{\eta}_k + \tilde{f}_k \frac{dz}{z} \right) \ , $$
		where
		$$ \tilde{\eta}_k = \eta_k + f_{k+1} d \psi + \sum_{j=1}^{\infty} \binom{k+j}{j} \psi^j \big( \eta_{k+j} + f_{k+1+j} \, d \psi \big) $$
		and
		$$ \tilde{f}_k = f_k + \sum_{j=1}^{\infty} \binom{k-1+j}{j} \psi^j f_{k+j} \ . $$
		First, suppose we have $ \delta(\omega;\mathcal{A},\gamma) = \gamma $. This is equivalent to say that for every $ k \geq 0 $ we have
		$$ \nu_{\mathcal{A}}(\omega_k) \geq \gamma - k\nu(z) \ , $$
		hence
		\begin{equation}\label{eq:explicit_value_levels_1}
			\nu_{\mathcal{A}}(\eta_k) \geq \gamma - k\nu(z) \quad \text{and} \quad \nu_{\mathcal{A}}(f_k) \geq \gamma - k\nu(z) \ .
		\end{equation}
		Recall that $ \nu_{\mathcal{A}}(\psi) \geq \nu(z) $ implies $ \nu_{\mathcal{A}}(d \psi) \geq \nu(z) $, thus in view of \eqref{eq:explicit_value_levels_1} we have
		$$  \nu_{\mathcal{A}}(\tilde{\eta}_k) \geq \gamma - k\nu(z) \quad \text{and} \quad \nu_{\mathcal{A}}(\tilde{f}_k) \geq \gamma - k\nu(z) \ , $$
		so
		$$ \nu_{\mathcal{B}}(\tilde{\omega}_k) = \nu_{\mathcal{A}}(\tilde{\omega}_k) \geq \gamma - k\nu(z) \geq \gamma -k\nu(\tilde{z}) \ , $$
		hence $ \delta(\omega;\mathcal{B},\gamma) = \gamma $.
		
		Now, suppose $ \delta(\omega;\mathcal{A},\gamma) < \gamma $. For short, denote by $ \chi $ the critical height $ \chi(\omega;\mathcal{A},\gamma) $. Since $ \omega $ is $ \gamma $-prepared, for all index $ t \geq 1 $ we have
		$$ \nu_{\mathcal{A}}(\omega_{\chi + t}) > \nu_{\mathcal{A}}(\omega_{\chi}) - t \nu(z) \ , $$
		so
		\begin{equation}\label{eq:explicit_value_levels_2}
			\nu_{\mathcal{A}}(\eta_{\chi + t}) > \nu_{\mathcal{A}}(\omega_{\chi}) - t \nu(z) \quad \text{and} \quad \nu_{\mathcal{A}}(f_{\chi + t}) > \nu_{\mathcal{A}}(\omega_{\chi}) - t \nu(z) \ .
		\end{equation}
		From \eqref{eq:explicit_value_levels_2} we have 
		$$ \nu_{\tilde{\mathcal{B}}}\left(\psi^j \big( \eta_{\chi+t+j} + f_{\chi+t+1+j} \, d \psi \big)\right) > \nu_{\tilde{\mathcal{B}}}(\omega_{\chi}) - t \nu(z) $$
		and
		$$ \nu_{\tilde{\mathcal{B}}}\left(\psi^j f_{\chi+t+j}\right) > \nu_{\tilde{\mathcal{B}}}(\omega_{\chi}) - t \nu(z) $$
		for all $ t\geq1 $ and all $ j\geq1 $. Thus we have
		$$ \nu_{\tilde{\mathcal{B}}}(\tilde{\eta}_{\chi+t}) > \nu_{\tilde{\mathcal{B}}}(\omega_{\chi}) - t \nu(z) $$
		and
		$$ \nu_{\tilde{\mathcal{B}}}(\tilde{f}_{\chi+t}) > \nu_{\tilde{\mathcal{B}}}(\omega_{\chi}) - t \nu(z) $$
		hence
		\begin{equation}\label{eq:explicit_value_new_levels1}
			\nu_{\tilde{\mathcal{B}}}(\tilde{\omega}_{\chi+t}) > \nu_{\tilde{\mathcal{B}}}(\omega_{\chi}) - t \nu(z) \quad \text{for all } t \geq 1 \ .	
		\end{equation}
		In the same way we see that
		\begin{equation}\label{eq:explicit_value_new_levels2}
			\nu_{\tilde{\mathcal{B}}}(\tilde{\omega}_{\chi-t}) \geq \nu_{\tilde{\mathcal{B}}}(\omega_{\chi}) + t \nu(z) \quad \text{for } 1 \leq t \leq \chi \ ,	
		\end{equation}
		and that $ \tilde{\omega}_{\chi} $ is dominant with explicit value
		\begin{equation}\label{eq:explicit_value_new_critical_level}
			\nu_{\tilde{\mathcal{B}}}(\tilde{\omega}_{\chi}) = \nu_{\mathcal{A}}(\omega_{\chi}) \ ,	
		\end{equation}
		After performing the $ \gamma $-preparation $ \tilde{\mathcal{B}} \rightarrow \mathcal{B} $ we still have the properties given in \eqref{eq:explicit_value_new_levels1}, \eqref{eq:explicit_value_new_levels2} and \eqref{eq:explicit_value_new_critical_level} replacing $ \nu_{\tilde{\mathcal{B}}} $ by $ \nu_{\mathcal{B}} $. Let $ z' = \tilde{z} $ be the $ (\ell+1) $-th dependent variable in $ \mathcal{B} $. Since $ \nu(z') \geq \nu(z) $ and taking into account \eqref{eq:explicit_value_new_levels1} and \eqref{eq:explicit_value_new_critical_level} we have
		$$ \delta(\omega;\mathcal{B},\gamma) \geq \nu_{\mathcal{B}}(\tilde{\omega}_{\chi}) + \chi \nu(z') \geq \nu_{\mathcal{A}}(\omega_{\chi}) + \chi \nu(z) = \delta(\omega;\mathcal{A},\gamma) $$
		as desired. In addition, if $ \delta(\omega;\mathcal{B},\gamma) < \gamma $, from \eqref{eq:explicit_value_new_levels2} and $ \nu(z') \geq \nu(z) $ we have
		$$ \chi(\omega;\mathcal{B},\gamma) \leq \chi(\omega;\mathcal{A},\gamma) \ .$$
	\end{proof}
\end{prop}

\begin{prop}\label{pr:stability_critical_height_Puiseux_package}
	Let $ \omega \in N_{\mathcal{A}}^{\ell+1} $ be a $ \gamma $-prepared $ 1 $-form. Suppose that $ \delta(\omega;\mathcal{A},\gamma) < \gamma $. Consider the $ (\ell+1) $-nested transformation
	\begin{equation*}
		\xymatrix{
			\mathcal{A} \ar[r]^{\pi} & \tilde{\mathcal{B}} \ar[r]^{\tau}  & \mathcal{B}
		}
	\end{equation*}
	where $ \pi : \mathcal{A} \rightarrow \tilde{\mathcal{B}} $ is a $ (\ell+1) $-Puiseux's package and $ \tau : \tilde{\mathcal{B}} \rightarrow \mathcal{B} $ is a $ \gamma $-preparation. Then
	$$ \beta(\mathcal{B}; \omega) = \delta(\omega;\mathcal{A},\gamma) \ . $$
	In addition, if $ \delta(\omega;\mathcal{B},\gamma) < \gamma $ we have that
	$$ \chi(\omega;\mathcal{B},\gamma) \leq \chi(\omega;\mathcal{A},\gamma) \ . $$
\end{prop}
We divide the proof of Proposition \ref{pr:stability_critical_height_Puiseux_package} in three parts. First, we prepare the writing of $ \omega $. Second, we study the effect of $ \pi : \mathcal{A} \rightarrow \tilde{\mathcal{B}} $ and finally the effect of $ \tau : \tilde{\mathcal{B}} \rightarrow \mathcal{B} $.

\paragraph{Preparing the writing of $ \omega $.}
Since the level $ \omega_{\chi} $ is dominant with value $ \nu_{\mathcal{A}}(\omega_{\chi}) = \tau_{\chi} = \delta - \chi \nu(z) $ (see Equation \ref{eq:explicit_and_critical_value}), we have that
\begin{equation}\label{eq:critical_level}
	\omega_{\chi} = \boldsymbol{x}^{\boldsymbol{q}_{\chi}} \tilde{\omega}_{\chi} + \bar{\omega}_{\chi} \ , \quad \nu(\boldsymbol{x}^{\boldsymbol{q}_{\chi}}) = \tau_{\chi} \ ,
\end{equation}
where $ \tilde{\omega}_{\chi} $ is log-elementary and $ \nu_{\mathcal{A}}(\bar{\omega}_{\chi}) > \tau_{\chi} $.
	
Let $ \phi \in K $ be the $ (\ell+1) $-contact rational function $ \phi_{\ell+1} = z^{d} / \boldsymbol{x^p} $, where $ d $ is the ramification index $ d(z;\mathcal{A}) $ (see Section \ref{se:Puiseux's_packages}).
	
Now, consider a level $ \omega_k $ which gives a point $ (\beta_k,k) = (\tau_k,k) $ in the critical segment. Since $ \tau_k = \delta - k \nu(z) = \tau_{\chi} + (\chi - k)\nu(z) $, the index $ k $ must be of the form $ k = \chi - td $ for some integer $ 0 \leq t \leq \chi / d $. Since $ \omega $ is $ \gamma $-prepared we know that $ \omega_k $ is dominant, so it can be written as
$$ \omega_{k} = \boldsymbol{x}^{\boldsymbol{q}_{k}} \tilde{\omega}_{k} + \bar{\omega}_{k} \ , \quad \nu(\boldsymbol{x}^{\boldsymbol{q}_{k}}) = \tau_{k} \ . $$
	
Following Remark \ref{re:monomial_after_Puiseux_package}, we have that
$$ \boldsymbol{x}^{\boldsymbol{q}_{\chi - td}} = \boldsymbol{x}^{\boldsymbol{q}_{\chi} + t \boldsymbol{p}} $$
hence
\begin{equation}\label{eq:critical_segment_1}
	z^{\chi-td}\tilde{\omega}_{\chi - td} = \boldsymbol{x}^{\boldsymbol{q}_{\chi}} z^{\chi} \phi^{-t} \tilde{\omega}_{\chi -td} \ .
\end{equation}
Moreover, we can choose the forms $ \tilde{\omega}_{\chi - td} $ with coefficients not depending on the variables $ \boldsymbol{x} $. Write
$$ \tilde{\omega}_{\chi -td} = \sum_{i=1}^r \lambda_{t,i} \frac{dx_i}{x_i} + \mu_t \frac{dz}{z} + \omega^*_t \ , \quad (\boldsymbol{\lambda}_t,\mu_t)\in \mathbb{C}^{r+1} \setminus \{ \boldsymbol{0} \} \ , $$
where $ \omega^*_t $ is not log-elementary. Note that in order to simplify the expressions we have change the subindices. Denote by $ \sigma_t $ the closed form
\begin{equation}\label{eq:initial_part_critical_level}
	\sigma_t := \sum_{i=1}^r \lambda_{t,i} \frac{dx_i}{x_i} + \mu_t \frac{dz}{z} \ .
\end{equation}
Denote by $ M $ the integer part of $ \chi / d $. Let $ \omega_{\text{crit}} $ be the $ 1 $-form defined by
$$ \omega_{\text{crit}} := \boldsymbol{x}^{\boldsymbol{q}_{\chi}} \sum_{t = 0}^{M} z^t \boldsymbol{x}^{t\boldsymbol{p}} \sigma_t \ ,  $$
where we set $ \sigma_t = 0 $ if $ \omega_{\chi-td} $ is not a dominant level of the critical segment (note that at least $ \sigma_0 \neq 0 $). By Equation \eqref{eq:critical_segment_1} we have
\begin{equation}\label{eq:critical_segment_2}
	\omega_{\text{crit}} = \boldsymbol{x}^{\boldsymbol{q}_{\chi}} z^{\chi} \sum_{t = 0}^{M} \phi^{-t} \sigma_t  \ .
\end{equation}
Let $ \omega^* $ be the $ 1 $-form defined by
$$ \omega^* := \boldsymbol{x}^{\boldsymbol{q}_{\chi}} z^{\chi} \sum_{t = 0}^{M} \phi^{-t} \omega^*_t \ ,  $$	
where we set $ \omega^*_t = 0 $ if $ \omega_{\chi-td} $ is not a dominant level of the critical segment. Note that none of the levels of $ \omega^* $ is dominant.
	
Finally, let $ \breve{\omega} $ be the $ 1 $-form
$$ \breve{\omega} := \omega - \omega_{\text{crit}} - \omega^* \ . $$	
	
\paragraph{The effect of $ \pi : \mathcal{A} \rightarrow \tilde{\mathcal{B}} $.}
Let $ (\tilde{\boldsymbol{x}},\tilde{\boldsymbol{y}},\tilde{z}) $ be the coordinates in the parameterized regular local model $ \tilde{\mathcal{B}} $ obtained from $ \mathcal{A} $ by means of a $ (\ell+1) $-Puiseux's package. We have
\begin{equation}\label{eq:critical_segment_3}
	\omega_{\text{crit}} = \tilde{\boldsymbol{x}}^{\boldsymbol{r}} \phi^e \sum_{t = 0}^{M} \phi^{-t} \sigma_t  \ ,
\end{equation}
where $ \nu(\tilde{\boldsymbol{x}}^{\boldsymbol{r}}) = \delta(\omega;\mathcal{A},\gamma) $. The exponents $ \boldsymbol{r} \in \mathbb{Z}_{\geq 0}^r $ and $ e \in \mathbb{Z}_{> 0} $ are determined by the equalities given in \eqref{eq:variables_after_Puiseux_package}. Note that $ \phi = \tilde{z} + \xi $ is a unit in $ R_{\tilde{\mathcal{B}}}^{\ell + 1} $. We can rewrite \eqref{eq:critical_segment_3} as
\begin{equation}\label{eq:critical_segment_4}
	\omega_{\text{crit}} = \tilde{\boldsymbol{x}}^{\boldsymbol{r}} U \sum_{t =0}^{M}  (\tilde{z} + \xi)^{M-t} \sigma_t  \ ,
\end{equation}
where $ U = U(\tilde{z}) = \phi^{e-M} $. For each index $ t $ denote
\begin{equation}\label{eq:critical_coefficients}
	 (\lambda'_{t,1},\dots,\lambda'_{t,r},\mu'_t) = (\lambda_{t,1},\lambda_{t,2},\dots,\lambda_{t,r},\mu_t) \text{\large{$ H $}} \ ,
\end{equation}
where $ H $ is the invertible matrix of non-negative integers corresponding to the $ (\ell+1) $-Puiseux's package (see Equations \eqref{eq:differentials_after_Puiseux_package}). We have
\begin{equation}\label{eq:coefficients_after_Puiseux_package}
	\sigma_t = \sum_{i=1}^r \lambda_{t,i} \frac{dx_i}{x_i} + \mu_t \frac{dz}{z} = \sum_{i=1}^r \lambda'_{t,i} \frac{d\tilde{x}_i}{\tilde{x}_i} + \mu'_t \, \phi^{-1} \tilde{z} \frac{d\tilde{z}}{\tilde{z}} \ .
\end{equation}
Thus we can rewrite \eqref{eq:critical_segment_4} as
\begin{equation}\label{eq:critical_segment_5}
	\omega_{\text{crit}} = \tilde{\boldsymbol{x}}^{\boldsymbol{r}} U \left\{ \sum_{i=1}^{r} P_i \frac{d\tilde{x}_i}{\tilde{x}_i} + \phi^{-1} \tilde{z} \, Q \frac{d\tilde{z}}{\tilde{z}} \right\} \ , 
\end{equation}
where $ P_i , Q \in k[\tilde{z}] $ are given by
\begin{equation}\label{eq:critical_polynomials}
	P_i = \sum_{t =0}^{M} \lambda'_{t,i} (\tilde{z} + \xi)^{M-t} \quad \text{and} \quad Q = \sum_{t =0}^{M} \mu'_t (\tilde{z} + \xi)^{M-t} \ .
\end{equation}
Note that from \eqref{eq:critical_segment_5} it follows that
$$ \nu_{\tilde{\mathcal{B}}}(\omega_{\text{crit}}) = \delta(\omega;\mathcal{A},\gamma) \ . $$
By construction, for each index $ 1\leq i \leq r $ we have
$$ P_i = 0 \Longleftrightarrow \lambda'_{t,i} = 0 \text{ for } t= 0,1,\dots M \ . $$
In the same way we have
$$ Q = 0 \Longleftrightarrow \mu'_t = 0 \text{ for } t= 0,1,\dots M \ . $$
Note that since $ \sigma_0 \neq 0 $ we have $ (P_1,P_2,\dots,P_r,Q) \neq \boldsymbol{0} $. Consider the non-negative integer $ \hbar $ defined by
\begin{equation}\label{eq:critical_height_bound_definition}
	\hbar := \min \left\{ \operatorname{ord}(P_1), \operatorname{ord}(P_2),\dots,\operatorname{ord}(P_r),\operatorname{ord}(Q) + 1 \right\} \ .
\end{equation}
Let us show that $ \hbar \leq \chi(\omega;\mathcal{A},\gamma) $. Suppose that $ (P_1,P_2,\dots,P_r) \neq \boldsymbol{0} $. We have
$$ \min \left\{ \operatorname{ord}(P_1),\operatorname{ord}(P_2),\dots,\operatorname{ord}(P_r) \right\} \leq M = \left[\frac{\chi}{d}\right] \leq \chi \ , $$
hence $ \hbar \leq \chi $. Now, suppose $ (P_1,P_2,\dots,P_r) = \boldsymbol{0} $, so $ Q \neq 0 $. If $ d \geq 2 $ we have
$$ \hbar = \operatorname{ord}(Q) + 1 \leq M+1 = \left[\frac{\chi}{d}\right] +1 \leq \left[\frac{\chi}{2}\right] +1 \leq \chi \ . $$
On the other hand, if $ d=1 $ (thus $ M= \chi $) we have that $ \mu'_{M} = 0 $. Let us explain in detail this last affirmation. By assumption we have $ \lambda'_{M,1} = \dots = \lambda'_{M,r} = 0 $. We also have that $ \mu_M = 0 $. In fact, we have
\begin{equation}\label{eq:d_divides_chi}
	d \text{ divides } \chi \Rightarrow \mu_M = 0
\end{equation}
since $ \sigma_M $ corresponds to the level $ \omega_0 $ (and $ f_0 = 0 $). Looking at \eqref{eq:critical_coefficients}, \eqref{eq:Puiseux_package_matrix_1} and \eqref{eq:Puiseux_package_matrix_2} we obtain
$$ \boldsymbol{0} = (\lambda'_{M,1},\lambda'_{M,2}, \dots , \lambda'_{M,r}) = (\lambda_{M,1},\lambda_{M,2}, \dots , \lambda_{M,r}) \text{\large $\check{H}$} \ . $$
Since $ \check{H} $ is an invertible matrix it follows that $ \lambda_{M,1}= \dots = \lambda_{M,r} =0 $. Looking again at \eqref{eq:critical_coefficients} we conclude that $ \mu'_{M} = 0 $ as desired. In consequence $ \operatorname{ord}(Q) \leq \chi - 1 $ hence $ \hbar \leq \chi $.
	
In view of the expression of $ \omega_{\text{crit}} $ given in \eqref{eq:critical_segment_5} we have that all non-zero levels of $ \omega_{\text{crit}} $ in $ \tilde{\mathcal{B}} $ are dominant with explicit value $ \delta(\omega;\mathcal{A},\gamma) $ and the lowest one is the one located at height $ \hbar $.

The above arguments used to study the properties of $ \omega_{\text{crit}} $ in $ \tilde{\mathcal{B}} $ give us information about $ \omega^* $ and $ \breve{\omega} $. In the same way that we have obtained Equation \eqref{eq:critical_segment_4}, we have	
\begin{equation}
	\omega^* = \tilde{\boldsymbol{x}}^{\boldsymbol{r}} U \sum_{t =0}^{M}  (\tilde{z} + \xi)^{M-t} \omega^*_t  \ ,
\end{equation}
where we recall that $ U = (\tilde{z} + \xi)^{e-M} $ is a unit. Recall also that the coefficients of each $ 1 $-form $ \omega^*_t $ in $ \mathcal{A} $ are series in the dependent variables $ \boldsymbol{y} $, and moreover, the coefficients corresponding to $ \frac{dx_1}{x_1},\frac{dx_2}{x_2}, \dots, \frac{dx_r}{x_r} $ are not units. Therefore, the coefficients of each $ 1 $-form $ \omega^*_t $ in $ \tilde{\mathcal{B}} $ corresponding to $ \frac{d\tilde{x}_1}{\tilde{x}_1},\frac{d\tilde{x}_2}{\tilde{x}_2}, \dots, \frac{d\tilde{x}_r}{\tilde{x}_r} $ cannot be units (since linear combinations of non-units can never be units). We conclude that the levels of $ \omega^* $ in $ \tilde{\mathcal{B}} $ are not dominant and
$$ \nu_{\tilde{\mathcal{B}}}(\omega^*) = \delta(\omega;\mathcal{A},\gamma) \ . $$
In the same way we obtain
$$ \nu_{\tilde{\mathcal{B}}}(\breve{\omega}) > \delta(\omega;\mathcal{A},\gamma) \ . $$
In oder to obtain information about	$ \omega = \omega_{\text{crit}} + \omega^* + \breve{\omega} $ let us summarize some of the properties we have obtained:
\begin{quote}\textbf{Properties of $ \omega_{\text{crit}} $, $ \omega^* $ and $ \breve{\omega} $ in $ \tilde{\mathcal{B}} $}
	\begin{enumerate}
		\item $ \nu_{\tilde{\mathcal{B}}}(\omega_{\text{crit}}) = \delta(\omega;\mathcal{A},\gamma) $;
		\item all the non-zero levels of $ \omega_{\text{crit}} $ in $ \tilde{\mathcal{B}} $ are dominant, and, after factorizing $ \tilde{\boldsymbol{x}}^{\boldsymbol{r}} $, the coefficients of each level are constants;
		\item the lowest non-zero level of $ \omega_{\text{crit}} $ in $ \tilde{\mathcal{B}} $ is the one at height $ \hbar \leq \chi(\omega;\mathcal{A},\gamma) $;
		\item $ \nu_{\tilde{\mathcal{B}}}(\omega^*) = \delta(\omega;\mathcal{A},\gamma) $;
		\item the levels of $ \omega^* $ in $ \tilde{\mathcal{B}} $ are non-dominant;
		\item $ \nu_{\tilde{\mathcal{B}}}(\breve{\omega}) > \delta(\omega;\mathcal{A},\gamma) $.
	\end{enumerate} 
\end{quote}
In view of these properties and the important fact that the sum of a log-elementary $ 1 $-form and a non-log-elementary one is log-elementary we have that the following properties are satisfied:
\begin{quote}\textbf{Properties of $ \omega $ in $ \tilde{\mathcal{B}} $}
	\begin{enumerate}
		\item $ \beta(\omega;\tilde{\mathcal{B}}) = \delta(\omega;\mathcal{A},\gamma) $;
		\item the $ \hbar $-level of $ \omega $ in $ \tilde{\mathcal{B}} $ is $ \beta(\omega;\tilde{\mathcal{B}}) $-final dominant;
		\item for each $ 0 \leq k \leq \hbar - 1 $ the $ k $-level of $ \omega $ in $ \tilde{\mathcal{B}} $ is $ \beta(\omega;\tilde{\mathcal{B}}) $-final recessive.
	\end{enumerate} 
\end{quote}
	
\paragraph{The effect of $ \tau : \tilde{\mathcal{B}} \rightarrow \mathcal{B} $.}
In order to complete the proof we have to analyze the behavior of $ \omega $ under the $ \gamma $-preparation $ \tau : \tilde{\mathcal{B}} \rightarrow \mathcal{B} $. 
	
Bearing in mind the properties of $ \omega $ in $ \tilde{\mathcal{B}} $ listed above and Lemma \ref{le:gamma_final_pairs_stable} we have that the following properties are satisfied:
\begin{quote}\textbf{Properties of $ \omega $ in $ \mathcal{B} $}
	\begin{enumerate}
		\item $ \beta(\omega;\mathcal{B}) = \delta(\omega;\mathcal{A},\gamma) $;
		\item the $ \hbar $-level of $ \omega $ in $ \mathcal{B} $ is $ \beta(\omega;\mathcal{B}) $-final dominant;
		\item for each $ 0 \leq k \leq \hbar - 1 $ the $ k $-level of $ \omega $ in $ \mathcal{B} $ is $ \beta(\omega;\mathcal{B}) $-final recessive.
	\end{enumerate} 
\end{quote}
The first of these properties is exactly the first assertion of Proposition \ref{pr:stability_critical_height_Puiseux_package}. On the other hand, since $ \beta_{\hbar}(\omega;\mathcal{B}) = \delta(\omega;\mathcal{A},\gamma) $, we have that
$$ \delta(\omega;\mathcal{B},\gamma) \leq \delta(\omega;\mathcal{A},\gamma) + \hbar \nu(z') \ , $$
where $ z' = \tilde{z} $ is the $ (\ell + 1) $-th dependent variable in $ \mathcal{B} $. From this last equation we have that if $ \delta(\omega;\mathcal{B},\gamma) < \gamma $ then
$$ \chi(\omega;\mathcal{B},\gamma) \leq \hbar \leq \chi(\omega;\mathcal{A},\gamma) $$
as desired.

\subsection{Resonant conditions}
Proposition \ref{pr:stability_critical_height_Puiseux_package} guarantees that the critical height cannot increase by means of a $ (\ell+1) $-Puiseux's package. Now we give conditions to assure that the critical height drops.

Assume the conditions of Proposition \ref{pr:stability_critical_height_Puiseux_package} and keep the notations used during its proof. In particular recall that we have defined a closed $ 1 $-form
$$ \sigma_t := \sum_{i=1}^r \lambda_{t,i} \frac{dx_i}{x_i} + \mu_t \frac{dz}{z} \ , \quad (\boldsymbol{\lambda},\mu)\in \mathbb{C}^{r+1} \setminus \{ \boldsymbol{0} \} \ , $$
related to each dominant level of the critical segment $ \omega_{\chi -td} $ (see Equation \eqref{eq:initial_part_critical_level}).

Now we establish the \textit{resonant conditions}:
\begin{quote}\textbf{Resonant Condition (R1)}: 
	We say that the condition (R1) is satisfied in $ \mathcal{A} $ if
	$$ \delta(\omega;\mathcal{A},\gamma) < \gamma \ , \quad \chi(\omega;\mathcal{A},\gamma) = 1 \ , \quad d(z;\mathcal{A}) \geq 2 \ , $$
	and the following equivalent conditions are satisfied:
	\begin{enumerate}
		\item The coefficients of $ \sigma_{\mathcal{A},0} $ satisfies
		\begin{equation}\label{eq:non_decreasing_condition_h_d>1}
			\left( \lambda_{0,1}:\lambda_{0,2}:\dots:\lambda_{0,r}:\mu_0 \right)  =  \left( p_1:p_2:\dots:p_r:-d \right) \in \mathbb{P}_{k}^{r} \ ;
		\end{equation}
		\item The $ 1 $-form $ \operatorname{crit}_{\mathcal{A}}(\omega) $ can be written as
		\begin{equation}
			\operatorname{crit}_{\mathcal{A}}(\omega) = \mu_0 \, \boldsymbol{x}^{\boldsymbol{q}_{1}} z \frac{d\phi}{\phi} \ , \quad \mu_0 \in k^* \ .
		\end{equation}
	\end{enumerate}
\end{quote}
\begin{quote}\textbf{Resonant Condition (R2)}: 
	We say that the condition (R2) is satisfied in $ \mathcal{A} $ if
	$$ \delta(\omega;\mathcal{A},\gamma) < \gamma \ , \quad \chi(\omega;\mathcal{A},\gamma) \geq 1 \ , \quad d(z;\mathcal{A}) = 1 \ , $$
	and the following equivalent conditions are satisfied:
	\begin{enumerate}
		\item For each index $ 1 \leq t \leq \chi $ the coefficients of $ \sigma_{\mathcal{A},t} $ are
		\begin{eqnarray}\label{eq:non_decreasing_condition_h_d=1}
			\lambda_{t,i} & = & (-1)^t \xi^t \left[  \binom{\chi}{t} \lambda_{0,i} + p_i \binom{\chi - 1}{t - 1} \mu_0 \right]  \, , \quad t=1,\dots \chi ; \nonumber \\
			\mu_t & = &  (-1)^t \binom{\chi - 1}{t} \xi^t \mu_0 \, , \quad t=1,\dots \chi - 1 \ ;
		\end{eqnarray}
		\item The $ 1 $-form $ \operatorname{crit}_{\mathcal{A}}(\omega) $ can be written as
		\begin{equation}
			\operatorname{crit}_{\mathcal{A}}(\omega) = \boldsymbol{x}^{\boldsymbol{q}_{\chi}} \left( z - \xi \boldsymbol{x}^{\boldsymbol{p}} \right)^{\chi} \left[ \frac{d\boldsymbol{x}^{\boldsymbol{\lambda}_{0}}}{\boldsymbol{x}^{\boldsymbol{\lambda}_{0}}} + \mu_0 \, \frac{d \left( z - \xi \boldsymbol{x}^{\boldsymbol{p}} \right)}{\left( z - \xi \boldsymbol{x}^{\boldsymbol{p}} \right)} \right] \ .
		\end{equation}
	\end{enumerate}	
\end{quote}
This section is devoted to prove the following proposition:
\begin{prop}\label{pr:no_resonant_conditions}
	Let $ \omega \in N_{\mathcal{A}}^{\ell+1} $ a $ \gamma $-prepared $ 1 $-form which is not pre-$ \gamma $-final.
	Consider the $ (\ell+1) $-nested transformation
	\begin{equation*}
		\xymatrix{\mathcal{A} \ar[r]^{\pi} & \tilde{\mathcal{B}} \ar[r]^{\tau}  & \mathcal{B}}
	\end{equation*}
	where $ \pi : \mathcal{A} \rightarrow \tilde{\mathcal{B}} $ is a $ (\ell+1) $-Puiseux's package and $ \tau : \tilde{\mathcal{B}} \rightarrow \mathcal{B} $ is a $ \gamma $-preparation. Suppose that $ \delta(\omega;\mathcal{B},\gamma) < \gamma $. If in addition neither (R1) nor (R2) are satisfied in $ \mathcal{A} $ then
	\begin{equation*}
		\chi(\omega;\mathcal{B},\gamma) < \chi(\omega;\mathcal{A},\gamma) \ .
	\end{equation*}
\end{prop}
In the proof of this proposition we use some calculations made in the proof of Preposition \ref{pr:stability_critical_height_Puiseux_package}. For short, denote by $ \chi $ and $ \chi' $ the critical heights $ \chi(\omega;\mathcal{A},\gamma) $ and $ \chi(\omega;\mathcal{B},\gamma) $ respectively, and denote by $ d $ the ramification index $ d(z;\mathcal{A}) $.
	
The integer number $ \hbar $ (defined in Equation \eqref{eq:critical_height_bound_definition}) is a bound for the new critical height $ \chi' $. It satisfies
\begin{equation}\label{eq:critical_height_bound}
	\chi' \leq \hbar \leq \left[ \frac{\chi}{d} \right] + 1 \ .
\end{equation}
We study separately the cases $ d=1 $ and $ d \geq 2 $.

\paragraph{The case $ d \geq 2 $.}
We have the following inequalities:
\begin{eqnarray*}
	\left[ \frac{\chi}{d} \right] + 1 \leq \frac{\chi}{d} + 1 \leq \frac{\chi}{2} + 1 < \chi \ , & \text{if }  \chi \geq 3 \text{ and } d \geq 2 \ ; \\
	\left[ \frac{\chi}{d} \right] + 1 = \left[ \frac{2}{d} \right] + 1 = 1 \ , & \text{if }   \chi=2 \text{ and } d > 2 \ .
\end{eqnarray*}
Therefore, except in the cases $ \chi = 1 $ or $ \chi = d = 2 $ the above inequalities and \eqref{eq:critical_height_bound} give us $ \chi' < \chi $.
		
Consider the case $ \chi = d = 2 $. We have $ M = [\chi/d] = 1 $. By \eqref{eq:d_divides_chi} we have that $ \mu_1 = 0 $. Therefore
\begin{eqnarray*}
	\left(P_1,P_2,\dots,P_r,Q \right) & = & \left( \tilde{z} + \xi \right) \left( \lambda'_{0,1},\lambda'_{0,2},\dots,\lambda'_{0,r},\mu'_0 \right) + \left( \lambda'_{1,1},\lambda'_{1,2},\dots,\lambda'_{1,r},\mu'_1 \right) \\ 
	& = & \phi \left( \lambda_{0,1},\lambda_{0,2},\dots,\lambda_{0,r},\mu_0 \right) \text{\large $ H $} + \left( \lambda_{1,1},\lambda_{1,2},\dots,\lambda_{1,r},0 \right) \text{\large $ H $} .
\end{eqnarray*}
If some $ P_i \neq 0 $ we have that $ \hbar \leq \operatorname{ord}(P_i) \leq 1 $. Suppose $ P_i = 0 $ for $ i=1,\dots,r $, thus $ Q \neq 0 $. We have that
$$ \left( P_1,P_2,\dots,P_r \right) = \boldsymbol{0} \Rightarrow \left( \lambda'_{1,1},\lambda'_{1,2},\dots,\lambda'_{1,r} \right) = \boldsymbol{0} \ . $$
It follows from $ \mu_1 = 0 $ that
$$ \left( \lambda'_{1,1},\lambda'_{1,2},\dots,\lambda'_{1,r} \right) = \left( \lambda_{1,1}, \lambda_{1,2},\dots,\lambda_{1,r} \right) \text{\large $ \check{H} $} \ . $$
Therefore we have
$$ \left( \lambda'_{1,1},\lambda'_{1,2},\dots,\lambda'_{1,r} \right) = \boldsymbol{0} \Rightarrow \left( \lambda_{1,1},\lambda_{1,2},\dots,\lambda_{1,r} \right) = \boldsymbol{0} \ , $$
since $ \check{H} $ is invertible. Thus we have $ \mu'_1 = 0 $ which implies $ \hbar = \operatorname{ord}(Q) + 1 = 1 $.
		
Now assume $ \chi = 1 $. We have $ M = 0 $ so
\begin{eqnarray*}
	\left(P_1,P_2,\dots,P_r,Q \right) & = & \left( \tilde{z} + \xi \right) \left( \lambda'_{0,1},\lambda'_{0,2},\dots,\lambda'_{0,r},\mu'_0 \right) \\ 
	& = & \phi \left( \lambda_{0,1},\lambda_{0,2},\dots,\lambda_{0,r},\mu_0 \right) \text{\large $ H $} \ .
\end{eqnarray*}
If some $ P_i \neq 0 $ we have that $ \hbar \leq \operatorname{ord}(P_i) \leq 0 $. On the other hand we have
$$ \left(P_1,P_2,\dots,P_r \right) = \boldsymbol{0} \Leftrightarrow \left( \lambda'_{0,1},\lambda'_{0,2},\dots,\lambda'_{0,r} \right) = \boldsymbol{0}  \ . $$
By Equation \eqref{eq:critical_coefficients} and \eqref{eq:Puiseux_package_matrix_2} we have
$$ \left( \lambda'_{0,1},\lambda'_{0,2},\dots,\lambda'_{0,r} \right) = \left( \lambda_{0,1},\lambda_{0,2},\dots,\lambda_{0,r} \right) \text{\large $ \check{H} $} + \mu_0 \boldsymbol{\alpha}_0 \ , $$
hence
$$ \left( \lambda_{0,1},\lambda_{0,2},\dots,\lambda_{0,r} \right) \text{\large $ \check{H} $} + \mu_0 \boldsymbol{\alpha}_0 = \boldsymbol{0} \ . $$		
Since $ \check{H} $ is invertible, following Equation \eqref{eq:dependence_alpha_p_1}, we have that
$$ \left( \lambda_{0,1}, \lambda_{0,2},\dots,\lambda_{0,r}, \mu_0 \right) = -\frac{\mu_0}{d} \left( p_{1},p_2,\dots,p_{r},-d \right) \ , $$
so
$$ \hbar = 1 \Leftrightarrow \text{ condition (R1) is satisfied} \ . $$

\paragraph{The case $ d = 1 $.}
First of all, recall that the matrix $ H $ of the $ (\ell+1) $-Puiseux's package has the form
\begin{equation*}
	\left( \begin{array}{ccc|c}
		&  &  & 0 \\
		& \text{\LARGE $\check{H}$} &  & \vdots \\
		&  &  & 0 \\
		\hline
		\check{p}_1 & \cdots & \check{p}_r & 1
	\end{array}	\right)
\end{equation*}
where $ \check{\boldsymbol{p}} = \boldsymbol{p} \check{H} $ (see the end of Section \ref{se:Puiseux's_packages}). Note also that $ M = \chi $. For all $ 0 \leq t \leq \chi $ denote
$$ \left( \check{\lambda}_{t,1},\check{\lambda}_{t,2},\dots,\check{\lambda}_{t,r} \right) :=  \left( \lambda_{t,1},\lambda_{t,2},\dots,\lambda_{t,r} \right) \check{H} \ . $$
We have that
\begin{eqnarray*}
	\left(P_1,P_2,\dots,P_r,Q \right) & = & \sum_{t=0}^{\chi} \left(\lambda'_{t,1},\lambda'_{t,2},\dots,\lambda'_{t,r},\mu'_t\right) \left(\tilde{z} + \xi\right)^{\chi-t} \\
	 & = & \sum_{t=0}^{\chi} \left(\lambda_{t,1},\lambda_{t,2},\dots,\lambda_{t,r},\mu_t\right) \text{\large $ H $} \left(\tilde{z} + \xi\right)^{\chi-t} \\
	 & = & \sum_{t=0}^{\chi} \left(\check{\lambda}_{t,1} + \check{p}_1 \mu_t ,\check{\lambda}_{t,2} + \check{p}_2 \mu_t ,\dots, \check{\lambda}_{t,r} + \check{p}_r \mu_t , \mu_t\right) \left(\tilde{z} + \xi\right)^{\chi-t} \ .
\end{eqnarray*}
On the other hand we have that $ \hbar = \chi $ if and only if
$$ \operatorname{ord}(P_i) \geq \chi \quad \text{for } i=1,\dots, r $$
and
$$ \operatorname{ord}(Q) \geq \chi - 1 \ . $$
Since $ \mu_{\chi} = 0 $ (see Equation \eqref{eq:d_divides_chi}), it follows that $ \hbar = \chi $ if and only if
$$ P_i = \left( \check{\lambda}_{0,i} + \check{p}_i \mu_0 \right) \tilde{z}^{\chi} \quad \text{for } i=1,\dots,r $$
and
$$ Q = \left(\tilde{z} + \xi\right) \mu_0 \tilde{z}^{\chi-1} \ . $$
These last two equalities are equivalent to condition (R2) so, in the conditions of the proposition we have $ \hbar < \chi $.
\begin{rem}\label{re:critical_height_stable}
	Proposition \ref{pr:no_resonant_conditions} give us necessary conditions for the critical height remains stable. Note that they are not sufficient conditions: in addition, it must happen that
	$$ \delta(\omega;\mathcal{B}) = \beta(\omega;\mathcal{B},\gamma) + \hbar \nu(z') \ ,  $$
	or, equivalently,
	$$ \nu_{\mathcal{B}}(\omega) = \nu_{\mathcal{B}}(\omega_{\text{crit}}) \ . $$
\end{rem}

\subsection{Reductions}
As we did in Section \ref{se:getting_gamma_final_functions} in the case of functions, we will complete the proof of Statement $ T_3(\ell+1) $ by reductio ad absurdum.

Let $ \mathcal{A} $ be a parameterized regular local model for $ K,\nu $, and let $ \omega \in N_{\mathcal{A}}^{\ell+1} $ be a $ 1 $-form such that $ \nu_{\mathcal{A}}( \omega \wedge d \omega ) \geq 2\gamma $. We assume
\begin{enumerate}
	\item $ \omega $ is $ \gamma $-prepared;
	\item for any $ (\ell+1) $-nested transformation $ \mathcal{A} \rightarrow \mathcal{B} $ we have that $ \omega $ is not pre-$ \gamma $-final in $ \mathcal{B} $.
\end{enumerate}
The first assumption is possible thanks to Theorem \ref{th:gammma_preparation}. In this section we will see some implications of the second assumption and finally, in the next section, we will get a contradiction.

As we said, our main control invariant is the critical height $ \chi(\omega;\mathcal{A},\gamma) $. By Proposition \ref{pr:no_resonant_conditions} we know that the critical height can only remain stable under a $ (\ell+1) $-Puiseux's package if one of the resonant conditions is satisfied.
\begin{lem}\label{le:R1_just_once}
	Suppose that condition (R1) is satisfied in $ \mathcal{A} $. Consider a $ (\ell+1) $-nested transformation
	\begin{equation*}
		\xymatrix{
			\mathcal{A} = \mathcal{A}_0 \ar[r]^{\pi_1}  & \tilde{\mathcal{A}}_1 \ar[r]^{\tau_1} & \mathcal{A}_1 \ar[r]^{\pi_2} & \cdots \ar[r]^{\pi_N} & \tilde{\mathcal{A}}_N \ar[r]^{\tau_N} & \mathcal{A}_N = \mathcal{B} }
	\end{equation*}
	where each $ \tau_i : \tilde{\mathcal{A}}_i \rightarrow \mathcal{A}_i  $ is a $ \gamma $-preparation and $ \pi_j : \mathcal{A}_j \rightarrow \tilde{\mathcal{A}}_{j+1} $ is a $ (\ell+1) $-Puiseux's package. If $ \chi(\omega;\mathcal{B},\gamma) = \chi(\omega;\mathcal{A},\gamma) $ then condition (R2) is satisfied in $ \mathcal{B} $.
	\begin{proof}
		Let $ (\boldsymbol{x}_0,\boldsymbol{y}_0,z_0) $ be the coordinates in $ \mathcal{A}_0 $. Since (R1) is satisfied in $ \mathcal{A}_0 $ we have that
		$$
		\operatorname{crit}_{\mathcal{A}_0}(\omega) = \mu \, \boldsymbol{x}_0^{\boldsymbol{q}_0} z_0 \frac{d\phi_0}{\phi_0} \ ,
		$$
		where $ \phi_0 = z_0^d / \boldsymbol{x}_0^{\boldsymbol{p}_0} $ is the $ (l+1) $-th contact rational function in $ \mathcal{A}_0 $. Let $ \xi_0 \in k^* $ be the constant such that $ \nu(\phi_0 - \xi_0) >0 $. After performing the $ (l+1) $-Puiseux's package $ \pi_1 : \mathcal{A}_0 \rightarrow \mathcal{A}_1 $ we obtain
		$$
		\operatorname{crit}_{\mathcal{A}_0}(\omega) = \mu \, \tilde{\boldsymbol{x}}_1^{\boldsymbol{r}} (\tilde{z}_1 + \xi)^u \tilde{z}_1 \frac{d\tilde{z}_1}{\tilde{z}_1} \ ,
		$$
		where $ \tilde{\boldsymbol{x}}_1 $ and $ \tilde{z}_1 $ are the new variables and the exponents $ \boldsymbol{r} $ and $ u $ are determined from Equations \eqref{eq:variables_after_Puiseux_package}, and in particular we have $ \nu(\tilde{\boldsymbol{x}}_1^{\boldsymbol{r}}) = \delta(\omega;\mathcal{A}_0;\gamma) $. By assumption the critical height remains stable after the $ \gamma $-preparation $ \tau_1 : \tilde{\mathcal{A}}_1 \rightarrow \mathcal{A}_1 $, so in $ \mathcal{A}_1 $ we have that
		$$
		\operatorname{crit}_{\mathcal{A}_1}(\omega) = \boldsymbol{x}_1^{\boldsymbol{q}_1} z_1 (\sigma_{\mathcal{A}_1,0} + \phi_1 \sigma_{\mathcal{A}_1,1}) \ ,
		$$
		where $ \boldsymbol{x}_1 $ and $ z_1 = \tilde{z}_1 $ are the new variables, $ \phi_1 $ is the $ (l+1) $-th contact rational function in $ \mathcal{A}_1 $, the exponent $ \boldsymbol{q}_1 $ satisfies $ \nu(\boldsymbol{x}_1^{\boldsymbol{q}_1}) = \delta(\omega;\mathcal{A}_0;\gamma) $ and, moreover, following Remark \ref{re:nested_transformation_matrix} we know that
		$$
		\sigma_{\mathcal{A}_1,0} = \mu \frac{dz_1}{z_1} \ .
		$$
		We see that condition (R1) is not satisfied in $ \mathcal{A}_1 $, so it must be satisfied condition (R2), thus
		$$
		\sigma_{\mathcal{A}_1,1} = - \xi_1 \mu \frac{d \boldsymbol{x}^{\boldsymbol{p}_1}}{\boldsymbol{x}^{\boldsymbol{p}_1}} \ ,
		$$
		where $ z_1 / \boldsymbol{x}^{\boldsymbol{p}_1} = \phi_1 $ and $ \nu(\phi_1 - \xi_1) > 0 $. Equivalently, we have
		$$
		\operatorname{crit}_{\mathcal{A}_1}(\omega) = \boldsymbol{x}_1^{\boldsymbol{q}_1} \mu \, d(z_1 - \xi_1 \boldsymbol{x}^{\boldsymbol{p}_1}) \ .
		$$
		After performing the $ (l+1) $-Puiseux's package $ \pi_1 : \mathcal{A}_1 \rightarrow \tilde{\mathcal{A}}_2 $ we obtain
		$$
		\operatorname{crit}_{\mathcal{A}_1}(\omega) = \mu \, \tilde{\boldsymbol{x}}_2^{\boldsymbol{q}_1} \tilde{z}_2 \left( \frac{d \tilde{\boldsymbol{x}}_2^{\boldsymbol{p}_1}}{\tilde{\boldsymbol{x}}_2^{\boldsymbol{p}_1}} + \frac{d \tilde{z}_2}{\tilde{z}_2}\right) \ ,
		$$
		where $ \tilde{\boldsymbol{x}}_2 = \boldsymbol{x}_1  $ and $ \tilde{z}_2 = \phi_1 - \xi_1 $ are the new variables. Again, since the critical height remains stable, after performing the $ \gamma $-preparation $ \tau_2 : \tilde{\mathcal{A}}_2 \rightarrow \mathcal{A}_2 $ we have
		$$
		\operatorname{crit}_{\mathcal{A}_2}(\omega) = \boldsymbol{x}_2^{\boldsymbol{q}_2} z_2 (\sigma_{\mathcal{A}_2,0} + \phi_2 \sigma_{\mathcal{A},2}) \ ,
		$$
		where $ \boldsymbol{x}_2 $ and $ z_2 $ are the new variables, $ \phi_2 $ is the $ (l+1) $-th contact rational function in $ \mathcal{A}_2 $, the exponent $ \boldsymbol{q}_2 $ satisfies $ \nu(\boldsymbol{x}_2^{\boldsymbol{q}_2}) = \delta(\omega;\mathcal{A}_1;\gamma) $ and, moreover, following Remark \ref{re:nested_transformation_matrix} we know that
		$$
		\sigma_{\mathcal{A}_2,0} = \mu \left( \frac{d \boldsymbol{x}_2^{\boldsymbol{t}_2} }{\boldsymbol{x}_2^{\boldsymbol{t}_2}} + \frac{dz_1}{z_1}\right) \ ,
		$$
		where $ \boldsymbol{t}_2 = C_{\pi_2} \boldsymbol{p}_1 $ being $ C_{\pi_2} $ the invertible matrix of non-negative integers related to $ \pi_2 $ (see Remark \ref{re:nested_transformation_matrix}). Thus we have that $ \boldsymbol{t}_2 $ is a non-zero vector of non-negative integers, so
		$$
		(t_{2,1}:t_{2,2}:\dots:t_{2,r}:1) \neq (p_{2,1}:p_{2,2}:\dots:p_{2,r}:-d(z,\mathcal{A}_2)) \in \mathbb{P}_{k}^{r} \ ,
		$$
		where $ \boldsymbol{p}_2 $ is given by $ \phi_2 = z_2^{d(z,\mathcal{A}_2)} / \boldsymbol{x}_2^{\boldsymbol{p}_2} $, hence condition (R1) is not satisfied in $ \mathcal{A}_2 $ (note that all the integers in the left side term has the same sign while in the right side term there are negative and positive integers).
		
		We have just check that in $ \mathcal{A}_2 $ condition (R2) must be satisfied. If we iterate the calculations made above, we obtain that in $ \mathcal{A}_s $, for $ 2 \leq s \leq N $, the critical part is given by 
		$$
		\operatorname{crit}_{\mathcal{A}_s}(\omega) = \boldsymbol{x}_s^{\boldsymbol{q}_s} z_s (\sigma_{\mathcal{A}_s,0} + \phi_s \sigma_{\mathcal{A}_s,1}) \ ,
		$$
		where $ \boldsymbol{x}_s $ and $ z_s $ are coordinates in $ \mathcal{A}_s $, $ \phi_s $ is the $ (l+1) $-th contact rational function in $ \mathcal{A}_s $, the exponent $ \boldsymbol{q}_s $ satisfies $ \nu(\boldsymbol{x}_s^{\boldsymbol{q}_s}) = \delta(\omega;\mathcal{A}_{s-1};\gamma) $ and 
		$$
		\sigma_{\mathcal{A}_s,0} = \mu \left( \frac{d \boldsymbol{x}_s^{\boldsymbol{t}_s} }{\boldsymbol{x}_s^{\boldsymbol{t}_s}} + \frac{dz_1}{z_1}\right) \ ,
		$$
		where $ \boldsymbol{t}_s = C_{\pi_s} \cdots C_{\pi_2} \boldsymbol{p}_1 $ is a non-zero vector of non-negative integers. Thus we have that
		$$
		(t_{s,1}:t_{s,2}:\dots :t_{s,r}:1) \neq (p_{s,1}:p_{s,2},\dots:p_{s,r}:-d(z,\mathcal{A}_s)) \in \mathbb{P}_{k}^{r} \ ,
		$$
		hence condition (R1) is not satisfied in $ \mathcal{A}_s $, which implies that condition (R2) is satisfied in $ \mathcal{A}_s $ for all $ 1 \leq s \leq N $ as desired.
	\end{proof}
\end{lem}

As we said, our main control invariant is the critical height $ \chi(\omega;\mathcal{A},\gamma) $. Proposition \ref{pr:stability_critical_height_Puiseux_package} allow us to make the following assumption:
\begin{quote}\textbf{Stability of the critical height.}
	Consider a $ (\ell+1) $-nested transformations of the kind
	\begin{equation*}
		\xymatrix{
			\mathcal{A} = \mathcal{A}_0 \ar[r]^{\pi_1}  & \tilde{\mathcal{A}}_1 \ar[r]^{\tau_1} & \mathcal{A}_1 \ar[r]^{\pi_2} & \cdots \ar[r]^{\pi_N} & \tilde{\mathcal{A}}_N \ar[r]^{\tau_N} & \mathcal{A}_N = \mathcal{B} }
	\end{equation*}
	where each $ \tau_i : \tilde{\mathcal{A}}_i \rightarrow \mathcal{A}_i  $ is a $ \gamma $-preparation and $ \pi_j : \mathcal{A}_j \rightarrow \tilde{\mathcal{A}}_{j+1} $ is a $ (\ell+1) $-Puiseux's package. We have
	$$ \chi(\omega;\mathcal{B},\gamma) = \chi(\omega;\mathcal{A},\gamma) \ . $$	
\end{quote}
If there is such a transformation with $ \chi(\omega;\mathcal{B},\gamma) < \chi(\omega;\mathcal{A},\gamma) $ we simply perform it. 

Now, since the critical height $ \chi(\omega;\mathcal{A},\gamma) $ does not drop performing a $ (\ell+1) $-Puiseux's package, we know that condition (R1) or (R2) are satisfied in $ \mathcal{A} $.

In view of Lemma \ref{le:R1_just_once} we can make one more additional assumption
\begin{quote}\textbf{Stability of resonant condition (R2).}
	Consider a $ (\ell+1) $-nested transformations of the kind
	\begin{equation*}
		\xymatrix{
			\mathcal{A} = \mathcal{A}_0 \ar[r]^{\pi_1}  & \tilde{\mathcal{A}}_1 \ar[r]^{\tau_1} & \mathcal{A}_1 \ar[r]^{\pi_2} & \cdots \ar[r]^{\pi_s} & \tilde{\mathcal{A}}_N \ar[r]^{\tau_N} & \mathcal{A}_N = \mathcal{B} }
	\end{equation*}
	where each $ \tau_i : \tilde{\mathcal{A}}_i \rightarrow \mathcal{A}_i  $ is a $ \gamma $-preparation and $ \pi_j : \mathcal{A}_j \rightarrow \tilde{\mathcal{A}}_{j+1} $ is a $ (\ell+1) $-Puiseux's package. We have that condition (R2) is satisfied in $ \mathcal{A}_j $ for every $ j = 1,\dots,N $.		
\end{quote}
One of the features of condition (R1) is that $ d(z;\mathcal{A}) = 1 $. As a consequence we have the following key property:
\begin{quote}\textbf{Stability of the $ z $-coefficient.}
	The coefficient $ f^j \in R_{\mathcal{A}_j}^{\ell+1}$ is the total transform of $ f^0 \in R_{\mathcal{A}_0}^{\ell+1}$.
\end{quote}
Now, we will use Statement $ T_4(\ell+1) $ in order to the $ z $-coefficient of $ \omega $ becomes $ \gamma $-final. Note that a $ \gamma $-preparation for $ \omega $ composed with a $ \ell $-nested transformation is still a $ \gamma $-preparation for $ \omega $. Thus, we can determine a $ (\ell+1) $-nested transformation of the kind
\begin{equation*}
	\xymatrix{
		\mathcal{A} = \mathcal{A}_0 \ar[r]^{\pi_1}  & \tilde{\mathcal{A}}_1 \ar[r]^{\tau_1} & \mathcal{A}_1 \ar[r]^{\pi_2} & \cdots \ar[r]^{\pi_s} & \tilde{\mathcal{A}}_N \ar[r]^{\tau_N} & \mathcal{A}_N = \mathcal{B} }
\end{equation*}
where each $ \tau_i : \tilde{\mathcal{A}}_i \rightarrow \mathcal{A}_i  $ is a $ \gamma $-preparation for both $ \gamma $ and $ f $ and $ \pi_j : \mathcal{A}_j \rightarrow \tilde{\mathcal{A}}_{j+1} $ is either a $ (\ell+1) $-Puiseux's package or an ordered change of the $ (\ell+1) $-th coordinate, such that $ f \in R_{\mathcal{B}}^{\ell+1} $ is $ \gamma $-final. Proposition \ref{pr:stability_critical_height_ordered_change} guarantees that the critical height of $ \omega $ can not increase. If $ \chi(\omega;\mathcal{B},\gamma) < \chi(\omega;\mathcal{A},\gamma) $ we perform it an start again. If it remains stable we get a $ 1 $-form whose $ z $-coefficient is $ \gamma $-final.

Following Remark \ref{re:critical_height_stable}, we can also assume that
$$ \nu_{\mathcal{A}}(\omega) = \nu_{\mathcal{A}}(\omega_{\chi}) \ . $$

Finally, just by performing a $ 0 $-nested transformation following Lemma \ref{le:simple_infinite_list} we can assume that the critical level has the form $ \omega_{\chi} = \boldsymbol{x}^{\boldsymbol{q}} \tilde{\omega}_{\chi} $ where $ \tilde{\omega}_{\chi} $ is log-elementary.

\subsection{End of proof of Theorem \ref{th:gamma_finalization_1_forms}}\label{se:end_proof_theorem_gamma_finalization_1_forms}
In this section we complete the proof of $ T_3(\ell+1) $. In view of the considerations of the previous section we can assume that we have a parameterized regular local model $ \mathcal{A} $, a value $ \gamma \in \Gamma $ and a $ 1 $-form
$$ \omega = \sum_{i=1}^{r}a_i \frac{dx_i}{x_i} + \sum_{j=1}^{\ell}b_j dy_j + zf \frac{dz}{z} \in N_{\mathcal{A}}^{\ell+1} \ , \quad \nu_{\mathcal{A}}(\omega \wedge d\omega) \geq 2\gamma \ , $$
such that
\begin{enumerate}
	\item $ \delta(\omega;\mathcal{A},\gamma) < \gamma$;
	\item $ \chi = \chi(\omega;\mathcal{A},\gamma) > 0 $;
	\item $ \nu_{\mathcal{A}}(\omega) = \nu_{\mathcal{A}}(\omega_{\chi}) $;
	\item The critical level has the form $ \omega_{\chi} = \boldsymbol{x}^{\boldsymbol{q}} \tilde{\omega}_{\chi} $ where $ \tilde{\omega}_{\chi} $ is log-elementary;
	\item The $ z $-coefficient $ f $ is $ \gamma $-final;
	\item Condition (R2) is satisfied;
	\item Properties 1, 2, 3, 4, 5 and 6 are stable for any $ (\ell+1) $-nested transformation $ \mathcal{A} \rightarrow \mathcal{B} $ such that $ \omega $ is $ \gamma $-prepared in $ \mathcal{B} $.
\end{enumerate}
We study separately different cases depending on the explicit value of the function $ f $. In particular we will see that the previous assumptions implies $ \chi = 1 $.

\subsubsection{The case $ \nu_{\mathcal{A}}(f) \geq \nu_{\mathcal{A}}(\omega) + 2 \nu(z)$.}
Recall that
$$ \nu_{\mathcal{A}}(\omega \wedge d\omega) \geq 2\gamma \Longrightarrow \nu_{\mathcal{A}}(\Delta_t) \geq 2\gamma $$
for all $ t\geq0 $. In particular taking $ t = 2 \chi - 1 $ we obtain
\begin{equation}\label{eq:recessive_case_truncation_1}
	\nu_{\mathcal{A}}\left( \sum_{i+j=2 \chi - 1} \big( j \eta_j \wedge \eta_i +f_i d \eta_j + \eta_i \wedge d f_j \big) \right) \geq 2 \gamma \ .
\end{equation}
Since $ \nu_{\mathcal{A}}(f) \geq \nu_{\mathcal{A}}(\omega) + 2 \nu(z)$ we have
\begin{equation}\label{eq:recessive_case_truncation_2}
	\nu_{\mathcal{A}}(f_k) \geq \nu_{\mathcal{A}}(\omega) + 2\nu(z) \ , \quad \text{for all } k \geq 0 \ .
\end{equation}
Since condition (R2) is satisfied we have
\begin{equation}\label{eq:recessive_case_truncation_3}
	\nu_{\mathcal{A}}(\eta_{\chi - k}) \geq \nu_{\mathcal{A}}(\omega) + k\nu(z) \ , \quad \text{for all } k = 0,1,\dots,\chi \ .
\end{equation}
Taking into account \eqref{eq:recessive_case_truncation_2} and \eqref{eq:recessive_case_truncation_3} we derive from \eqref{eq:recessive_case_truncation_1} that
\begin{equation*}
	\nu_{\mathcal{A}}(\eta_{\chi - 1} \wedge \eta_{\chi}) \geq 2 \nu_{\mathcal{A}}(\omega) + 2 \nu(z) \ .
\end{equation*}
Since $ \eta_{\chi} = \boldsymbol{x}^{\boldsymbol{q}} \tilde{\eta}_{\chi} $, after factorizing $ \boldsymbol{x}^{\boldsymbol{q}} $ in the above expression we obtain
\begin{equation}\label{eq:recessive_case_truncation_4}
	\nu_{\mathcal{A}}(\eta_{\chi - 1} \wedge \tilde{\eta}_{\chi}) \geq \nu_{\mathcal{A}}(\omega) + 2 \nu(z) \ .
\end{equation}
By Lemma \ref{le:truncated_proportionality} of truncated proportionality we know there is a function $ g \in R_{\mathcal{A}}^{\ell} $ and a $ 1 $-form $ \bar{\eta} \in N_{\mathcal{A}}^{\ell} $ with $ \nu_{\mathcal{A}}(\bar{\eta}) \geq \nu_{\mathcal{A}}(\omega) + 2 \nu(z) $ such that
\begin{equation}\label{eq:recessive_case_truncation_5}
	\eta_{\chi - 1} = g \, \tilde{\eta}_{\chi} +  \bar{\eta} \ .
\end{equation}
Note that \eqref{eq:recessive_case_truncation_3} implies that $ \nu_{\mathcal{A}}(g) \geq \nu_{\mathcal{A}}(\omega) + \nu(z) $. Let us write $ g $ as a power series
$$ g = \sum_{(I,J)\in \mathbb{Z}^{r+\ell}_{\geq 0}} g_{IJ} \boldsymbol{x}^I \boldsymbol{y}^J \ , \quad f_{IJ} \in k \ . $$
Denote
$$ g = G + H $$
where $ G \in k[\boldsymbol{x},\boldsymbol{y}] \subset R_{\mathcal{A}}^{\ell} $ is the polynomial
$$ G = \sum_{\substack{(I,J)\in \mathbb{Z}^{r+\ell}_{\geq 0} \\ \nu(\boldsymbol{x}^I \boldsymbol{y}^J) \leq \nu_{\mathcal{A}}(\omega) + 2 \nu(z) }} g_{IJ} \boldsymbol{x}^I\boldsymbol{y}^J \  \ . $$
Now we perform the ordered change of coordinates $ \mathcal{A} \rightarrow \tilde{\mathcal{A}} $ given by
$$ \tilde{z} := z - \phi \ , \quad \phi := \frac{-1}{\chi} G \ . $$
As we saw in the proof of Proposition \ref{pr:stability_critical_height_ordered_change} we have that
$$ \eta'_{\chi - 1} = \eta_{\chi - 1} + \chi \phi \eta_{\chi} + \phi^2(\cdots)  $$
where $ \eta'_{\chi-1}\in N_{\tilde{\mathcal{A}}}^{\ell} $ is the $ (\chi-1) $-level of $ \omega $ in $ \tilde{\mathcal{A}} $. We have that
$$ \eta'_{\chi - 1} = g \, \tilde{\eta}_{\chi} +  \bar{\eta} - G \eta_{\chi} + \phi^2(\cdots)  = H \tilde{\eta}_{\chi} +  \bar{\eta} + \phi^2(\cdots) \ . $$
Now, perform a $ \gamma $-preparation $ \tilde{\mathcal{A}} \rightarrow \mathcal{B} $. By definition of $ H $ we have that
$$ \nu_{\mathcal{A}}(H) \geq \nu_{\mathcal{A}}(\omega) + 2 \nu(z) \Longrightarrow  \nu_{\mathcal{A}}(\eta'_{\chi - 1}) \geq \nu_{\mathcal{A}}(\omega) + 2 \nu(z) \ . $$
Since $ \nu_{\mathcal{B}}(\eta_{\chi}) = \nu_{\mathcal{A}}(\eta_{\chi}) = \nu_{\mathcal{A}}(\omega) $ and condition (R2) must be satisfied in $ \mathcal{B} $ we have that
$$ \nu(z') \geq 2 \nu(z) \ . $$
Iterating this procedure we obtain a sequence of parameterized regular local models whose $ (\ell+1) $-th dependent variable has at least twice value that the previous one. The value of the $ (\ell+1) $-th dependent variable can not be greater than
$$ \frac{\gamma - \nu_{\mathcal{A}}(\omega)}{\chi} \ , $$
since this implies that $ \omega $ is pre-$ \gamma $-final in such model. So, after finitely many steps we reach a model in which the value of the $ (\ell+1) $-th dependent variable is greater than
$$ \frac{\nu_{\mathcal{A}}(f) - \nu_{\mathcal{A}}(\omega)}{2} \ . $$

\subsubsection{The case $ \nu_{\mathcal{A}}(\omega) + \nu(z)\leq \nu_{\mathcal{A}}(f) < \nu_{\mathcal{A}}(\omega) + 2 \nu(z)$.}
Since $ f $ is dominant we have that
$$ \nu_{\mathcal{A}}(f) \geq \nu_{\mathcal{A}}(\omega_1) \ . $$
On the other hand, since $ \omega $ is $ \gamma $-prepared we have
$$ \nu_{\mathcal{A}}(\omega_1) = \nu_{\mathcal{A}}(\omega) + (\chi-1)\nu(z) \ . $$
Since by assumption $ \nu_{\mathcal{A}}(f) < \nu_{\mathcal{A}}(\omega) + 2 \nu(z) $, we have that
$$ \chi \leq 2 \ . $$
Repeating the arguments of the previous case we obtain
$$ \nu_{\mathcal{A}}(\eta_{\chi - 1} \wedge \tilde{\eta}_{\chi}) \geq \nu_{\mathcal{A}}(f) \ . $$
Exactly as we did, we can perform an ordered change of coordinates followed by a $ \gamma $-preparation and obtain a parameterized regular local model whose $ (\ell+1) $-th dependent variable $ \tilde{z} $ has value
$$ \nu(\tilde{z}) > \nu_{\mathcal{A}}(f) - \nu_{\mathcal{A}}(\omega) \ . $$

\subsubsection{The case $ \nu_{\mathcal{A}}(\omega) \leq \nu_{\mathcal{A}}(f) < \nu_{\mathcal{A}}(\omega) + \nu(z)$.}
In this case the only possibility is $ \chi = 1 $. Let $ \epsilon > 0 $ be the value given by
$$ \epsilon:= \nu_{\mathcal{A}}(\omega) - \nu_{\mathcal{A}}(f) \ . $$
Again, repeating the above arguments, we can perform an ordered change of coordinates followed by a $ \gamma $-preparation and obtain a parameterized regular local model whose $ (\ell+1) $-th dependent variable $ \tilde{z} $ has value
$$ \nu(\tilde{z}) \geq \nu(z) + \epsilon \ . $$
Iterating, in finitely many steps we
obtain a parameterized regular local model whose $ (\ell+1) $-th dependent variable $ z' $ has value 
$$ \nu(z') \geq \gamma - \nu_{\mathcal{A}}(\omega) \ , $$
which implies that $ \omega $ is pre-$ \gamma $-final in such model in contradiction with our assumptions.

\subsubsection{The case $ \nu_{\mathcal{A}}(\omega) = \nu_{\mathcal{A}}(f) $.}
We have just proved that $ \chi = 1 $ and $ \nu_{\mathcal{A}}(\omega) = \nu_{\mathcal{A}}(f) $ are the only possibilities which are not in contradiction with our assumptions.

Since $ f $ is dominant, we can perform a $ 0 $-nested transformation given by Lemma \ref{le:simple_infinite_list} in order to obtain a parameterized regular local model in which $ f $ is a monomial in the independent variables times a unit. With one more application of Lemma \ref{le:simple_infinite_list} we can obtain a parameterized regular local model $ \mathcal{A}' $ in which $ f $ divides $ \omega $. Denote $ \gamma' = \gamma - \nu_{\mathcal{A}}(f) $. The $ 1 $-form $ \omega' = f^{-1} \omega $ satisfies
$$ \nu_{\mathcal{A}'}(\omega')= 0 \quad \text{and} \quad \nu_{\mathcal{A}'}(\omega' \wedge d \omega) \geq 2 \gamma' \ . $$

So, replacing $ \omega $ by $ \omega' $, $ \mathcal{A} $ by $ \mathcal{A}' $ and $ \gamma $ by $ \gamma' $ we can ``improve'' our list of assumptions:
\begin{enumerate}
	\item $ \delta(\omega;\mathcal{A},\gamma) < \gamma$;
	\item $ \chi = \chi(\omega;\mathcal{A},\gamma) = 1 $;
	\item $ \nu_{\mathcal{A}}(\omega) = \nu_{\mathcal{A}}(\omega_{1}) = 0  $;
	\item The critical level $ \omega_1 $ is log-elementary;
	\item The $ z $-coefficient is $ f = 1 $;
	\item Condition (R2) is satisfied;
	\item Properties 1, 2, 3, 4, 5 and 6 are stable for any $ (\ell+1) $-nested transformation $ \mathcal{A} \rightarrow \mathcal{B} $ such that $ \omega $ is $ \gamma $-prepared in $ \mathcal{B} $.
\end{enumerate}
In this situation, we will show that it is always possible to determine an ordered change of the $ (\ell+1) $-th coordinate such that
$$ \nu(z') \geq 2 \nu(z) \ . $$
This is enough to get de desired contradiction, since iterating this procedure we necessarily reach a parameterized regular local model in which $ \omega $ is pre-$ \gamma $-final.

Since condition (R2) is satisfied we know that the critical part of $ \omega $ can be written as
$$ \omega_{\text{crit}} = \left( z - \xi \boldsymbol{x}^{\boldsymbol{p}} \right) \left[ \frac{d\boldsymbol{x}^{\boldsymbol{\lambda}}}{\boldsymbol{x}^{\boldsymbol{\lambda}}} + \frac{d \left( z - \xi \boldsymbol{x}^{\boldsymbol{p}} \right)}{\left( z - \xi \boldsymbol{x}^{\boldsymbol{p}} \right)} \right] \ , $$
where $ \boldsymbol{\lambda} \in k^r \setminus \{ \boldsymbol{0} \} $, $ \boldsymbol{p} \in \mathbb{Z}^{r}_{\geq 0} \setminus \{ \boldsymbol{0} \} $, $ \xi \in k^* $ and $ \nu(z - \xi \boldsymbol{x}^{\boldsymbol{p}}) > \nu(z) $. This implies that
$$ \eta_1 = \frac{d\boldsymbol{x}^{\boldsymbol{\lambda}}}{\boldsymbol{x}^{\boldsymbol{\lambda}}} + \bar{\eta}_1 \ , $$
where $ \bar{\eta}_1 $ is not log-elementary, and
$$ \eta_0 = - \xi \boldsymbol{x}^{\boldsymbol{p}} \left( \frac{d\boldsymbol{x}^{\boldsymbol{\lambda}}}{\boldsymbol{x}^{\boldsymbol{\lambda}}} + \frac{d\boldsymbol{x}^{\boldsymbol{p}}}{\boldsymbol{x}^{\boldsymbol{p}}}\right) + \bar{\eta}_0 \ , \quad \nu_{\mathcal{A}}(\bar{\eta}_0) \geq \nu(z) \ , $$
where $ \bar{\eta}_0 $ is not $ \nu(z) $-final dominant. Denoting
$$ \sigma := \frac{d\boldsymbol{x}^{\boldsymbol{\lambda}}}{\boldsymbol{x}^{\boldsymbol{\lambda}}} \quad \text{and} \quad \psi_1 := \xi \boldsymbol{x}^{\boldsymbol{p}} $$
we have
$$ \eta_1 = \sigma + \bar{\eta}_1 $$
and
\begin{equation}\label{eq:0_level_1}
	\eta_0 = - \psi_1 \sigma - d \psi_1 + \bar{\eta}_0 \ .
\end{equation}
Let $ \mathcal{A} \rightarrow \mathcal{A}_1 $ be the $ \ell $-nested transformation given by Lemma \ref{le:push_right}. We have
$$ \nu_{\mathcal{A}_1}(\bar{\eta}_1)  > 0 \quad \text{and} \quad \nu_{\mathcal{A}_1}(\bar{\eta}_0) = \epsilon_1 > \nu(z) \ . $$
Consider the ordered change of the $ (\ell+1) $-th coordinate $ \mathcal{A}_1 \rightarrow \tilde{\mathcal{A}}_2 $ given by
$$ \tilde{z}_2 := z + \psi_1 \ . $$
In $ \tilde{\mathcal{A}}_2 $ the critical level is
$$ \eta'_1 = \sigma + \bar{\eta}_1 + \psi_1 (\cdots) $$
and the $ 0 $-level is
$$ \eta'_0 = \bar{\eta}_0 + \psi_1^2 (\cdots) \ . $$
If $ \epsilon_1 \geq 2\nu(z) $ we are done. Indeed, if $ \epsilon_1 \geq 2\nu(z) $ we have
$$ \nu_{\tilde{\mathcal{A}}_2} (\eta'_0) = \nu_{\tilde{\mathcal{A}}_2}(\bar{\eta}_0 + \psi_1^2 (\cdots)) = \nu_{\mathcal{A}_1}(\bar{\eta}_0 + \psi_1^2 (\cdots)) \geq 2 \nu(z) \ , $$
hence necessarily we have $ \nu(\tilde{z}_2) \geq 2 \nu(z) $, since after a $ \gamma $-preparation condition (R2) must be satisfied.

Thus we have $ \nu(z) < \epsilon_1 < 2\nu(z) $.  As we said, after a $ \gamma $-preparation $ \tilde{\mathcal{A}}_2 \rightarrow \mathcal{A}_2 $ condition (R2) must be satisfied, thus in $ \mathcal{A}_2 $ we have
\begin{equation}\label{eq:tschirhausen_1}
	\bar{\eta}_0 = - \psi_2 \, \sigma - d \psi_2 + \bar{\bar{\eta}}_0 \ ,
\end{equation} 
where
$$ \psi_2 = \xi_2 \boldsymbol{x}_2^{\boldsymbol{p}_2} \ , \quad \xi_2 \in k^* \ , \nu(\boldsymbol{x}_2^{\boldsymbol{p}_2}) \geq \epsilon_1 > \nu(z) \ . $$
Since $ \bar{\eta_0} $ is a element of $ N_{\mathcal{A}_1}^{\ell} \subset  N_{\mathcal{A}_2}^{\ell} $, we have that the equality given in \eqref{eq:tschirhausen_1} is also valid in $ \mathcal{A}_1 $, so we have that Equation \eqref{eq:0_level_1} can be rewrite as
\begin{equation}\label{eq:0_level_2}
	\eta_0 = - (\psi_1 + \psi_2) \sigma - d (\psi_1 + \psi_2) + \bar{\bar{\eta}}_0 \ .
\end{equation}
We can iterate this method an obtain functions $ \psi_3,\psi_4,\dots\psi_k \in R_{\mathcal{A}_1}^{\ell}$ with increasing value. Since in $ R_{\mathcal{A}_1}^{\ell} $ the amount of monomials with value lower than $ 2 \nu(z) $ this procedure will provide an ordered change of the $ (\ell+1) $-th coordinate $ \mathcal{A}_1 \rightarrow \mathcal{A}' $ such that $ \nu(z') \geq 2 \nu(z) $.

\section{Proof of the main Theorem}\label{ch:final}
In this chapter we end the proof of Theorem \ref{th:2}. Let us recall the precise statement. Let $ K $ be the function field of an algebraic variety defined over an algebraically closed field $ k $ of characteristic $ 0 $:
\begin{quote}
	\textbf{Theorem $ 1 $}: Let $ \mathcal{F} \subset \Omega_{K/k} $ be a rational codimension one foliation of $ K/k $. Given a rational archimedean valuation $ \nu $ of $ K/k $ and a projective model $ M $ of $ K $ there exists a sequence of blow-ups with codimension two centers $ \pi : \tilde{M} \longrightarrow M $ such that $ \mathcal{F} $ is log-final at the center of $ \nu $ in $ \tilde{M} $.
\end{quote}
As we detail in Chapter \ref{ch:prlm} this result is a consequence of Theorem \ref{th:2}:
\begin{quote}
	\textbf{Theorem $ 2 $}: Let $ \mathcal{F} \subset \Omega_{K/k} $ be a rational codimension one foliation of $ K/k $. Given a rational archimedean valuation $ \nu $ of $ K/k $ and a projective model $ M $ of $ K $ there exists a nested transformation
	$$
		{\mathcal A} \longrightarrow {\mathcal B}
	$$
	such that ${\mathcal F}$ is ${\mathcal B}$-final.  
\end{quote}
Let $ \mathcal{A} = \big( \mathcal{O},(\boldsymbol{x},\boldsymbol{y},z) \big) $ be a parameterized regular local model for $ K,\nu $, where we denote the dependent variables as $ (\boldsymbol{y},z) = (y_1,y_2,\dots,y_{\ell},z) $, being $ \ell = \operatorname{tr.deg}(K/k) - \operatorname{rat.rk}(\nu) - 1$. Take a generator of $ \mathcal{F}_{\mathcal{A}} $
$$ \omega = \sum_{i=1}^r a_i \frac{dx_i}{x_i} + \sum_{j=1}^{s-1} b_j dy_j + f dz \in \mathcal{F}_{\mathcal{A}} \subset \Omega_{\mathcal{O}/k}(\log \boldsymbol{x}) \ . $$
Since
$$ \omega \in \Omega_{\mathcal{O}/k}(\log \boldsymbol{x}) \subset \Omega_{\mathcal{O}/k}(\log \boldsymbol{x}) \otimes_{\mathcal{O}} \hat{\mathcal{O}} =N_{\mathcal{A}}^{\ell+1} \ , $$
we can apply Theorem \ref{th:gamma_finalization_1_forms} to $ \omega $. There are two possibilities:
\begin{enumerate}
	\item For every $ \gamma \in \Gamma $ there is a $ (\ell+1) $-nested transformation
	$$ \mathcal{A} \rightarrow \mathcal{B}_{\gamma} $$
	such that $ \omega $ is $ \gamma $-final recessive in $ \mathcal{B}_{\gamma} $;
	\item There is $ \gamma \in \Gamma $ and a $ (\ell+1) $-nested transformation
	$$ \mathcal{A} \rightarrow \mathcal{B} $$
	such that $ \omega $ is $ \gamma $-final dominant in $ \mathcal{B} $.
\end{enumerate}
Suppose we are in the second case. Since $ \omega $ is dominant in $ \mathcal{B} $, we can perform a $ 0 $-nested transformation $ \mathcal{B} \rightarrow \mathcal{C} $
given by Lemma \ref{le:simple_infinite_list} such that in $ \mathcal{C} $ we have
$$ \omega = Q \tilde{\omega} $$
where $ Q $ is a monomial in the independent variables and $ \tilde{\omega} $ is log-elementary. Let $ \tilde{M} $ be the projective model of $ K $ related to $ \mathcal{C} $ and let $ \tilde{P} \in \tilde{M} $ be the center of $ \nu $. Let $ \tilde{\boldsymbol{x}} $ be the independent variables. Since the nested transformations are algebraic, we have that 
$$ \tilde{\omega} \in \mathcal{F}_{\tilde{M},\tilde{P}}(\log \tilde{\boldsymbol{x}}) \ , $$
thus $ \mathcal{F} $ is log-elementary at $ \tilde{P} $.

Now, suppose we are in the first case. We study separately the cases $ f \neq 0 $ and $ f = 0 $.
\paragraph{The case $ f\neq0 $:} Consider the decomposition of $ \omega $ in $ z $-levels:
$$ \omega = \sum_{k=0}^{\infty} z^k \left( \eta_k + f_k \frac{dz}{z}  \right) \ , \quad \eta_k \in N_{\mathcal{A}}^{\ell} \ , f_k \in R_{\mathcal{A}}^{\ell} \ . $$
By Theorem \ref{th:gamma_finalization_functions} we know that for each index $ k \geq 1 $ there are two possibilities:
\begin{enumerate}
	\item For every $ \gamma \in \Gamma $ there is a $ \ell $-nested transformation
	$$ \mathcal{A} \rightarrow \mathcal{B}_{\gamma} $$
	such that $ f_k $ is $ \gamma $-final recessive in $ \mathcal{B}_{\gamma} $;
	\item There is $ \gamma \in \Gamma $ and a $ \ell $-nested transformation
	$$ \mathcal{A} \rightarrow \mathcal{B} $$
	such that $ f_k $ is $ \gamma $-final dominant in $ \mathcal{B} $.
\end{enumerate}
Suppose that for all $ k \geq 0 $ we are in the first case. In this situation, the  same happens for the function $ f $: fix an index $ \gamma $, perform a $ \ell $-nested transformation such that all the functions $ f_k $ with $ k \leq \gamma / \nu(z) $ are $ \gamma $-final recessive and then perform a $ (\ell+1) $-Puiseux's package. This is not possible, since $ f \in \mathcal{O} \subset R_{\mathcal{A}}^{\ell +1} $ so it has a value $ \nu(f) \in \Gamma $ which keeps stable by means of birational morphisms.

So let $ k_0 $ be the lowest index such that $ f_{k_0} $ can be transformed into a $ \gamma $-final dominant function for some $ \gamma $ by means of a $ \ell $-nested transformation. Now take a value $ \gamma_1 $ such that
$$ \gamma_1 > \gamma + k_0 \nu(z) \  . $$
Since $ \omega \wedge d \omega = 0 $ we have $ \nu_{\mathcal{A}}(\omega \wedge d \omega) > 2 \gamma $, so by Theorem \ref{th:gammma_preparation} we know there is a $ \ell $-nested transformation $ \mathcal{A} \rightarrow \mathcal{A}_1 $ such that $ \omega $ is $ \gamma_1 $-prepared in $ \mathcal{A}_1 $. Since
$$ \nu_{\mathcal{A}_1}(f_{k_0}) \leq \gamma < \gamma_1 - k_0 \nu(z) \ , $$
we have that
$$ \delta(\omega;\mathcal{A}_1,\gamma_1) < \gamma_1 \ , $$
thus the critical height $ \chi(\omega;\mathcal{A}_1,\gamma_1) $ is defined.

Now, perform a $ (\ell+1) $-Puiseux's package $ \mathcal{A}_1 \rightarrow \tilde{\mathcal{A}}_2 $. As we saw in the proof of Proposition \ref{pr:stability_critical_height_Puiseux_package}, there is an integer $ \hbar \leq \chi(\omega;\mathcal{A}_1,\gamma_1) $ such that the $ \hbar $-level of $ \omega $ in $ \tilde{\mathcal{A}}_2 $ is dominant at it has the same explicit value than $ \omega $ (which is exactly $ \delta(\omega;\mathcal{A}_1,\gamma_1) $). Let $ \tilde{z}_2 $ be the $ (\ell+1) $-th dependent variable in $ \tilde{\mathcal{A}}_2 $ and take a value $ \gamma_2 $ such that
$$ \gamma_2 > \delta(\omega;\mathcal{A}_1,\gamma_1) + \hbar \nu(\tilde{z}_2) \  . $$

Again, since $ \omega \wedge d \omega = 0 $ we have $ \nu_{\mathcal{A}}(\omega \wedge d \omega) > 2 \gamma_2 $, so by Theorem \ref{th:gammma_preparation} we know there is a $ \ell $-nested transformation $ \tilde{\mathcal{A}}_2 \rightarrow \mathcal{A}_2 $ such that $ \omega $ is $ \gamma_2 $-prepared in $ \mathcal{A}_2 $. Furthermore we know that
$$ \delta(\omega;\mathcal{A}_2,\gamma_2) < \gamma_2 $$
and
$$ \chi(\omega;\mathcal{A}_2,\gamma_2) \leq \chi(\omega;\mathcal{A}_1,\gamma_1) \ . $$
We can iterate this procedure how many times as we want and we obtain parameterized regular local models $ \mathcal{A}_3, \mathcal{A}_4 , \dots $ such that
$$\chi(\omega;\mathcal{A}_t,\gamma_t) \leq \chi(\omega;\mathcal{A}_{t-1},\gamma_{t-1}) \ . $$
We know that these critical heights are always strictly greater than $ 0 $ since we are assuming that $ \omega $ can not be transformed into a dominant $ 1 $-form.

As we saw in Section \ref{se:end_proof_theorem_gamma_finalization_1_forms} the only possibility in this situation is that, after a finite number of steps, say $ T $, we have $ \chi(\omega;\mathcal{A}_T,\gamma_T) = 1 $ and condition (R2) is satisfied in $ \mathcal{A}_T $. After performing a $ 0 $-nested transformation $ \mathcal{A}_T \rightarrow \mathcal{B}$ given by Lemma \ref{le:simple_infinite_list} if necessary, we have that
$$ \omega = \boldsymbol{x}^{\boldsymbol{q}} U \tilde{\omega} \ , $$
where $ U \in R_{\mathcal{B}}^{\ell+1} $ is a unit and $ \tilde{\omega}\in N_{\mathcal{B}}^{\ell+1} $ is of the form
$$ \tilde{\omega} = \sum_{i=1}^{r} \tilde{a}_i \frac{d\tilde{x}_i}{\tilde{x}_i} + \sum_{j=1}^{\ell} \tilde{b}_j \tilde{y}_j + dz \ ,  $$
where there is at least one index $ 1 \leq i_0 \leq r $ such that
$$ \tilde{a}_{i_0} \equiv z \mod (\tilde{\boldsymbol{x}},\tilde{\boldsymbol{y}},\tilde{z})^2 \ . $$
We have that
$$ U \tilde{\omega} \in \mathcal{F}_{\tilde{M},\tilde{P}}(\log \tilde{\boldsymbol{x}}) $$
and it is a log-canonical $ 1 $-form, so $ \mathcal{F} $ is log-canonical at $ \tilde{P} $.

\paragraph{The case $ f=0 $:} Thanks to the integrability condition, we have that this case corresponds to a foliation of lower dimensional type, it means, the foliation is an analytic cylinder over a foliation defined on a hypersurface. Let us check this assertion. Suppose without lost of generality that $ a_1 \neq 0 $. Fix an index $ 2 \leq i \leq r $. The coefficient of $ \omega \wedge d \omega $ multiplying $ \frac{dx_1}{x_1} \wedge \frac{dx_i}{x_i} \wedge dz $ is
$$ a_i \frac{\partial a_1}{\partial z} - a_1 \frac{\partial a_i}{\partial z} = - a_1^{-2} \frac{\partial a_i/a_1}{\partial z} $$ 
Due to the integrability condition it must be equal to zero, hence
$$ \frac{\partial a_i/a_1}{\partial z} = 0 \ . $$
This is equivalent to say that there is a function $ g_i \in k(\boldsymbol{x},\boldsymbol{y}) $ such that
$$ a_i(\boldsymbol{x},\boldsymbol{y},z) = g_i(\boldsymbol{x},\boldsymbol{y}) a_1(\boldsymbol{x},\boldsymbol{y},z) \ . $$
In the same way, fix an index $ 1 \leq j \leq \ell $. The coefficient of $ \omega \wedge d \omega $ multiplying $ \frac{dx_1}{x_1} \wedge dy_j \wedge dz $ is
$$ b_j \frac{\partial a_1}{\partial z} - a_1 \frac{\partial b_j}{\partial z} = - a_1^{-2} \frac{\partial a_i/a_1}{\partial z} \ . $$ 
Again, it is equivalent to say that there is a function $ h_j \in k(\boldsymbol{x},\boldsymbol{y}) $ such that
$$ b_j(\boldsymbol{x},\boldsymbol{y},z) = h_j(\boldsymbol{x},\boldsymbol{y}) a_1(\boldsymbol{x},\boldsymbol{y},z) \ . $$
Let $ d(\boldsymbol{x},\boldsymbol{y}) \in k[\boldsymbol{x},\boldsymbol{y}] $ be the common denominator of $ g_2,\dots,g_r,h_1,\dots,h_{\ell} $. Denote $ G_i(\boldsymbol{x},\boldsymbol{y}) = g_i(\boldsymbol{x},\boldsymbol{y}) / d(\boldsymbol{x},\boldsymbol{y}) $ and $ H_j(\boldsymbol{x},\boldsymbol{y}) = h_j(\boldsymbol{x},\boldsymbol{y}) / d(\boldsymbol{x},\boldsymbol{y}) $. We have that
$$ \frac{a_1(\boldsymbol{x},\boldsymbol{y},z)}{d(\boldsymbol{x},\boldsymbol{y})} \omega(\boldsymbol{x},\boldsymbol{y},z) =
d(\boldsymbol{x},\boldsymbol{y}) \frac{dx_1}{x_1}  + \sum_{i=2}^{r} G_i(\boldsymbol{x},\boldsymbol{y}) \frac{dx_i}{x_i} + \sum_{j=1}^{\ell} H_j(\boldsymbol{x},\boldsymbol{y}) dy_j \ . $$
This $ 1 $-form belongs to $ N_{\mathcal{A}}^{\ell} $ and it generates the foliation $ \mathcal{F} $.

\begin{thebibliography}{a}



\bibitem{Abh56} \textsc{Abhyankar, S.},
\textit{Local uniformization on algebraic surfaces over ground fields of characteristic $p\not=0$.}
Annals of Mathematics, 1956, p. 491-526.

\bibitem{Abh66} \textsc{Abhyankar, S.},
\textit{Resolution of singularities of embedded algebraic surfaces.}
Pure and Applied Math., Vol. 24, Academic Press, New York, 1966.

\bibitem{AroHiVi} \textsc{Aroca, J. M. \& Hironaka, H. \& Vicente, J. L.},
\textit{Desingularization theorems.}
Volume 30 of Memorias de Matemática del Instituto “Jorge Juan”[Mathematical Memoirs of the Jorge Juan Institute]. Consejo Superior de Investigaciones Científicas, Madrid, 1977.

\bibitem{BierMil} \textsc{Bierstone, E. \& Milman, P. D.},
\textit{A simple constructive proof of canonical resolution of singularities}
Effective methods in algebraic geometry. Birkhäuser Boston, 1991. p. 11-30.

\bibitem{Ca04} \textsc{Cano, F.},
\textit{Reduction of the singularities of codimension one singular foliations in dimension three.}
Annals of mathematics, 2004, p. 907-1011.

\bibitem{Ca87} \textsc{Cano, F.},
\textit{Desingularization strategies for 3-dimensional vector fields.}
Lecture Notes in Math, 1987, vol. 1259.

\bibitem{CaCe} \textsc{Cano, F. \& Cerveau, D.},
\textit{Desingularization of non-dicritical holomorphic foliations and existence of separatrices.}
Acta Mathematica, 1992, vol. 169, no 1, p. 1-103.

\bibitem{CaCeDe} \textsc{Cano, F. \& Cerveau, D. \& Déserti, J.},
\textit{Théorie élémentaire des feuilletages holomorphes singuliers.}
Éditions Belin, 2013.


\bibitem{CaRoSp} \textsc{Cano, F. \& Roche, C. \& Spivakovsky, M.},
\textit{Reduction of singularities of three-dimensional line foliations.}
Revista de la Real Academia de Ciencias Exactas, Fisicas y Naturales. Serie A. Matematicas, 2014, vol. 108, no 1, p. 221-258.

\bibitem{CoPil08} \textsc{Cossart, V. \& Piltant, O.},
\textit{Resolution of singularities of threefolds in positive characteristic. I.: Reduction to local uniformization on Artin–Schreier and purely inseparable coverings.}
Journal of Algebra, 2008, vol. 320, no 3, p. 1051-1082.

\bibitem{CoPil09} \textsc{Cossart, V. \& Piltant, O.},
\textit{Resolution of singularities of threefolds in positive characteristic II.} Journal of Algebra, 2009, vol. 321, no 7, p. 1836-1976.

\bibitem{Cut} \textsc{Cutkosky, S. D.},
\textit{Local monomialization and factorization of morphisms.}
Astérisque, 1999.

\bibitem{Gir} \textsc{Giraud, J.},
\textit{Sur la théorie du contact maximal.}
Mathematische Zeitschrift, 1974, vol. 137, no 4, p. 285-310.

\bibitem{Hi1} \textsc{Hironaka, H.},
\textit{Resolution of singularities of an algebraic variety over a field of characteristic zero: I.}
Annals of Mathematics, 1964, p. 109-203.

\bibitem{Hi2} \textsc{Hironaka, H.},
\textit{Resolution of singularities of an algebraic variety over a field of characteristic zero: II.}
Annals of Mathematics, 1964, p. 205-326.


\bibitem{Ku} \textsc{Kunz, E.},
\textit{Kähler Differentials.}
Advanced Lectures in Mathematics, Springer, 1986.


\bibitem{McQPa} \textsc{McQuillan, M. \& Panazzolo, D.},
\textit{Almost étale resolution of foliations.}
Journal of Differential Geometry, 2013, vol. 95, no 2, p. 279-319.

\bibitem{NoSpi} \textsc{Novacoski, J. \& Spivakovsky, M.},
\textit{Reduction of local uniformization to the rank one case.}
Valuation theory in interaction. EMS Series of Congress Reports, Eur. Math. Soc., Zürich, 2014, p. 404–431.

\bibitem{Pil} \textsc{Piltant, O.},
\textit{An axiomatic version of Zariski’s patching theorem.}
Revista de la Real Academia de Ciencias Exactas, Fisicas y Naturales. Serie A. Matematicas, 2013, vol. 107, no 1, p. 91-121.

\bibitem{Va} \textsc{Vaquié, M.},
\textit{Valuations and local uniformization}
Singularity theory and its applications. Adv. Stud. Pure Math., 43, Math. Soc. Japan, Tokyo, 2006, p. 477-527. 

\bibitem{Sai} \textsc{Saito, K.},
\textit{On a generalization of de Rham lemma.}
Annales de l'institut Fourier, 1976, p. 165-170.

\bibitem{Sei} \textsc{Seidenberg, A.},
\textit{Reduction of singularities of the differential equation $Ady=Bdx$.}
American Journal of Mathematics, 1968, p. 248-269.

\bibitem{Spi1} \textsc{Spivakovsky, M.},
\textit{Sandwiched singularities and desingularization of surfaces by normalized Nash transformations.}
Annals of mathematics, 1990, p. 411-491.

\bibitem{Spi83} \textsc{Spivakovsky, M.},
\textit{A solution to Hironaka’s polyhedra game.}
Arithmetic and geometry. Birkhäuser Boston, 1983. p. 419-432.

\bibitem{Vil} \textsc{Villamayor, O.},
\textit{Constructiveness of Hironaka’s resolution.}
Annales scientifiques de l'École Normale Supérieure, 1989. p. 1-32.

\bibitem{Wal} \textsc{Walker, R.},
\textit{ Reduction of the singularities of an algebraic surface.}
Annals of Mathematics, 1935, p. 336-365.

\bibitem{Zar40} \textsc{Zariski, O.},
\textit{Local uniformization on algebraic varieties.}
Annals of Mathematics, 1940, p. 852-896.

\bibitem{Zar44Ann} \textsc{Zariski, O.},
\textit{Reduction of the singularities of algebraic three dimensional varieties.}
Annals of Mathematics, 1944, p. 472–542.

\bibitem{Zar44Bul} \textsc{Zariski, O.},
\textit{The compactness of the Riemann manifold of an abstract field of algebraic functions.}
Bulletin of the American Mathematical Society, 1944, vol. 50, no 10, p. 683-691.

\bibitem{ZarSa} \textsc{Zariski, O. \& Samuel, P.},
\textit{Commutative Algebra II.}
Van Nostrand, Princeton, 1960.



\end{thebibliography}
\end{document}